%% file: main.tex
\pdfoutput=1   
\documentclass{article}

\input{00_preamble}

\begin{document}

\maketitle

\input{01_abstract}

\input{02_introduction}
\input{03_preliminaries}
\input{04_gd_dynamics}
\input{05_lyapunov_method}
\input{06_conclusion}

\bibliographystyle{plainnat}
\bibliography{bib}

\appendix

\include{apxA_model_catalog}
\include{apxB_gd_submersion}
\include{apxC_scalar_vector}
\include{apxD_LG_plane}
\include{apxE_boundary_inward}
\include{apxF_delta_lt_2}
\include{apxG_terminal_manifold}
\include{apxH_terminal_inheritance}
\include{apxI_local_obstruction}
\include{apxJ_uniqueness_proof}
\include{apxK_matrix_construction}
\include{apxL_matrix_numerics}
\include{apxM_objective_extensions}
\include{apxN_related_work}

\input{99_llm_disclosure}
\end{document}

%% file: 00_preamble.tex
\usepackage{mathtools}
\mathtoolsset{showonlyrefs=true}

\newif\ifarxiv
\arxivtrue

\ifarxiv
  \usepackage{arxiv}
  \usepackage{natbib}
  \renewcommand{\headeright}{A Preprint}
  \renewcommand{\undertitle}{A Preprint}
  \date{}          
\else

\usepackage[preprint]{neurips_2026}
\fi

\usepackage{header}
\usepackage[utf8]{inputenc} 
\usepackage[T1]{fontenc}    
\usepackage{hyperref}       
\usepackage{url}            
\usepackage{booktabs}       
\usepackage{amsfonts}       
\usepackage{nicefrac}       
\usepackage{microtype}      
\usepackage{xcolor}         

\title{State-Dependent Lyapunov Analysis of Rank-1 Matrix Factorization}

\author{%
  Jaehong Moon\\
  Industrial \& Enterprise Systems Engineering\\
  University of Illinois Urbana-Champaign\\
  Urbana, IL 61801 \\
  \texttt{jm133@illinois.edu} \\
}

%% file: 01_abstract.tex
\begin{abstract}
We develop a state-dependent Lyapunov framework for gradient descent on rank-1 matrix factorization. A parameterized quadratic certificate
$I(\delta;\,\cdot)$ generates strictly nested sublevel sets. Their ordering
assigns each point a state $\delta$, while a boundary-inward property
makes this state monotone along gradient-descent trajectories.
Together with internal chain transitivity, this geometry identifies the
limiting dynamics even when all relevant stationary points are unstable. For
scalar and rank-1 factorization, it yields convergence to a global minimizer
below the stability threshold and to a balanced period-$2$ orbit in a
post-critical interval, for almost every initialization in an explicit region.

We formalize the mechanism through structural and dynamical axioms. Within the
origin-centered, exchange-symmetric quadratic class, the axioms determine the
ordered level-set geometry uniquely up to a monotone relabeling of the state,
recovering the scalar geometry identified by Liang--Mont{\'u}far.
Additional analytic and numerical examples suggest broader applicability.
\end{abstract}

%% file: 02_introduction.tex
\section{Introduction}

Low-rank matrix factorization is the nonconvex optimization problem
\begin{align}\label{eq:lrmf_problem}
    \min_{B \in \RR^{m \times r},\, A \in \RR^{n \times r}}
    \risk(A,B)
    := \frac{1}{2}\,\norm{BA^\top - X}_F^2,
\end{align}
where $X \in \RR^{m \times n}$ is a target matrix. Although the best rank-$r$
approximation of $X$ is characterized by \citet{eckart1936approximation}, its
factorized formulation has a nonconvex landscape and a highly symmetric set of
global minimizers. These features make matrix factorization a
natural testbed for identifying geometric structure and implicit regularization
in gradient dynamics.

A classical starting point for matrix factorization is that gradient flow (GF)
preserves $B^\top B-A^\top A$. For certain initializations, this invariant
allows the gradient-flow ODE to be solved explicitly
\citep{tarmoun2021understanding}. The corresponding discrete-time gradient
descent (GD) dynamics have no such exact invariant. Even in the small-step
regime, convergence proofs must carefully control descent, initialization, and
the evolving imbalance between the factors
\citep{du2018algorithmic,ye2021global,jiang2023algorithmic}.

Beyond problem-specific invariants, one of the few broadly applicable tools for
discrete GD is stable-manifold theory. Under suitable regularity conditions, it
shows that the set of initializations converging to an unstable stationary
point has Lebesgue measure zero
\citep{lee2016gradient,panageas2017gradient}. It rules out pointwise convergence
to such a point, but not approach to a positive-dimensional strict-saddle set
without a pointwise limit. In a descent regime, a {\L}ojasiewicz argument may
supply the missing pointwise argument \citep{absil2005convergence}.

Beyond these small-step results, the discovery of the edge-of-stability (EoS)
phenomenon has motivated the analysis of GD at larger step sizes
\citep{cohen2021gradient}. In matrix factorization, \citet{wang2022large} proved
convergence and balancing in several large-step settings. The relation between
bifurcation and EoS dynamics has also been studied in a range of model
problems \citep{song2023trajectory,zhu2022understanding,chen2023beyond}.
Together, these
works show that structured dynamics can persist beyond the small-step regime
and motivate the search for a geometric object that organizes both convergence
and post-critical behavior.

For scalar factorization, \citet{liang2025gradient} used an explicit quantity
$Q$ to characterize an essentially tight convergence region. We instead focus
on the geometry of the corresponding sublevel sets and reinterpret the scalar
problem through a parameterized quadratic function that generates them. We call
this function a \textbf{certificate}. It assigns each point a state $\delta$,
and the GD dynamics monotonically increase this state. This monotonicity is the
key to our convergence analysis.

In rank-one matrix factorization, the non-minimizing stationary set is
positive-dimensional. Stable-manifold theory rules out convergence to an
individual strict saddle, but it does not ensure a pointwise limit for a
trajectory approaching this set. Our shrinking level-set argument supplies the
missing pointwise-convergence step. Together with stable-manifold theory, it
excludes non-minimizing stationary limits for almost every initialization in
the region of interest.

Interestingly, when $\eta$ is large enough that no stationary point in the
region of interest is stable under GD, as in the instability setting of
\citet{ahn2022understanding}, the monotone state drives the trajectory toward a
lower-dimensional set. Its $\omega$-limit set must be internally chain
transitive on this set \citep{hirsch2001chain}. For scalar and rank-one
factorization, classifying these limit sets yields convergence to a period-$2$
cycle when $1<\eta<\sqrt5-1$. The same limit-set argument completes pointwise
convergence to a global minimizer when $\eta<1$.

We axiomatize this state-dependent Lyapunov mechanism. Under suitable
additional assumptions, these axioms uniquely determine the scalar certificate
up to a monotone relabeling of the state. In several related cases, simple
geometric calculations verify that analogous certificates satisfy the same
axioms. Numerical experiments suggest that such certificates may exist for
further examples.
Appendix~\ref{apx:related_work} provides a detailed comparison with the
preceding literature and clarifies the positioning of our results.

%% file: 03_preliminaries.tex
\section{Preliminaries}\label{sec:preliminaries}
Unless stated otherwise, all measure-theoretic language --- ``almost every'',
``almost everywhere'', ``measure zero'', ``null'' --- refers to Lebesgue measure
on the ambient Euclidean space.

\subsection{Orthogonal reduction for gradient descent}\label{subsec:gd_transformation}
The gradients of $\risk$ with respect to $A$ and $B$ are given by
\begin{align}
    \nabla_A \risk (A,B) = (BA^\top - X)^\top B, \quad
    \nabla_B \risk (A,B) = (BA^\top -X)A.
\end{align}
For a step size $\eta > 0$ and initializations
$A_0 \in \RR^{n \times r}, B_0 \in \RR^{m \times r}$, the gradient descent
(GD) dynamics $(A_{t+1}, B_{t+1}) = \gd_\eta(A_t,B_t)$ are defined iteratively as
\begin{align}
    \begin{aligned}
        A_{t+1} = A_t - \eta (B_t A_t^\top - X)^\top B_t, \quad 
        B_{t+1} = B_t - \eta (B_t A_t^\top - X) A_t.
    \end{aligned} \label{eq:gd_dynamics}
\end{align}
Let the singular value decomposition (SVD) of $X$ be $X = L_X \Sigma R_X^\top$, where $L_X \in \OO{m}$, $R_X \in \OO{n}$, and $\Sigma \in \RR^{m \times n}$ is a rectangular diagonal matrix (here, $\OO{n}=\{Q\in \RR^{n\times n}: QQ^\top=Q^\top Q=I_n\}$ denotes the set of $n \times n$ orthogonal matrices). The subscript distinguishes these singular vectors from the residual $L_t$ and the remainder $R_t$ used throughout. For any $Q \in \OO{r}$, the orthogonal invariance of the Frobenius norm ensures that the transformation
\begin{align}
    A \leftarrow R_X^\top A Q, \quad B \leftarrow L_X^\top B Q
\end{align}
yields equivalent GD dynamics. Consequently, we may assume $X = \Sigma$ without loss of generality.

\subsection{Rank-1 matrix factorization/approximation}\label{subsec:rank 1 matrix factorization}
In the rank-1 case, $A\in\RR^{n\times 1}$ and $B\in\RR^{m\times 1}$ are vectors.
Assume $X$ has rank $k>0$ and, after the reduction above, can be written as
$X=\diag(\sigma_1,\dots,\sigma_k,0)\in\RR^{m\times n}$
with $\sigma_1\ge \sigma_2\ge \cdots \ge \sigma_k>0$.
We distinguish two cases. When $k=1$, the target matrix is itself rank one, and
we refer to the problem as \textbf{rank-1 matrix factorization}. When $k>1$,
the rank-one factors can only approximate the higher-rank target, and we refer
to the problem as \textbf{rank-1 matrix approximation}.

Write $A=(a^\top,u^\top)^\top$ and $B=(b^\top,v^\top)^\top$,
where $a,b\in\RR^k$, $u\in\RR^{n-k}$, and $v\in\RR^{m-k}$. 
Note that the null-space components interact with the GD dynamics only through their norms. This allows us to reduce the dynamics to a $(k+1)$-dimensional system:
\begin{align}
    \tilde X = \diag(\sigma_1,\dots,\sigma_k,0)\in\RR^{(k+1)\times (k+1)},\,
    \tilde A = (a^\top,\tilde u)^\top\in\RR^{k+1},\,
    \tilde B = (b^\top,\tilde v)^\top\in\RR^{k+1},
\end{align}
initialized with $\tilde u_0=\|u_0\|$ and $\tilde v_0=\|v_0\|$.
With the convention that the corresponding component remains identically zero when the initial norm is zero, the original trajectories can be recovered via
\begin{align}
    u_t(i)=\frac{u_0(i)}{\tilde u_0}\,\tilde u_t,\quad v_t(i)=\frac{v_0(i)}{\tilde v_0}\,\tilde v_t.
\end{align}

\subsection{Recent results on scalar factorization}\label{subsec:liang_summary}
Recently, \cite{liang2025gradient} carried out a sharp analysis of the scalar factorization problem \begin{align} \min_{a,b \in\mathbb R} \risk_\sca (a,b) = \frac 1 2 (ab- \sigma)^2, \end{align} for some $\sigma > 0$, which corresponds to rank-1 matrix factorization with $n = m = k = 1$. The gradient descent dynamics $(a_{t+1},b_{t+1}) = \gd_\eta^\sca (a_t,b_t)$ are given by \begin{align} \begin{aligned} a_{t+1} = (1-\eta b_t^2) a_t + \eta\sigma b_t, \quad b_{t+1} = (1- \eta a_t^2)b_t + \eta \sigma a_t. \end{aligned}\label{eq:scalar_vector_gd} \end{align} 
By applying the rescaling $\eta \leftarrow \eta \sigma$, $a_t \leftarrow a_t/\sqrt{\sigma}$, and $b_t \leftarrow b_t/\sqrt{\sigma}$, we may assume $\sigma = 1$ without loss of generality.

A key ingredient in the convergence analysis of \cite{liang2025gradient} is the Lyapunov-like function
\begin{align}
    Q(\sigma; a,b) = a^2 + b^2 + \sqrt{(a^2 + b^2)^2 - 16 \sigma(ab-\sigma)}.
\end{align}
\begin{theorem}[\citet{liang2025gradient}]\label{thm:convergence_region}
Let $(a_0,b_0)\in \RR^2$ be an initialization.

(1) Given $\eta \in (0,1/\sigma)$, if $Q(\sigma;a_0,b_0) < 8/\eta$, the GD dynamics $(a_t,b_t)_{t\ge0}$ of $\gd_\eta^\sca$ converge to a global minimizer for almost every initialization. Moreover, $Q(\sigma; a_t,b_t)$ monotonically decreases along the dynamics.

(2) Given $\eta \in (0,1/\sigma)$, if $Q(\sigma;a_0,b_0) > 8/\eta$, the GD dynamics $(a_t,b_t)_{t\ge0}$ of $\gd_\eta^\sca$ diverge for almost every initialization.

(3) Given $\eta \in (1/\sigma, \infty)$, the GD dynamics $(a_t,b_t)_{t\ge0}$ of $\gd_\eta^\sca$ do not converge to any global minimizer for almost every initialization.
\end{theorem}

%% file: 04_gd_dynamics.tex
%

\section{Gradient Descent Dynamics via State-Dependent Level Sets}\label{sec:rank-1 convergence}

We treat three models: scalar factorization, rank-1 matrix factorization, and
isotropic-signal rank-1 approximation of $X=\diag(I_{n-1},0)$. Each model first
gets a certificate and the state it induces; the argument then runs along the
same six steps in all three models, so we develop it in full only for the scalar
model and then record what changes at each step for the two matrix models.
Steps 1--5 correspond one for one to the five steps of the abstract convergence
proof in Appendix~\ref{apdx:abstract_convergence}, the boundary-inward
propositions behind Step 1 being proved in
Appendix~\ref{sec:boundary_inward_proofs}; Step 6 is settled in
Appendices~\ref{sec:convergence when delta_* = 2}
and~\ref{sec:terminal_inheritance}.

Table~\ref{tab:three_models} records the certificate and the stationary
geometry of each. The reduction to signal and off-signal coordinates is
Subsection~\ref{subsec:rank 1 matrix factorization}, and the gradient maps in
those coordinates together with the derivations of the minimizer and stationary
sets are collected in Appendix~\ref{subsec:model_catalog_core}; we use both
below without restating them.

All three models share one regularity input, based on a theorem
of \cite{ponomarev1987submersions} on preimages of sets of measure zero under
submersions.

\begin{corollary}
\label{cor:gd_preimage_null}
Let $g:\RR^d\to\RR^d$ be $C^1$ and a submersion almost everywhere. If
$E\subseteq\RR^d$ has measure zero, then
$\bigcup_{T=0}^\infty g^{-T}(E)$ has measure zero.
\end{corollary}

Appendix~\ref{apdx:gd_submersion} proves this and verifies the
almost-everywhere submersion property for the gradient descent maps considered
here. Consequently, for almost every initialization, no finite iterate reaches
the relevant stationary or terminal exceptional sets, and we use this below
without repeating the argument in each subsection.

Throughout the section we take the top singular value of the target to be
$\sigma=1$; by the rescaling of Subsection~\ref{subsec:liang_summary} this is
without loss of generality. The three results below undo the rescaling and are
stated for general $\sigma$.

\begin{table}[t]
\centering
\scriptsize
\setlength{\tabcolsep}{5pt}
\caption{The three models, normalized to $\sigma=1$. The gradient maps and the
derivations are collected in Appendix~\ref{subsec:model_catalog_core}. At each
nonterminal state, $\mathcal S\cap\partial K_\delta$ is finite in the two
factorization models and positive-dimensional in the approximation model.}
\label{tab:three_models}
\begin{tabular}{@{}lllll@{}}
\toprule
Model & Ambient & Certificate & $\mathcal M$ & $\mathcal S\setminus\mathcal M$\\
\midrule
Scalar factorization & $\RR^2$ & $\isc$ & $\{ab=1\}$ & $\{(0,0)\}$\\
Rank-1 factorization & $\RR^4$ & $\ifac$ & $\{ab=1,\ u=v=0\}$ & $\{a=b=0,\ uv=0\}$\\
Isotropic-signal approximation & $\RR^{2n}$ & $\iapx$ & $\{a\parallel b,\ \langle a,b\rangle=1,\ u=v=0\}$ & $\{a=b=0,\ uv=0\}$\\
\bottomrule
\end{tabular}
\end{table}

\subsection{Scalar factorization}\label{subsec:scalar_factorization}

We revisit scalar factorization and express the convergence region of
Theorem~\ref{thm:convergence_region} using a state-dependent certificate.

Define the parameterized certificate
\begin{align}
    \isc(\delta;\,a,b) := \delta(a^2 + b^2)  - \delta^2ab + \delta^2 - 4.
    \label{eq:I_delta}
\end{align}
For each $\delta\in(0,2)$ the sublevel set $\{\isc(\delta;\cdot)\le 0\}$ is a
compact ellipse, and these ellipses shrink strictly as $\delta$ increases.
We define their intersection as the \textbf{terminal set}
\begin{align}
    K_2^\sca
    :=
    \bigcap_{0<\delta<2}\{(a,b):\isc(\delta;\,a,b)\le 0\}
    =
    \{(a,a):a^2\le 2\},
    \label{eq:K2_scalar}
\end{align}
computed in Appendix~\ref{appx:K_2_description}. It is the limiting sublevel
set, not the zero set at the endpoint: $\isc(2;\,a,b)=2(a-b)^2$ vanishes on the
entire balanced line, but $K_2^\sca$ is a compact segment.

Every point of $\RR^2$ then carries a \textbf{state} $\delta$. Each $(a,b)\notin K_2^\sca$ admits a unique
$\delta(a,b)\in(0,2)$ with $\isc(\delta(a,b);\,a,b)=0$, and on $K_2^\sca$ the
state is $2$ (see Lemma~\ref{lem:concrete_axioms}(1) for the details).
Intuitively, the root exists because the certificate changes sign
across the state range, $\isc(0;\,a,b)=-4<0$ while $\isc(2;\,a,b)\ge0$, and it
is unique because the ellipses are strictly nested.
Writing $\delta_t$ for the state of $(a_t,b_t)$ along the GD trajectory, a
direct computation gives $\delta_t=8/Q(1;a_t,b_t)$, so $\{\isc(\eta;\cdot)<0\}$
is exactly the region $\{Q<8/\eta\}$ of
Theorem~\ref{thm:convergence_region}.

\textbf{Step 1. Boundary-inward property.}
Writing $L_t := 1 - a_t b_t$ for the residual, an algebraic computation yields
the quotient--remainder identity
\begin{align}\label{eq:certificate_for_scalar}
    \isc(\delta;\, a_{t+1}, b_{t+1})
    &= M_t^\sca(\delta)\, \isc(\delta;\, a_t, b_t)
       + R_t^\sca(\delta),
\end{align}
where
\begin{align}
M_t^\sca(\delta) := 1 - \eta\delta L_t + \eta^2 L_t^2, \quad
R_t^\sca(\delta) := \eta(\delta - \eta)(\delta^2 - 4)\, L_t^2.
\label{eq:sc_M_t_R_t_def}
\end{align}
The first term vanishes on the level set, so at $\delta=\delta_t$ the value of
the certificate after one step is $R_t^\sca(\delta_t)$, and for
$0<\eta<\delta_t<2$ each of its factors has a known sign:
$\eta(\delta_t-\eta)>0$, $\delta_t^2-4<0$, and $L_t^2\ge0$, with $L_t=0$ exactly
at a global minimizer.
\emph{Every nonstationary boundary point is therefore mapped strictly into the
current sublevel set, with equality exactly at stationary points.}
By nesting the state increases strictly, so
$\eta<\delta_0<\delta_1<\cdots<2$ with $\isc(\delta_t;\,a_t,b_t)=0$.

Since $\delta_t$ is increasing and bounded, it converges to some
$\delta_\ast\le2$, and the two cases $\delta_\ast<2$ and $\delta_\ast=2$ are the
next two steps.

\textbf{Step 2. The nonterminal case.}
If $\delta_\ast<2$, the remainder $R_t^\sca(\delta_t)$ left by one step is
controlled by the state increment,
\begin{align}
    0<-R_t^\sca(\delta_t)\le M(\delta_{t+1}-\delta_t),
\end{align}
and hence vanishes, forcing every accumulation point onto the stationary set of
the limiting level.

\textbf{Step 3. The terminal case.}
If instead $\delta_\ast=2$, nesting gives $\dist{(a_t,b_t)}{K_2^\sca}\to0$; the
trajectory approaches the terminal set and nothing more is claimed at this
stage. See Steps~2 and~3 of the proof in
Appendix~\ref{sec:extension_state_dependent}.

\textbf{Step 4. Pointwise upgrade.}
On every nonterminal level $\delta$ the stationary points lying on
$\partial K_\delta$ are finitely many. Together with vanishing increments, this
upgrades convergence to that finite set into convergence to a single point; see Appendix~\ref{sec:extension_state_dependent}.

\textbf{Step 5. Exclusion of unstable limits.}
Here the step size splits the analysis. At a global minimizer the
Jacobian of the gradient map has the eigenvalues $1$ and
$1-\eta(a_\ast^2+b_\ast^2)$, so the minimizers with no expanding direction are
exactly those with $a_\ast^2+b_\ast^2\le2/\eta$. Since $a_\ast b_\ast=1$ forces
$a_\ast^2+b_\ast^2\ge2$, that set is nonempty precisely when $\eta\le1$. For
$\eta<1$ some minimizers can therefore still attract, and nothing has to be
excluded: the stationary points left by Step 4 are all global minimizers. For
$\eta>1$ the set is empty, so every minimizer has an expanding direction and
its local nonescaping set is null. A compact-covering plus preimage-null
argument excludes convergence to any of them almost surely, leaving
$\delta_\ast=2$ for almost every certified initialization. See
Appendix~\ref{subsec:minimizer_instability_eta_gt_one}.

\textbf{Step 6. The terminal dynamics.}
On $K_2^\sca$ a point is determined by one coordinate, and the gradient step
acts on the product $s=ab$ through the
one-dimensional map $s\mapsto s(1+\eta-\eta s)^2$.
Appendix~\ref{sec:convergence when delta_* = 2} classifies its orbits: for
$0<\eta<1$, every nonzero terminal orbit converges to the fixed point $s=1$,
while $s=0$ is the repelling fixed point corresponding to the origin; for
$\eta\in(1,\sqrt5-1)$, an attracting period-$2$ orbit appears.

That the full dynamics follow this classification is not automatic: a
trajectory with $\delta_\ast=2$ only approaches $K_2^\sca$, so it need not obey
the reduced recursion at any finite time. The bridge is the exact semiconjugacy
on $K_2^\sca$: taking the product coordinate after one gradient step gives the
same result as first taking $s=ab$ and then applying the reduced map. The
trajectory's $\omega$-limit set lies in $K_2^\sca$ and is internally chain
transitive \citep{hirsch2001chain}. The product coordinate therefore maps it to
an internally chain transitive set of the reduced map, and such sets
can be classified.
Appendix~\ref{sec:terminal_inheritance} does this, and returns both the
$\eta<1$ upgrade of terminal-set convergence to convergence to a minimizer
(Corollary~\ref{cor:subcritical_terminal_inheritance}) and assertion~(2) of
Proposition~\ref{prop:scalar_convergence}.

Below the threshold $\eta=1$ the certificate recovers what is already known: on
the region identified above, the conclusion is the same as
Theorem~\ref{thm:convergence_region}(1). We call
$\eta\in(0,1)$ the \textbf{pre-critical convergence regime}, since the
trajectory converges to a global minimizer for almost every certified
initialization. For
$\eta\in(1,2)$, Theorem~\ref{thm:convergence_region}(3) already rules out
convergence to a global minimizer almost surely. Our additional conclusion
describes what happens inside the certified region: the state nevertheless
reaches its terminal value. We call this the \textbf{post-critical terminal
regime}.

\begin{proposition}[Post-critical scalar dynamics]\label{prop:scalar_convergence}
Let $(a_0,b_0)\in\RR^2$ be an initialization of the GD dynamics
$(a_t,b_t)_{t\ge0}$ of $\gd_\eta^\sca$.

(1) If $\eta\in(1/\sigma,2/\sigma)$ and the rescaled initialization satisfies
$\isc(\eta\sigma;\,\sigma^{-1/2}(a_0,b_0))<0$, then
for almost every such initialization $\delta_\ast=2$ and
$\dist{(a_t,b_t)}{\sqrt\sigma\,K_2^\sca}\to 0$, where
$\sqrt\sigma\,K_2^\sca=\{(a,a)\in\RR^2:a^2\le2\sigma\}$ is the balanced
terminal set.

(2) Moreover, if $\eta\sigma \in (1, \sqrt5-1)$, then for almost every such
initialization the $\omega$-limit set of $(a_t,b_t)_{t\ge0}$ is one of the two
signed balanced period-$2$ orbits
$\pm\bigl\{(\sqrt{\sigma s_-},\sqrt{\sigma s_-}),
(\sqrt{\sigma s_+},\sqrt{\sigma s_+})\bigr\}$, where
\begin{align}
    s_\pm
    :=
    \frac{1+\eta\sigma\pm\sqrt{(\eta\sigma-1)(\eta\sigma+3)}}{2\eta\sigma}
    \in(\tfrac12,2).
    \label{eq:s_pm_main}
\end{align}
In particular the even and odd subsequences converge to its two distinct
points.
\end{proposition}

\begin{figure}[t]
    \centering
    \includegraphics[width=\linewidth]{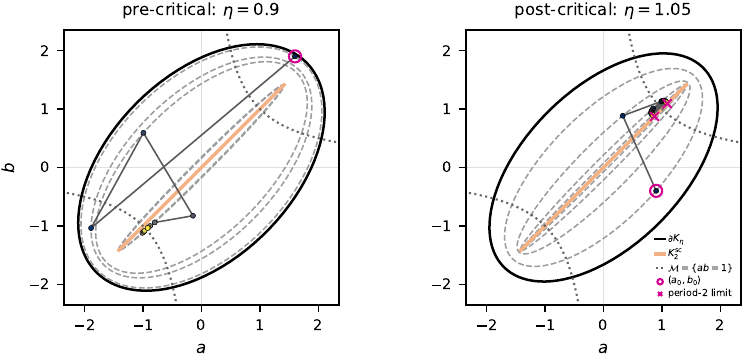}
    \caption{Scalar GD in the original factor coordinates; dashed ellipses show
    selected state boundaries along each trajectory. The pre-critical
    trajectory (left) approaches a nonbalanced minimizer, whereas the
    post-critical trajectory (right) approaches $K_2^\sca$ and its two
    subsequences approach the period-$2$ orbit. The respective parameters
    $(T,\eta,a_0,b_0)$ are $(10,0.9,1.6,1.9)$ and
    $(25,1.05,0.9,-0.4)$.}
    \label{fig:scalar_dynamics_ab}
\end{figure}

\subsection{Rank-1 matrix factorization}
\label{subsec:rank1_gd_general}

For $k=1$ in the setup of
Section~\ref{subsec:rank 1 matrix factorization}, the orthogonal reduction and
rescaling give a state $(a_t,b_t,u_t,v_t)\in\RR^4$, with $(u_t,v_t)$
off-signal. The certificate becomes
\begin{align}
\ifac(\delta;\, a,b,u,v)
:= \isc(\delta;\,a,b) + \delta(u^2+v^2),
\label{eq:I_delta_def}
\end{align}
and $\ifac(2;\cdot)=2(a-b)^2+2(u^2+v^2)>0$ almost everywhere. The endpoint and
nesting argument from the scalar case therefore gives every certified
trajectory a unique state $\delta_t\in(\eta,2)$.

As in the scalar case, the certificate admits a quotient--remainder
decomposition
\begin{align}
\ifac(\delta;\, a_{t+1},b_{t+1},u_{t+1},v_{t+1})
= M_t^\fac(\delta)\,\ifac(\delta;\, a_t,b_t,u_t,v_t) + R_t^\fac(\delta),
\label{eq:I_delta_affine}
\end{align}
so on the level set only the remainder remains. Here
\begin{align}
    R_t^\fac(\delta)
    =
    -(\eta\delta)^2u_t^2v_t^2
    -\eta(\delta-\eta)
    \Bigl[(4-\delta^2)L_t^2+4D_t-\delta N_t\Bigr],
\end{align}
with $L_t=1-a_tb_t$, $N_t=u_t^2+v_t^2$, and
$D_t=a_t^2v_t^2+b_t^2u_t^2+u_t^2v_t^2$. Because of $-\delta N_t$, the bracket
is not termwise nonnegative. The boundary constraint nevertheless controls
$N_t$: Appendix~\ref{sec:boundary_inward_proofs} proves that
$R_t^\fac(\delta)\le0$ on the level set, with equality only at stationary
points.

\begin{proposition}\label{prop:boundary_inward_rank1}
Let $0 < \eta < \delta < 2$.
If $\ifac(\delta;\, a_t,b_t,u_t,v_t) = 0$ and $(a_t,b_t,u_t,v_t) \notin \mathcal{S}$, then
\begin{align}
\ifac(\delta;\, a_{t+1},b_{t+1},u_{t+1},v_{t+1}) = R_t^\fac(\delta) < 0,
\end{align}
and hence the next iterate lies strictly inside the sublevel set $\{\ifac(\delta) \le 0\}$.
\end{proposition}

Beyond the boundary-inward argument of Step~1, only Steps~4 and~5 differ from
the scalar case. Step~4 remains valid because every nonterminal level contains
at most eight stationary points, despite the positive-dimensional branch
$\mathcal S\setminus\mathcal M$. In Step~5, all possible non-minimizing limits
lie in a compact set on which every point has an expanding Jacobian direction.
They are therefore excluded almost surely, so for $\eta<1$ every remaining
nonterminal limit is a global minimizer. For $\eta>1$ the same exclusion
applies to the possible minimizer limits, forcing $\delta_\ast=2$ almost
surely; see
Appendices~\ref{subsec:exclude_nonmin_stationary_rank1}
and~\ref{subsec:minimizer_instability_eta_gt_one}.

\begin{theorem}\label{thm:convergence_region_rank_1}
Let $(a_0, b_0, u_0, v_0) \in \RR^4$ be an initialization of the GD
dynamics $(a_t,b_t,u_t,v_t)_{t\ge0}$ of $\gd_\eta^\fac$.

(1) If $\eta \in (0, 1/\sigma)$ and the rescaled initialization satisfies
$\ifac(\eta\sigma;\,\sigma^{-1/2}(a_0,b_0,u_0,v_0)) < 0$, then for almost
every such initialization the dynamics converge to a global minimizer.

(2) If $\eta \in (1/\sigma,2/\sigma)$ and the rescaled initialization satisfies
$\ifac(\eta\sigma;\,\sigma^{-1/2}(a_0,b_0,u_0,v_0)) < 0$, then for almost
every such initialization the dynamics fail to converge to a global minimizer and instead
satisfy $\dist{(a_t,b_t,u_t,v_t)}{\sqrt\sigma\,K_2^\fac}\to 0$, where
$\sqrt\sigma\,K_2^\fac=\{(a,a,0,0)\in\RR^4 : a^2\le 2\sigma\}$ is the balanced
terminal set.

(3) Moreover, if $\eta\sigma \in (1,\sqrt5-1)$, then for almost every such
initialization the $\omega$-limit set of $(a_t,b_t,u_t,v_t)_{t\ge0}$ is one
of the two signed balanced period-$2$ orbits
\begin{align}
    \pm\bigl\{
        (\sqrt{\sigma s_-},\sqrt{\sigma s_-},0,0),\ \,
        (\sqrt{\sigma s_+},\sqrt{\sigma s_+},0,0)
    \bigr\},
\end{align}
with $s_\pm$ as in Eq.~\eqref{eq:s_pm_main}.
\end{theorem}

\begin{remark}[The convergence region is not sharp]
\label{rmk:convergence_region_not_sharp}
Unlike in the scalar case, the certified region $\{\ifac(\eta;\cdot)<0\}$ is
sufficient but not necessary: initializations outside it may enter the
certified region after one GD step. Appendix~\ref{apx:ifac_tightness} shows
that this occurs on a set of positive measure. Apply
Theorem~\ref{thm:convergence_region_rank_1} to the trajectory after this first
entrance. Corollary~\ref{cor:gd_preimage_null} shows that pulling the
exceptional set back through the first step preserves nullity, so the theorem
applies to almost every initialization in that set.
\end{remark}

\subsection{Isotropic-signal rank-1 approximation}
\label{subsec:rank1_approx_general}

We now consider \textbf{isotropic-signal rank-1 approximation} of
$X=\diag(I_{n-1},0)\in\RR^{n\times n}$ with $n\ge3$. The isotropic signal
subspace makes the minimizer set positive-dimensional.

Decompose $A_t = (a_t^\top,\, u_t)^\top$ and $B_t = (b_t^\top,\, v_t)^\top$
with $a_t,b_t\in\RR^{n-1}$ and $u_t,v_t\in\RR$. The gradient map is the vector
analogue of the factorization map, and the certificate replaces the scalar
product by the inner product:
\begin{align}
    \iapx(\delta;\, A_t, B_t)
    := \delta\bigl(\|A_t\|^2 + \|B_t\|^2\bigr)
       - \delta^2 \langle a_t, b_t \rangle + \delta^2 - 4.
\end{align}

Step 1 changes and determines the range of the result. Relative to
$R_t^\fac(\delta)$, the remainder $R_t^\apx(\delta)$ gains the term
$\eta\,q_\eta(\delta)\,D_t^S$, where
\begin{align}
    q_\eta(\delta) := \eta\delta^2 - 4\delta + 4\eta,
    \qquad
    D_t^S := \|a_t\|^2 \|b_t\|^2 - \langle a_t, b_t \rangle^2 \ge 0.
\end{align}
Thus the
boundary-inward estimate requires $q_\eta(\delta)<0$, a restriction on the
state as well as the step size.

\begin{proposition}\label{prop:boundary_inward_rank1_approx}
Let $0 < \eta < \delta < 2$ and suppose $q_\eta(\delta) < 0$.
If $\iapx(\delta;\, A_t, B_t) = 0$ and $(a_t, b_t, u_t, v_t) \notin \mathcal{S}$,
then
\begin{align}
    \iapx(\delta;\, A_{t+1}, B_{t+1}) = R_t^\apx(\delta) < 0,
\end{align}
and hence the next iterate lies strictly inside the sublevel set $\{\iapx(\delta) \le 0\}$.
\end{proposition}

For $\eta\in(0,1)$ the condition $q_\eta(\delta)<0$ holds exactly on
$(\delta_\mathrm{th}^\apx,2]$, where
$\delta_\mathrm{th}^\apx := 2(1-\sqrt{1-\eta^2})/\eta>\eta$, so the certified
region must be defined at this stronger threshold.

Beyond Step 1, only Steps 4--6 differ. Step 4 gives no pointwise conclusion:
for $n-1\ge2$, the relevant minimizer level contains a continuum rather than
finitely many points. Step 5 still excludes the non-minimizing branch almost
surely because its intersection with each fixed level is finite and isolated
from $\mathcal M$. This leaves $\dist{x_t}{\mathcal M}\to0$ in the nonterminal
case; see Appendix~\ref{subsec:applications_of_theorem5}. In Step 6, convergence
is initially only to
$K_2^\apx=\{(A,B):a=b,\ u=v=0,\ \|a\|^2\le2\}$. For $\eta<1$,
Corollary~\ref{cor:subcritical_terminal_inheritance} upgrades this to
$\dist{x_t}{\mathcal M}\to0$.

\begin{theorem}\label{thm:convergence_region_rank_1_approx}
Let $(A_0,B_0)$ be an initialization and let $\eta \in (0, 1/\sigma)$. Set
\begin{align}
    \delta_\mathrm{th}^\apx
    := \frac{2\bigl(1 - \sqrt{1 - (\eta\sigma)^2}\bigr)}{\eta\sigma}.
\end{align}
Suppose that the rescaled initialization satisfies
$\iapx(\delta_\mathrm{th}^\apx;\,\sigma^{-1/2}(A_0,B_0)) < 0$. Then, for
almost every such initialization the GD dynamics $(A_t,B_t)_{t\ge0}$ of
$\gd_\eta^\apx$ converge to the set of global minimizers.
\end{theorem}

Note that the proof pipeline stops at Step 1 in
the post-critical regime: for $\eta\in(1,2)$ the discriminant of $q_\eta$ is
$16(1-\eta^2)<0$, so $q_\eta(\delta)>0$ for every $\delta$. The sign condition
used to prove the boundary-inward property therefore fails at every state, and this certificate argument does not extend to the post-critical regime.

\begin{remark}[Post-critical approximation dynamics]
\label{rmk:no_edge_of_stability_apx}
For $\eta>1$, the approximation trajectories tested in our experiments
numerically appear to diverge in norm rather than approach a bounded
edge-of-stability regime. This is an empirical observation; see
Appendix~\ref{apx:postcritical} for the protocol and results.
\end{remark}

%% file: 05_lyapunov_method.tex
\section{State-Dependent Lyapunov Method}\label{sec:state_dependent_lyapunov}

One observation from Section~\ref{sec:rank-1 convergence} is that the certificate
$I(\delta;\,\cdot)$ resembles a quadratic Lyapunov function. Indeed, in the scalar
factorization setting it can be written as
\begin{align}
    \isc(\delta;\, a, b)
    = x^\top
      \begin{pmatrix} \delta & -\delta^2/2 \\ -\delta^2/2 & \delta \end{pmatrix}
      x + \delta^2 - 4,
\end{align}
where $x = (a, \; b)^\top$. This is a quadratic function
$x^\top P(\delta)\,x +2 q(\delta)^\top x + r(\delta)$ with $q\equiv 0$.
The key property driving the convergence proof was that for
any $x = (a, \; b)^\top$ on the level set $\isc(\delta;\, a, b) = 0$, the
certificate becomes negative after a gradient step:
\begin{align}
    \isc(\delta;\, \gd_\eta(a, b)) \le \isc(\delta;\, a, b) = 0.
\end{align}
This naturally raises the question of why this particular quadratic form appears and
whether it can be derived from structural requirements rather than guessed a priori.

A fixed quadratic Lyapunov function is too rigid for these dynamics. Local
nonincrease of $|V(x)-V(x_*)|$ near a limiting minimizer $x_*$ requires a
minimizer-dependent Hessian alignment condition, so one fixed candidate is not
generically compatible with the minimizer selected by the trajectory; see
Appendix~\ref{appendix:state-dependent-lyapunov}.

The certificate $\isc(\delta;\cdot)$ avoids this obstruction through a nested
family of quadratic level sets indexed by $\delta\in(0,2]$. As $\delta_t$
increases, the active level set adapts to the selected minimizer and contracts
toward the terminal state $\delta=2$. We formalize this state-dependent
Lyapunov viewpoint, recover the scalar certificate, and constrain its
higher-dimensional analogues.

\subsection{State-dependent quadratic Lyapunov families and the structure of \texorpdfstring{$\isc$}{Isc}}
\label{subsec:structure_of_state_dependent_lyapunov}

In this subsection, we formalize the state-dependent Lyapunov viewpoint
underlying the certificate in Eq.~\eqref{eq:I_delta}. We first define
state-dependent quadratic certificates in general, and then specialize to the
scalar origin-centered, exchange-symmetric subclass considered in the
uniqueness theorem.

Fix a $C^2$ loss function $\risk$ on $\RR^n$, a step size $\eta>0$ with
gradient-descent map $\gd_\eta$, and a state-parameter interval
$(\underline\delta,\overline\delta]$. A
\textbf{state-dependent quadratic certificate} of $\risk$ is a function
$I\colon(\underline\delta,\overline\delta)\times\RR^n\to\RR$ of the form
\begin{align}
    I(\delta;\,x)
    = x^\top P(\delta)\,x + 2\,q(\delta)^\top x + r(\delta),
    \quad
    P(\delta) = P(\delta)^\top \in \RR^{n \times n},
    \;\; q(\delta) \in \RR^n,
    \;\; r(\delta) \in \RR,
\end{align}
where the coefficient functions $P,q,r$ are $C^1$ on the nonterminal state
interval $(\underline\delta,\overline\delta)$. For every nonterminal state,
define
\begin{align}
    K_\delta := \{x \in \RR^n : I(\delta;\,x) \le 0\}
    \: \text{for}\ \delta \in (\underline\delta, \overline\delta).
\end{align}
The terminal set is
\begin{align}
    K_{\overline\delta}
    := \bigcap_{\underline\delta < \delta < \overline\delta} K_\delta.
\end{align}
The terminal set $K_{\overline\delta}$ is the nested intersection
$\bigcap_\delta K_\delta$; the certificate need not be defined at the terminal
endpoint.

We assume the following state-dependent Lyapunov axioms.

\textbf{Structural axioms.}
\begin{enumerate}[label=(A\arabic*), ref=A\arabic*]
    \item \label{cond:P_is_pd}
    \textbf{(Compactness)}
    For every $\delta \in (\underline\delta, \overline\delta)$, $K_\delta$ is
    nonempty and compact.

    \item \label{cond:nesting}
    \textbf{(Level-set nesting)}
    For $\delta, \delta' \in (\underline\delta, \overline\delta]$ with
    $\delta < \delta'$, $K_{\delta'} \subseteq \mathrm{int}(K_\delta)$.

    \item \label{cond:end_point_measure_zero}
    \textbf{(Terminal negligibility)}
    $K_{\overline\delta}$ has measure zero.

    \item \label{cond:level_set_as_state}
    \textbf{(Level set as a state)}
    Every
    $x\in\bigl(\bigcup_{\underline\delta<\delta<\overline\delta}K_\delta\bigr)
    \setminus K_{\overline\delta}$
    lies on $\partial K_\delta$ for exactly one
    $\delta\in(\underline\delta,\overline\delta)$. We write $\delta(x)$ for
    that state, and set $\delta(x)=\overline\delta$ on $K_{\overline\delta}$.
\end{enumerate}

The next two are required only above a threshold
$\delta_{\mathrm{th}}\in(\underline\delta,\overline\delta)$.

\textbf{Dynamical axioms.}
\begin{enumerate}[label=(A\arabic*), ref=A\arabic*, start=5]
    \item \label{cond:monotonicity}
    \textbf{(Boundary-inward property)}
    For every $\delta \in (\delta_{\mathrm{th}}, \overline\delta)$,
    \begin{align}
        \gd_\eta(\partial K_\delta)\subseteq K_\delta.
    \end{align}

    \item \label{cond:stationarity}
    \textbf{(Stationarity)}
    For every $\delta \in (\delta_{\mathrm{th}}, \overline\delta)$ and every
    $x\in\partial K_\delta$,
    \begin{align}
        \gd_\eta(x)\in\partial K_\delta
        \quad\Longleftrightarrow\quad
        \nabla\risk(x)=0.
    \end{align}
\end{enumerate}

By Axioms~\ref{cond:P_is_pd} and~\ref{cond:nesting} every nonterminal
$K_\delta$ is a nondegenerate ellipsoid
(Lemma~\ref{lem:compact_quadratic_sublevel} in
Appendix~\ref{subsec:theorem5_and_auxiliary_lemmas}), so
$\partial K_\delta=\{x:I(\delta;\,x)=0\}$ and
$\operatorname{int}K_\delta=\{x:I(\delta;\,x)<0\}$. The two dynamical axioms may
therefore be read either geometrically, as above, or through the sign of $I$.

Being a nondegenerate ellipsoid, each nonterminal $K_\delta$ is also strictly
convex, and that alone yields a step-size downgrade:
Axiom~\ref{cond:monotonicity} at $\eta_0$ implies both dynamical axioms at
every $\eta\in(0,\eta_0)$; see Lemma~\ref{lem:step_size_downgrade} in
Appendix~\ref{appendix:uniqueness_state_dependent_lyapunov}.

The axioms use $I$ only through the ordered sublevel sets it generates. To record the sets together with their ordering, write
\begin{align}
    \mathcal K_I
    :=(K_\delta)_{\underline\delta<\delta\le\overline\delta}.
\end{align}
If $\mathcal K_I$ satisfies the structural axioms and, for the fixed update map
$\gd_\eta$, the dynamical axioms hold above a threshold
$\delta_{\mathrm{th}}$, we call the triplet
$(\mathcal K_I,\gd_\eta,\delta_{\mathrm{th}})$ a
\textbf{state-dependent Lyapunov family} generated by the certificate
$I$.

Let $(\mathcal K_I,\gd_\eta,\delta_{\mathrm{th}})$ and
$(\mathcal K_{\widetilde I},\gd_\eta,\zeta_{\mathrm{th}})$ be
state-dependent Lyapunov families for the same update map --- we compare
families only at a common step size --- with terminal state parameters
$\overline\delta$ and $\overline\zeta$, respectively. We call the
two families \textbf{equivalent} if there exists a strictly increasing
continuous bijection
$\phi:(\delta_{\mathrm{th}},\overline\delta)
\to(\zeta_{\mathrm{th}},\overline\zeta)$ such that
\begin{align}
    K_\delta=\widetilde K_{\phi(\delta)}
    \qquad\text{for every }
    \delta\in(\delta_{\mathrm{th}},\overline\delta).
    \label{eq:family_equivalence}
\end{align}
Thus equivalence identifies families with the same dynamics and ordered
level-set geometry up to a monotone relabeling of the state. Multiplying $I$ by
any positive function of the state leaves $\mathcal K_I$, and hence the whole
triplet, unchanged.

For a state-dependent Lyapunov family
$(\mathcal K_I,\gd_\eta,\delta_{\mathrm{th}})$, define the
\textbf{certified region}
\begin{align}
    \mathcal C
    :=\bigcup_{\delta_{\mathrm{th}}<\delta<\overline\delta}K_\delta
    =\operatorname{int}K_{\delta_{\mathrm{th}}}
    =\{x:I(\delta_{\mathrm{th}};\,x)<0\},
    \label{eq:certified_region_threshold}
\end{align}
the second equality by coefficient continuity and strict nesting. Its points
are called \textbf{certified initializations}.

Each of $\isc$, $\ifac$ and $\iapx$, together with its applicable update map
and a suitable threshold, induces a state-dependent Lyapunov family. This is
verified in Lemma~\ref{lem:concrete_axioms} of
Appendix~\ref{apdx:abstract_convergence}.

Under the hypotheses of Theorem~\ref{thm:abstract_state_dependent_convergence},
outside a measure-zero set of certified initializations the state
$\delta_t:=\delta(x_t)$ is well defined along the whole trajectory and strictly
increasing. The theorem makes the exceptional set precise and turns the
monotonicity into the dichotomy between a nonterminal limit and the terminal
set; its proof is in Appendix~\ref{apdx:abstract_convergence}.

The two dynamical axioms turn each fixed level set into a constrained
optimization problem. Fix $\delta\in(\delta_{\mathrm{th}},\overline\delta)$.
Axiom~\ref{cond:monotonicity} bounds $x\mapsto I(\delta;\,\gd_\eta(x))$ by zero
on $\{I(\delta;\,x)=0\}$, and Axiom~\ref{cond:stationarity} says the bound is
attained exactly at the stationary points of $\risk$ there. Those points are
therefore constrained maximizers, and the first-order necessary condition
applies to them. With exchange symmetry
$P_{11}(\delta)=P_{22}(\delta)$ and origin centering ($q\equiv0$), these
conditions pin the induced family down completely, up to equivalence.
The details are deferred to
Appendix~\ref{appendix:uniqueness_state_dependent_lyapunov}.

\begin{theorem}[Uniqueness up to equivalence]\label{thm:uniqueness_state_dependent_lyapunov}
For the scalar factorization problem, let $\eta \in (0,2)$ and let $I$ be an
origin-centered ($q\equiv0$) state-dependent quadratic certificate with state
interval $(\underline\delta,\overline\delta]$. Suppose that
$(\mathcal K_I,\gd_\eta,\delta_{\mathrm{th}})$ is a state-dependent Lyapunov
family and that $I$ satisfies the exchange symmetry
\begin{enumerate}[label=(A\arabic*), ref=A\arabic*, start=7]
\item \label{cond:symmetry}
\textbf{(Symmetry)} For every $\delta \in (\underline\delta, \overline\delta)$,
$P_{11}(\delta) = P_{22}(\delta)$.
\end{enumerate}
Then there exists $\zeta_{\mathrm{th}}\in[\eta,2)$ such that
$(\mathcal K_I,\gd_\eta,\delta_{\mathrm{th}})$ is equivalent to
\begin{align}
    (\mathcal K_{\isc},\gd_\eta,\zeta_{\mathrm{th}}).
\end{align}
\end{theorem}

Thus the triplet induced by $\isc$ represents, up to equivalence, the unique
origin-centered, exchange-symmetric quadratic state-dependent Lyapunov family
in this class. The origin-centering hypothesis $q\equiv0$ is essential for
that uniqueness: for the same loss, and for every $\eta\in(0,1)$,
Appendix~\ref{subsec:moving_center_sharpness} exhibits an exchange-symmetric
triplet induced by a certificate with a state-dependent center that satisfies
all the axioms but is not equivalent to the triplet induced by $\isc$.

\subsection{Extensions beyond the proved cases}
\label{subsec:extensions_beyond}

We discuss two extensions beyond the established convergence regimes. The
first constructs the certificate family analytically but tests its dynamical
axiom only numerically. The second proves an analytic extension under a
quartic augmentation and uses its endpoint failure to motivate empirical
moving-center candidates for more general objectives.

\subsubsection{A two-parameter approximation certificate}

Consider the 2-dimensional rank-1 approximation problem
$X=\diag(1,\sigma)$ with $\sigma\in(0,1)$, $A=(a,u)^\top$, $B=(b,v)^\top$, and
$x=(a,b,u,v)^\top\in\RR^4$.
Since the problem reduces exactly to scalar factorization on the invariant
slice $u=v=0$, compatibility with the scalar certificate forces the signal
block to inherit the same $\delta$-parameterization over $(0,2]$. A local
Lagrange analysis at the signal and noise stationary slices, together with the
natural exchange symmetry, then reduces the admissible quadratic family to the
two-parameter certificates (see Appendix~\ref{appendix:state_dependent_higher_dim} for the derivation)
\begin{align}
    I(\delta,\xi;a,b,u,v)
    :=
    \frac{\delta(a^2+b^2)-\delta^2 ab}{4-\delta^2}
    +
    \frac{\xi(u^2+v^2)-\xi^2\sigma\,uv}
    {4-\xi^2\sigma^2}
    -1,
    \quad
    \delta\in(0,2),\,\xi\in(0,2/\sigma).
    \label{eq:diag_1_sigma_two_param_certificate}
\end{align}

To extract a one-parameter candidate from this family, we single out a branch
$\xi=\xi(\delta)$ motivated by a shared feature of the two certificates proved
in this paper: their signal and secondary blocks have the same diagonal
quadratic coefficient. Imposing this diagonal-matching rule on
Eq.~\eqref{eq:diag_1_sigma_two_param_certificate} gives
\begin{align}
    \frac{\delta}{4-\delta^2}=\frac{\xi}{4-\sigma^2\xi^2}
    \quad\Longleftrightarrow\quad
    \xi=\xi_\sigma(\delta)
    :=\frac{8\delta}
        {(4-\delta^2)+\sqrt{(4-\delta^2)^2+16\sigma^2\delta^2}} .
    \label{eq:main_xi_selector}
\end{align}
Using Proposition~\ref{prop:xi_selector},
Appendix~\ref{subsec:diag1sigma_terminal_reduction} verifies
Axioms~\ref{cond:P_is_pd}--\ref{cond:level_set_as_state} and $C^1$ coefficient
regularity on the nonterminal interval $(0,2)$. The dynamical axioms remain
unproved: Axiom~\ref{cond:monotonicity} has only been tested numerically at
finitely many sampled boundary points, while Axiom~\ref{cond:stationarity} has
not been established; see Appendix~\ref{apx:xi_selector_verification}.
The same diagonal-matching rule also appears to satisfy
Axiom~\ref{cond:monotonicity} in higher dimensions, based on the finite sampled
boundary tests reported in Appendix~\ref{apx:general_spectrum}.

\subsubsection{Quartic augmentations and \texorpdfstring{$f(ab)$}{f(ab)}-type losses}

For the quartic-augmented scalar factorization loss
\begin{align}
    \risk_\mu(a,b)
    =
    \tfrac12(ab-1)^2+\mu(ab-1)^4,
\end{align}
the GD update is the scalar-factorization update with the iterate-dependent
effective step $\eta[1+4\mu(ab-1)^2]$.
Proposition~\ref{prop:quartic_boundary_inward_sufficient} shows that, for every
$\mu>-1/4$ and a suitable step size, the same origin-centered certificate
$\isc$ satisfies Axiom~\ref{cond:monotonicity} on a nonempty terminal-side
interval of states.

At $\mu=-1/4$ the analytic condition above yields no nonterminal range, and
sampled tests of $\isc$ find violations, although numerical trajectories still
exhibit a nontrivial region converging to $\{ab=1\}$. Appendix~\ref{apx:moving_center}
therefore uses this endpoint to motivate a moving-center family. No violation
of Axiom~\ref{cond:monotonicity} was observed for the candidate on the sampled
interval, and the corresponding sublevel family partially captures the
observed convergence region.
Similar finite-sample behavior for other polynomial $f(ab)$-type losses
suggests that this construction is not specific to the quartic example; no
general verification of Axiom~\ref{cond:monotonicity} is claimed.

%% file: 06_conclusion.tex
\section{Conclusion}

We developed a state-dependent Lyapunov framework for discrete gradient
descent. A certificate generates an ordered collection of nested quadratic
level sets whose ordering implicitly defines a scalar state function. The
boundary-inward property makes this state monotone along the trajectory, while
the geometry of the level sets supplies the convergence argument. Depending on
whether the limiting state is nonterminal
or terminal, this geometry can yield pointwise convergence or reduce the
asymptotic analysis to the dynamics on the terminal set.

The framework is formulated in terms of ordered level sets and is therefore
unchanged by a monotone relabeling of the state. Under suitable additional
assumptions, the scalar origin-centered, exchange-symmetric quadratic
certificate is unique up to this equivalence. The moving-center construction
shows that origin-centering is essential for the uniqueness statement. The
proved factorization examples, together with the partial analytic and numerical
extensions, suggest that these certificates reflect a recurring geometric
structure rather than an isolated algebraic construction.

Several questions remain open at the level of the framework. First, can a
certificate be constructed systematically from a given gradient map and its
stationary geometry? Second, when can the dynamical axioms be established
algebraically in higher dimensions? Finally, our analysis uses internal chain
transitivity to relate the $\omega$-limit set of a full GD trajectory to the
reduced dynamics on the terminal set. Can this relation be formulated as a
general principle that is independent of the particular GD map? Addressing
these questions would clarify how broadly state-dependent Lyapunov analysis
applies beyond the examples considered here.

%% file: apxA_model_catalog.tex
\section{Model catalog: objectives, gradient maps, and stationary geometry}
\label{apx:model_catalog}

This appendix fixes the model-specific notation used throughout the paper. Its
purpose is to provide a single reference for each objective, its normalized
gradient-descent map, its global minimizers, and its full stationary set.

For a target matrix $X\in\RR^{m\times n}$ and rank-one factors
$A\in\RR^n$, $B\in\RR^m$, the underlying matrix objective and its GD map are
\begin{align}
    \risk_X(A,B)
    &:=\frac12\|BA^\top-X\|_F^2,
    \label{eq:model_catalog_general_loss}\\
    A^+&=A-\eta(BA^\top-X)^\top B,
    \qquad
    B^+=B-\eta(BA^\top-X)A.
    \label{eq:model_catalog_general_gd}
\end{align}
For any step size $\eta>0$, the stationary points of an objective $\risk$ are
exactly the fixed points of its gradient-descent map:
\begin{align}
    \nabla\risk(x)=0
    \quad\Longleftrightarrow\quad
    \gd_\eta(x)=x.
    \label{eq:stationary_fixed_equivalence}
\end{align}
Unless stated otherwise, the matrix objectives below are normalized so that
their leading nonzero singular value is one.

\subsection{The three core models}
\label{subsec:model_catalog_core}

\paragraph{Scalar factorization.}
The scalar objective and its GD map are
\begin{align}
    \risk_\sca(a,b)
    &:=\frac12(ab-1)^2,
    \label{eq:model_catalog_scalar_loss}\\
    a^+&=(1-\eta b^2)a+\eta b,
    \qquad
    b^+=(1-\eta a^2)b+\eta a.
    \label{eq:model_catalog_scalar_gd}
\end{align}
Its minimizer and stationary sets are
\begin{align}
    \mathcal M_\sca
    &:=\{(a,b)\in\RR^2:ab=1\},\\
    \mathcal S_\sca
    &:=\mathcal M_\sca\cup\{(0,0)\}.
    \label{eq:model_catalog_scalar_sets}
\end{align}

\paragraph{A unified signal--off-signal form.}
Rank-1 factorization and isotropic-signal rank-1 approximation share the same
normalized reduced dynamics. Let $a,b\in\RR^d$ be signal coordinates and
$u,v\in\RR$ be off-signal coordinates. One GD step is
\begin{align}
\begin{aligned}
    b^+&=\bigl(1-\eta(\|a\|^2+u^2)\bigr)b+\eta a,
    &\qquad
    v^+&=\bigl(1-\eta(\|a\|^2+u^2)\bigr)v,\\
    a^+&=\bigl(1-\eta(\|b\|^2+v^2)\bigr)a+\eta b,
    &
    u^+&=\bigl(1-\eta(\|b\|^2+v^2)\bigr)u.
\end{aligned}
\label{eq:model_catalog_unified_gd}
\end{align}
For rank-1 matrix factorization, $d=1$ and the normalized target is
$X=\diag(1,0)$. For isotropic-signal rank-1 approximation,
$d=n-1\ge2$ and the target is
\begin{align}
    X=\diag(I_d,0)\in\RR^{(d+1)\times(d+1)}.
\end{align}
The latter name emphasizes that the nonzero singular subspace is isotropic and
that the model retains one null direction.

Writing $A=(a^\top,u)^\top$ and $B=(b^\top,v)^\top$, the objective behind
Eq.~\eqref{eq:model_catalog_unified_gd} is
\begin{align}
    \risk_d(a,b,u,v)
    :=\frac12\left(
        \|ba^\top-I_d\|_F^2
        +u^2\|b\|^2+v^2\|a\|^2+u^2v^2
    \right).
    \label{eq:model_catalog_unified_loss}
\end{align}
We write $\risk_d$ only here, to treat the two models at once: for $d=1$ it is
$\risk_\fac$, and for $d\ge2$ it is $\risk_\apx$, the normalized
isotropic-signal rank-1 approximation objective.

The corresponding minimizer and stationary sets are
\begin{align}
    \mathcal M_\fac
    &:=\{(a,b,u,v)\in\RR^4:ab=1,\ u=v=0\},\\
    \mathcal S_\fac
    &:=\mathcal M_\fac
      \cup\{(a,b,u,v)\in\RR^4:a=b=0,\ uv=0\},
    \label{eq:model_catalog_fac_sets}\\
    \mathcal M_\apx
    &:=\{(a,b,u,v):a\parallel b,\ \langle a,b\rangle=1,\ u=v=0\},\\
    \mathcal S_\apx
    &:=\mathcal M_\apx
      \cup\{(a,b,u,v):a=b=0,\ uv=0\},
    \label{eq:model_catalog_apx_sets}
\end{align}
where $a,b\in\RR^d$ in the approximation model. In particular, the
non-minimizing stationary branch is the same pair of off-signal axes in both
models, whereas $\mathcal M_\apx$ is positive-dimensional because its signal
direction is arbitrary.

\begin{proposition}[Stationary-set classification for the core models]
\label{prop:model_catalog_core_stationary}
Eqs.~\eqref{eq:model_catalog_scalar_sets},
\eqref{eq:model_catalog_fac_sets}, and
\eqref{eq:model_catalog_apx_sets} give the full stationary sets of the three
core objectives. Each stationary set has measure zero in its ambient Euclidean
space.
\end{proposition}

\begin{proof}
For scalar factorization,
\begin{align}
    \nabla\risk_\sca(a,b)=\bigl((ab-1)b,(ab-1)a\bigr),
\end{align}
which vanishes exactly when $ab=1$ or $a=b=0$.

For either model governed by Eq.~\eqref{eq:model_catalog_unified_gd}, set
\begin{align}
    \alpha:=\|a\|^2+u^2,
    \qquad
    \beta:=\|b\|^2+v^2.
\end{align}
Stationarity is equivalent to
\begin{align}
    \beta a=b,
    \qquad
    \beta u=0,
    \qquad
    \alpha b=a,
    \qquad
    \alpha v=0.
    \label{eq:model_catalog_core_stationarity_equations}
\end{align}
If $\alpha\beta>0$, then $u=v=0$, $a$ and $b$ are parallel, and
$\alpha\beta=1$. Hence $\langle a,b\rangle=1$, which gives the stated
minimizer set; when $d=1$, this reduces to $ab=1$. If $\alpha=0$, then
$a=u=0$, and the first equation in
Eq.~\eqref{eq:model_catalog_core_stationarity_equations} also gives $b=0$,
while $v$ is arbitrary. The case $\beta=0$ gives the other off-signal axis.
These cases exhaust the stationary set.

The scalar and factorization minimizer claims follow directly from the
nonnegative loss. For $X=\diag(I_d,0)$, the best rank-one approximation selects
an arbitrary unit direction in the isotropic signal subspace; the factor
scaling is free, yielding precisely $\mathcal M_\apx$. Finally, each displayed
stationary set is a finite union of lower-dimensional algebraic sets and hence
has measure zero.
\end{proof}

\subsection{Two-dimensional rank-1 approximation}
\label{subsec:model_catalog_two_mode}

The two-dimensional approximation model used in
Subsection~\ref{subsec:extensions_beyond} and
Appendices~\ref{appendix:state_dependent_higher_dim}
and~\ref{apx:xi_selector_verification} has
\begin{align}
    X=\diag(1,\sigma),
    \qquad
    0<\sigma<1,
\end{align}
and rank-one factors $A=(a,u)^\top$ and $B=(b,v)^\top$. Its loss is
\begin{align}
    \risk_\sigma(a,b,u,v)
    :=\frac12\Bigl((ab-1)^2+b^2u^2+a^2v^2+(uv-\sigma)^2\Bigr),
    \label{eq:model_catalog_two_mode_loss}
\end{align}
and its GD map is
\begin{align}
\begin{aligned}
    b^+&=\bigl(1-\eta(a^2+u^2)\bigr)b+\eta a,
    &\qquad
    v^+&=\bigl(1-\eta(a^2+u^2)\bigr)v+\eta\sigma u,\\
    a^+&=\bigl(1-\eta(b^2+v^2)\bigr)a+\eta b,
    &
    u^+&=\bigl(1-\eta(b^2+v^2)\bigr)u+\eta\sigma v.
\end{aligned}
\label{eq:model_catalog_two_mode_gd}
\end{align}
Because the singular values are distinct, the best rank-one approximation uses
the leading mode. The minimizer and full stationary sets are
\begin{align}
    \mathcal M_\sigma
    &:=\{(a,b,0,0):ab=1\},\\
    \mathcal S_\sigma
    &:=\mathcal M_\sigma
      \cup\{(0,0,u,v):uv=\sigma\}
      \cup\{(0,0,0,0)\}.
    \label{eq:model_catalog_two_mode_sets}
\end{align}
The restriction $0<\sigma<1$ is essential:
at the degenerate endpoint $\sigma=1$, the stationary geometry changes because
the leading singular subspace is no longer one-dimensional.

To verify Eq.~\eqref{eq:model_catalog_two_mode_sets}, write
$\alpha:=\|A\|^2$ and $\beta:=\|B\|^2$. Stationarity is equivalent to
\begin{align}
    \beta A=XB,
    \qquad
    \alpha B=XA.
\end{align}
Here $X$ is invertible, unlike the targets of the core models, so if one factor
vanishes the other must vanish as well. Otherwise,
\begin{align}
    \alpha\beta A=X^2A.
\end{align}
Since $X^2=\diag(1,\sigma^2)$ has distinct eigenvalues, $A$, and consequently
$B$, lies entirely in either the leading or the secondary coordinate. The two
possibilities give $ab=1$ and $uv=\sigma$, respectively. The singular-value
gap makes only the leading branch globally minimizing.

\subsection{Composite scalar objectives}
\label{subsec:model_catalog_composite_scalar}

Let
$\overline\risk(\bar a,\bar b)=f(\bar a\bar b)$ with $f\in C^2$, and suppose
$m\ne0$ satisfies $f'(m)=0$ and $f''(m)>0$. Then
$\{\bar a\bar b=m\}$ is a local-minimizer manifold. Introduce normalized
coordinates
\begin{align}
    a=\frac{\bar a}{\sqrt{|m|}},
    \qquad
    b=\frac{\operatorname{sgn}(m)\bar b}{\sqrt{|m|}},
    \label{eq:model_catalog_composite_rescaling}
\end{align}
so that $\bar a\bar b=mab$ and the target manifold becomes
$\mathcal M=\{ab=1\}$. In these coordinates the objective is
$\risk_m(a,b):=f(mab)$, and a GD step of size $\eta$ in the original
coordinates becomes
\begin{align}
    (a^+,b^+)
    &=(a,b)-\eta\operatorname{sgn}(m)f'(mab)(b,a) \\
    &=(a,b)-\frac{\eta}{|m|}\nabla\risk_m(a,b).
    \label{eq:model_catalog_composite_gd}
\end{align}
Thus the normalized step size is $\eta/|m|$. At
$x_*=(a_*,1/a_*)^\top\in\mathcal M$,
\begin{align}
    \nabla^2\risk_m(x_*)
    =
    f''(m)m^2
    \begin{pmatrix}
        a_*^{-2} & 1 \\
        1 & a_*^2
    \end{pmatrix}.
    \label{eq:model_catalog_composite_hessian}
\end{align}
Its null direction $(a_*^2,-1)^\top$ is tangent to $\mathcal M$, while
$(1,a_*^2)^\top$ is the transverse eigendirection. These directions are
independent of $f$; only the Hessian scale, and hence the effective step
$\eta f''(m)|m|$, depends on the local curvature of $f$. Scalar factorization
and the quartic family below both have $m=1$ and $f''(1)=1$.

\subsection{Quartic-augmented scalar factorization}
\label{subsec:model_catalog_quartic}

For $\mu\in\RR$, let
\begin{align}
    \risk_\mu(a,b)
    :=\frac12(ab-1)^2+\mu(ab-1)^4.
    \label{eq:model_catalog_quartic_loss}
\end{align}
Writing $L:=1-ab$, its GD map is
\begin{align}
    a^+&=a+\eta L(1+4\mu L^2)b,
    &
    b^+&=b+\eta L(1+4\mu L^2)a.
    \label{eq:model_catalog_quartic_gd}
\end{align}
Define the set of stationary products
\begin{align}
    Z_\mu
    :=\begin{cases}
        \{1\}, & \mu\ge0,\\[2pt]
        \displaystyle
        \left\{1,\ 1-\frac{1}{2\sqrt{-\mu}},\
                     1+\frac{1}{2\sqrt{-\mu}}\right\},
            & \mu<0.
    \end{cases}
\end{align}
Then the full stationary set is
\begin{align}
    \mathcal S_\mu
    =\{(0,0)\}\cup\{(a,b):ab\in Z_\mu\}.
    \label{eq:model_catalog_quartic_stationary}
\end{align}
Indeed, for $p:=ab$ and
$f_\mu(p):=\frac12(p-1)^2+\mu(p-1)^4$,
\begin{align}
    \nabla\risk_\mu(a,b)
    =f_\mu'(p)(b,a),
    \qquad
    f_\mu'(p)=(p-1)\bigl(1+4\mu(p-1)^2\bigr),
\end{align}
which gives Eq.~\eqref{eq:model_catalog_quartic_stationary}.
For $\mu\ge0$, the global minimizer set is
$\mathcal M_\mu=\{ab=1\}$. For $\mu<0$, the objective is unbounded below, so
it has no global minimizer; nevertheless, $\{ab=1\}$ remains a local-minimizer
manifold and is the target manifold used in the local and numerical analyses.
In particular, when $\mu=-1/4$,
\begin{align}
    Z_{-1/4}=\{0,1,2\},
\end{align}
so the stationary set additionally contains the coordinate axes $\{ab=0\}$
and the hyperbola $\{ab=2\}$.

\subsection{Regularized rank-1 factorization}
\label{subsec:model_catalog_regularized}

For $\lambda>0$, the regularized objective is
\begin{align}
    \risk_\lambda(x)
    :=\risk_\fac(x)+\frac{\lambda}{2}\|x\|^2,
    \qquad
    x=(a,b,u,v)\in\RR^4,
    \label{eq:model_catalog_regularized_loss}
\end{align}
and its GD map is
\begin{align}
    \gd_\eta^\lambda(x)
    =(1-\eta\lambda)x-\eta\nabla\risk_\fac(x).
    \label{eq:model_catalog_regularized_gd}
\end{align}
The regularizer removes the factor-scaling symmetry. The stationary and
minimizer sets are
\begin{align}
    \mathcal S_\lambda
    &=\begin{cases}
        \{(0,0,0,0)\}
        \cup
        \{(s,s,0,0):s=\pm\sqrt{1-\lambda}\},
            &0<\lambda<1,\\
        \{(0,0,0,0)\}, &\lambda\ge1,
    \end{cases}
    \label{eq:model_catalog_regularized_stationary}\\
    \mathcal M_\lambda
    &=\begin{cases}
        \{(s,s,0,0):s=\pm\sqrt{1-\lambda}\},
            &0<\lambda<1,\\
        \{(0,0,0,0)\}, &\lambda\ge1.
    \end{cases}
    \label{eq:model_catalog_regularized_minimizers}
\end{align}

Indeed, the off-signal stationarity equations are
\begin{align}
    (b^2+v^2+\lambda)u=0,
    \qquad
    (a^2+u^2+\lambda)v=0.
\end{align}
Since $\lambda>0$, they force $u=v=0$. The remaining two equations are
\begin{align}
    a(b^2+\lambda)=b,
    \qquad
    b(a^2+\lambda)=a.
\end{align}
They give either $a=b=0$ or $a=b=\pm\sqrt{1-\lambda}$, with the latter possible
only for $\lambda<1$. For the minimizers, the off-signal terms are
nonnegative and vanish at $u=v=0$, while
\begin{align}
    \frac12(ab-1)^2+\frac{\lambda}{2}(a^2+b^2)
    \ge \frac12(ab-1)^2+\lambda|ab|.
\end{align}
Write $p:=ab$ and minimize the right-hand side over $p$. On $p<0$ the
derivative is $p-1-\lambda<0$, so no minimum occurs there; on $p\ge0$ it gives
$p=1-\lambda$ for $0<\lambda<1$ and the boundary value $p=0$ for
$\lambda\ge1$. Equality in
$a^2+b^2\ge2|ab|$ gives
Eq.~\eqref{eq:model_catalog_regularized_minimizers}.

%% file: apxB_gd_submersion.tex
\section{The gradient descent map is a submersion almost everywhere}
\label{apdx:gd_submersion}

In this appendix we prove Corollary~\ref{cor:gd_preimage_null} and verify its
hypothesis for the maps used in the main text. For a fixed step size $\eta>0$, define the
gradient descent map associated with  Eq.~\eqref{eq:gd_dynamics} by
\begin{align}
\gd_\eta(A,B)
:=
\Bigl(
A - \eta (BA^\top - X)^\top B,\;
B - \eta (BA^\top - X)A
\Bigr),
\end{align}
where $A\in \RR^{n\times r}$ and $B\in \RR^{m\times r}$.

\begin{proposition}
\label{prop:gd_submersion_ae}
For every fixed $\eta>0$ and $m \ge r$, the map
\begin{align}
\gd_\eta:\RR^{n\times r}\times \RR^{m\times r}
\to
\RR^{n\times r}\times \RR^{m\times r}
\end{align}
is a submersion almost everywhere.
\end{proposition}

\begin{proof}
First rewrite the update as
\begin{align}
\gd_\eta(A,B)
=
\Bigl(
A - \eta A(B^\top B) + \eta X^\top B,\;
B - \eta B(A^\top A) + \eta XA
\Bigr).
\end{align}
In particular, each component of $\gd_\eta$ is a polynomial in the entries of $(A,B)$.

Let $(H_A,H_B)\in \RR^{n\times r}\times \RR^{m\times r}$ be a perturbation.
Differentiating the formula above gives
\begin{align}
D\gd_\eta(A,B)[H_A,H_B]
=
\Bigl(
&H_A - \eta H_A(B^\top B) - \eta A(H_B^\top B + B^\top H_B) + \eta X^\top H_B, \notag\\
&H_B - \eta H_B(A^\top A) - \eta B(H_A^\top A + A^\top H_A) + \eta XH_A
\Bigr).
\label{eq:DGD_general}
\end{align}

We now evaluate this differential at a convenient point.
Let
\begin{align}
\bar A := 0,
\quad
\bar B := t\begin{pmatrix} I_r \\ 0 \end{pmatrix}\in \RR^{m\times r},
\end{align}
where $t>1/\sqrt{\eta}$.
Since $\bar B^\top \bar B = t^2 I_r$,  Eq.~\eqref{eq:DGD_general} simplifies to
\begin{align}
D\gd_\eta(\bar A,\bar B)[H_A,H_B]
=
\Bigl(
(1-\eta t^2)H_A + \eta X^\top H_B,\;
H_B + \eta XH_A
\Bigr).
\label{eq:DGD_special}
\end{align}

We claim that this linear map is injective.
Suppose $D\gd_\eta(\bar A,\bar B)[H_A,H_B] = (0,0)$.
Then from the second component of  Eq.~\eqref{eq:DGD_special}, we have $H_B = -\eta XH_A$.
Substituting into the first component yields
\begin{align}
\bigl((1-\eta t^2)I_n - \eta^2 X^\top X\bigr)H_A = 0.
\end{align}
Taking Frobenius inner products with $H_A$ gives
\begin{align}
0
&=
\langle H_A,\bigl((1-\eta t^2)I_n - \eta^2 X^\top X\bigr)H_A\rangle \\
&=
(1-\eta t^2)\norm{H_A}_F^2 - \eta^2 \norm{XH_A}_F^2.
\end{align}
Because $t>1/\sqrt{\eta}$, we have $1-\eta t^2<0$, and therefore
\begin{align}
(1-\eta t^2)\norm{H_A}_F^2 - \eta^2 \norm{XH_A}_F^2 < 0
\quad\text{for every } H_A\neq 0.
\end{align}
Hence necessarily $H_A=0$, and then also $H_B=0$.
Thus $D\gd_\eta(\bar A,\bar B)$ is injective.
Since the domain and codomain have the same finite dimension, it is an isomorphism.

The determinant of the Jacobian of $\gd_\eta$ is nonzero at
$(\bar A,\bar B)$ and is therefore a nonzero polynomial in $(A,B)$. Its zero
set has measure zero~\citep{mityagin2015zero}, so the singular set
\begin{align}
\{(A,B): \det D\gd_\eta(A,B)=0\}
\end{align}
has measure zero.
Therefore, $\gd_\eta$ is a submersion almost everywhere.
\end{proof}

We next prove Corollary~\ref{cor:gd_preimage_null}.

\begin{proof}[Proof of Corollary~\ref{cor:gd_preimage_null}]
Let $g:\RR^d\to\RR^d$ be $C^1$, let $Z$ be the measure-zero set off which $g$ is a
submersion, and let $E\subseteq\RR^d$ have measure zero. On
$\RR^d\setminus Z$ the map $g$ is a submersion, so Ponomarev's
theorem~\citep{ponomarev1987submersions} gives that
$g^{-1}(E)\setminus Z$ has measure zero; since $Z$ does too,
$g^{-1}(E)$ has measure zero. An induction on $T$ then shows the same for
$g^{-T}(E)$ for every $T\ge 0$, and a countable union of sets of measure zero
has measure zero, so
$\bigcup_{T=0}^{\infty}g^{-T}(E)$ has measure zero.
\end{proof}

The three maps $\gd_\eta^\sca$, $\gd_\eta^\fac$, and $\gd_\eta^\apx$ are direct
specializations of Proposition~\ref{prop:gd_submersion_ae}. The scalar model has
$n=m=r=1$, while the latter two have $r=1$ and the normalized targets recorded
in Appendix~\ref{apx:model_catalog}. Hence all three maps are submersions almost
everywhere.
Applying the corollary to the stationary and terminal sets gives the regularity
property used throughout Section~\ref{sec:rank-1 convergence}.

\begin{proposition}[Finite-time avoidance of exceptional sets]
\label{prop:regularity}
Let $\gd_\eta$ be any of $\gd_\eta^\sca$, $\gd_\eta^\fac$, $\gd_\eta^\apx$ with
step size $\eta>0$, and let $\mathcal S$ and $K_2$ be the stationary and
terminal sets of the corresponding model. Then the set of initializations whose
trajectory meets $\mathcal S\cup K_2$ at some finite time has measure zero. In
particular, for almost every initialization no finite iterate is stationary or
terminal.
\end{proposition}

\begin{proof}
By Proposition~\ref{prop:model_catalog_core_stationary}, the stationary set
$\mathcal S$ has measure zero in each model. If $x\in K_2$, then
$I(\delta;x)\le 0$ for every $\delta<2$. Continuity in $\delta$ and the endpoint
evaluations
\begin{align}
\isc(2;\cdot)&=2(a-b)^2,\\
\ifac(2;\cdot)&=2(a-b)^2+2(u^2+v^2),\\
\iapx(2;\cdot)&=2\norm{a-b}^2+2(u^2+v^2)
\end{align}
therefore show that $K_2$ is contained in $\{a=b\}$ in the scalar model and in
$\{a=b,\ u=v=0\}$ in the other two models. These are proper linear subspaces,
so $K_2$ also has measure zero. The union $\mathcal S\cup K_2$ is therefore
null, and Corollary~\ref{cor:gd_preimage_null} applies to it.
\end{proof}

%% file: apxC_scalar_vector.tex
\section{Extension to scalar-vector factorization}
\label{app:scalar_vector_proof}
Consider the scalar-vector factorization problem studied in \cite{liang2025gradient},
which is the special case $m = n = 1$, $r = d$ of  Eq.~\eqref{eq:lrmf_problem}:
\begin{align}
    \min_{a,\, b \,\in\, \RR^{1 \times d}} \;\risk_\scv(a,b) = \tfrac{1}{2}(ba^\top - 1)^2,
\end{align}
where the rescaling $\sigma = 1$ has been applied. Since $ba^\top \in \RR$ is a scalar,
the gradient descent dynamics from  Eq.~\eqref{eq:gd_dynamics} reduce to
\begin{align}
    a_{t+1} = a_t + \eta\, L_t\, b_t, \quad
    b_{t+1} = b_t + \eta\, L_t\, a_t,
    \label{eq:sv_gd}
\end{align}
where $L_t := 1 - b_t a_t^\top$. Define the certificate
\begin{align}
    \isv(\delta;\, a, b) := \delta\bigl(\|a\|^2 + \|b\|^2\bigr)
    + \delta^2(1 - ba^\top) - 4,
\end{align}
with $\|a\|^2 = aa^\top$ and $\|b\|^2 = bb^\top$.
A direct expansion of $\isv(\delta;\, a_{t+1}, b_{t+1})$ gives
\begin{align}\label{eq:sv_affine}
    \isv(\delta;\, a_{t+1}, b_{t+1})
    = \bigl(1 - \eta\delta\, L_t + \eta^2 L_t^2\bigr)\,\isv(\delta;\, a_t, b_t)
    + \eta(\delta - \eta)(\delta^2 - 4)\,L_t^2,
\end{align}
which is identical to  Eq.~\eqref{eq:certificate_for_scalar}.

%% file: apxD_LG_plane.tex
\section{Residual--imbalance geometry of the certificate}
\label{sec:LG_plane}

Factor imbalance is a standard observable of implicit regularization in
homogeneous matrix factorization. In particular, \cite{wang2022large} showed
that large learning rates can induce a balancing effect between the factors (see Appendix~\ref{apx:related_work}).
The coordinates below place this quantity alongside the signed optimization
residual. Although the scalar certificate $\isc(\delta;\,a,b)$ already defines
ellipses in the $(a,b)$-plane, this representation makes the relation between
optimization and balancing explicit.

\subsection{Scalar factorization}
\label{subsec:LG_scalar}
Let
\begin{align}
    L_t := 1 - a_t b_t, \quad G_t := b_t^2 - a_t^2,
\end{align}
where $L_t$ is the signed loss residual and $G_t$ is the factorization
imbalance. Indeed, the scalar factorization loss is
$\risk_\sca(a_t,b_t)=\tfrac12 L_t^2$, whereas $L_t$ retains the sign of the
residual. We refer to these coordinates as the $(L,G)$-plane.

The scalar GD dynamics close in these two variables. Using
$(a_t^2 + b_t^2)^2 = G_t^2 + 4(1-L_t)^2$, the GD updates
Eq.~\eqref{eq:model_catalog_scalar_gd} become
\begin{align}
    G_{t+1} &= G_t\bigl(1 - \eta^2 L_t^2\bigr),
    \label{eq:B2-A2_update} \\
    L_{t+1} &= L_t\Bigl(1 - \eta^2 L_t(1 - L_t)
               - \eta\sqrt{4(1 - L_t)^2 + G_t^2}\Bigr).
    \label{eq:1-AB_update}
\end{align}
Each $(L,G)$-state uniquely determines $(a,b)$ up to the global sign symmetry
$(a,b)\leftrightarrow(-a,-b)$. Consequently, an $(L,G)$-trajectory uniquely
determines the original GD trajectory up to the same symmetry.
In particular, Eq.~\eqref{eq:B2-A2_update} shows that the evolution of the
imbalance is modulated directly by the signed loss residual. The certificate,
when expressed in these coordinates, constrains the joint values of these two
quantities along the trajectory.

Define the ellipse function
\begin{align}
    \esc(\delta;\, L, G)
    := L^2 + \frac{G^2}{4 - \delta^2} - \frac{4}{\delta^2}.
    \label{eq:ellipse}
\end{align}

\begin{proposition}\label{prop:ellipse_equivalence}
For $\delta \in (0,2)$, $\;\sgn\bigl(\esc(\delta;\,L,G)\bigr)
= \sgn\bigl(\isc(\delta;\,a,b)\bigr)$.
\end{proposition}
\begin{proof}
In the $(L,G)$ coordinates, $a^2 + b^2 = \sqrt{G^2 + 4(1-L)^2}$, so
$\isc(\delta;\,a,b) = \delta\sqrt{G^2 + 4(1-L)^2} + \delta^2 L - 4$.
If $\isc \le 0$, then $\delta\sqrt{G^2 + 4(1-L)^2} \le 4 - \delta^2 L$, and the right-hand side
must be nonneg\-ative (since the left-hand side is). Squaring both sides preserves
the inequality and gives
$\delta^2\bigl(G^2 + 4(1-L)^2\bigr) \le (4 - \delta^2 L)^2$.
Expanding and simplifying yields
$\delta^2(4-\delta^2)\,L^2 + \delta^2 G^2 \le 4(4-\delta^2)$,
which is $\esc(\delta;\,L,G) \le 0$.
The reverse direction follows by the same chain of equivalences, since the
squaring step is reversible when $4 - \delta^2 L \ge 0$.
\end{proof}

Thus the level sets of $\isc(\delta)$ become axis-aligned ellipses in the
$(L,G)$-plane with semi-axes $2/\delta$ along $L$ and
$\frac{2}{\delta}\sqrt{4-\delta^2}$ along $G$.
In particular, the convergence region $\{\isc(\eta;a_0, b_0) < 0\}$ is equal to
$\{\esc(\eta;\,L_0,G_0) < 0\}$, and the imbalance $G_t$ appears directly as a
coordinate axis.
Since the state parameter $\delta_t$ is strictly increasing
(Section~\ref{subsec:scalar_factorization}), the trajectory resides on a
nested sequence of shrinking ellipses. This is the $(L,G)$-coordinate
representation of the native-coordinate geometry shown in
Figure~\ref{fig:scalar_dynamics_ab}.

\begin{remark}[Shrinking imbalance envelope]\label{rmk:imbalance_envelope}
The $G$-axis semi-radius of the ellipse $\esc(\delta;\,L,G) = 0$ is
\begin{align}
    \mathcal{G}(\delta) = \frac{2}{\delta}\sqrt{4 - \delta^2},
\end{align}
which is strictly decreasing for $\delta \in (0,2)$.
Since $\delta_t$ is strictly increasing along convergent trajectories,
the envelope $\mathcal{G}(\delta_t)$ provides a monotonically
decreasing upper bound on $|G_t|$.
Moreover, the true imbalance eventually reflects this trend:
from  Eq.~\eqref{eq:B2-A2_update}, $|G_{t+1}| \le |G_t|$ holds whenever
$\eta^2 L_t^2 \le 2$, and since $|L_t| \le 2/\delta_t$ on the level set
$\isc(\delta_t) = 0$, this condition is satisfied for all $t \ge T$
once $\delta_T > \sqrt{2}\,\eta$.
\end{remark}

As $\delta \uparrow 2$, the imbalance radius satisfies
\begin{align}
    \mathcal{G}(\delta) = \frac{2}{\delta}\sqrt{4-\delta^2} \;\longrightarrow\; 0.
\end{align}
Thus, the limiting ellipse collapses onto the balanced segment
$\{G = 0,\ |L|\le 1\}$, which corresponds exactly to the scalar terminal set
$K_2^\sca = \{(a,a) : a^2 \le 2\}$ of
Eq.~\eqref{eq:K2_scalar}. Hence $\delta_t\to2$ forces
$G_t\to0$ and leaves only the reduced residual dynamics described in
Appendix~\ref{sec:convergence when delta_* = 2}.

This tightening should not be interpreted as universal convergence to a
balanced minimizer. If $\delta_t\to\delta_\ast<2$, the limiting ellipse retains
a nonzero imbalance radius, and the limiting imbalance may depend on the
initialization \citep{liang2025gradient}.

\subsection{Rank-1 matrix factorization}
\label{subsec:LG_rank1}
For rank-1 factorization with off-signal components, fixed values of the
off-signal energy give a corresponding family of planar slices. Because this
energy changes along the trajectory, these slices do not define a closed
two-dimensional dynamics.
Let $L := 1 - ab$, $G := b^2 - a^2$, and $N := u^2 + v^2$.

\begin{proposition}\label{prop:shifted_ellipse}
For $\delta \in (0,2)$, the sublevel set $\{\ifac(\delta) \le 0\}$ at fixed $N$
is characterized by
\begin{align}
    \Bigl(L - \frac{\delta N}{4-\delta^2}\Bigr)^{\!2}
    + \frac{G^2}{4-\delta^2}
    \;\le\;
    \Bigl(\frac{2}{\delta} - \frac{2N}{4-\delta^2}\Bigr)^{\!2},
    \quad
    \frac{4}{\delta} - \delta L - N \ge 0.
    \label{eq:ellipse_shifted}
\end{align}
\end{proposition}
\begin{proof}
In these coordinates $a^2 + b^2 + N = \sqrt{4(1-L)^2 + G^2} + N$, so
$\ifac \le 0$ reads
$\delta\bigl(\sqrt{4(1-L)^2 + G^2} + N\bigr) + \delta^2 L - 4 \le 0$, i.e.,
$\delta\sqrt{4(1-L)^2 + G^2} \le 4 - \delta^2 L - \delta N$.
The right-hand side is nonnegative (since the left-hand side is), and
squaring yields
\begin{align}
    \delta^2\bigl(4(1-L)^2 + G^2\bigr)
    \le \bigl(4 - \delta^2 L - \delta N\bigr)^2.
\end{align}
Expanding both sides, canceling common terms, and dividing by
$\delta^2 (4-\delta^2) > 0$ gives Eq.~\eqref{eq:ellipse_shifted}.
The constraint $\frac{4}{\delta} - \delta L - N \ge 0$ is the nonnegativity
condition on the right-hand side before squaring.
\end{proof}

That is, if $\frac{4}{\delta} - \delta L - N \ge 0$, the level set $\{\ifac(\delta) \le 0\}$ is a shifted
ellipse in the $(L,G)$-plane centered at
$\bigl(\delta N / (4-\delta^2),\; 0\bigr)$ with semi-axes
$|2/\delta - 2N/(4-\delta^2)|$ along $L$ and
$\sqrt{4-\delta^2}\,|2/\delta - 2N/(4-\delta^2)|$ along~$G$.
As $N \to 0$, the shifted ellipse approaches the unshifted ellipse
$\esc(\delta;\,L,G) = 0$ from Section~\ref{subsec:LG_scalar}. Thus the scalar
geometry is recovered along trajectories for which the off-signal energy
vanishes, including those approaching the terminal set.

%% file: apxE_boundary_inward.tex
\section{Proof of the boundary-inward propositions}
\label{sec:boundary_inward_proofs}

\subsection{Rank-1 matrix factorization (Proposition~\ref{prop:boundary_inward_rank1})}
\label{app:boundary_inward_fac}
We use the normalized rank-1 factorization model described in
Appendix~\ref{subsec:model_catalog_core}. Its reduced GD map is
Eq.~\eqref{eq:model_catalog_unified_gd} with $d=1$, and its minimizer and
stationary sets are $\mathcal M_\fac$ and $\mathcal S_\fac$ in
Eq.~\eqref{eq:model_catalog_fac_sets}.

Define
\begin{align}
L_t &:= 1-a_tb_t, \quad N_t := u_t^2+v_t^2, \quad D_t := (a_t^2+u_t^2)(b_t^2+v_t^2)-a_t^2b_t^2.
\end{align}
Then the iterates of $\gd_\eta^\fac$ satisfy
\begin{align}
\ifac(\delta;\, a_{t+1},b_{t+1},u_{t+1},v_{t+1})
= M_t^\fac(\delta)\,\ifac(\delta;\, a_t,b_t,u_t,v_t) + R_t^\fac(\delta),
\end{align}
where
\begin{align}
M_t^\fac(\delta) &:= 1-\eta\delta L_t + \eta^2(L_t^2+D_t), \label{eq:M_t_def}\\
R_t^\fac(\delta) &:=
-(\eta\delta)^2 u_t^2 v_t^2
+\eta(\delta-\eta)\bigl((\delta^2-4)L_t^2-4D_t+\delta N_t\bigr).
\label{eq:R_t_def}
\end{align}

\begin{proof}[Proof of Proposition~\ref{prop:boundary_inward_rank1}]
Fix $0<\eta<\delta<2$ and a state $(a_t,b_t,u_t,v_t)$ such that
\begin{align*}
    \ifac(\delta;\,a_t,b_t,u_t,v_t)=0.
\end{align*}
If this point lies in $\mathcal S_\fac$, then the gradient step fixes it, and the
quotient--remainder identity gives $R_t^\fac(\delta)=0$. We prove that
$R_t^\fac(\delta)<0$ at every boundary point outside $\mathcal S_\fac$.

Rewrite Eq.~\eqref{eq:R_t_def} as
\begin{align}
    R_t^\fac(\delta)
    &=-(\eta\delta)^2u_t^2v_t^2
      -\eta(\delta-\eta)r_t(\delta),
      \label{eq:R_as_minus}\\
    r_t(\delta)
    &:= (4-\delta^2)L_t^2+4D_t-\delta N_t.
    \label{eq:r_fac_def}
\end{align}
Because $\eta(\delta-\eta)>0$, it remains to prove that $r_t(\delta)\ge0$
on the boundary $\{\ifac(\delta)=0\}$ and to identify when equality can occur.

\paragraph{Reduction to $u_tv_t=0$.}
Fix $(a_t,b_t)$ and $N_t=u_t^2+v_t^2$. The constraint
$\ifac(\delta;\,a_t,b_t,u_t,v_t)=0$ is unchanged when $(u_t,v_t)$ varies
with $N_t$ fixed. Using
$v_t^2=N_t-u_t^2$, we may write
\begin{align}
    D_t
    =a_t^2(N_t-u_t^2)+b_t^2u_t^2+u_t^2(N_t-u_t^2).
\end{align}
Viewed as a function of $u_t^2\in[0,N_t]$, this is a concave quadratic, so its
minimum is attained at $u_t^2=0$ or $u_t^2=N_t$. Since $r_t$ is increasing in
$D_t$, its minimum at fixed $(a_t,b_t,N_t)$ is therefore attained when
$u_tv_t=0$. It suffices to analyze these two endpoints. By the exchange symmetry
$(a_t,u_t)\leftrightarrow(b_t,v_t)$, we first take $u_t=0$. Then
\begin{align}
    r_t(\delta)
    =(4-\delta^2)L_t^2+(4a_t^2-\delta)v_t^2.
    \label{eq:r_reduced_fac}
\end{align}

\paragraph{The case $a_t\ne0$.}
Set $\gamma:=b_t/a_t$. At the endpoint $u_t=0$, the
constraint $\ifac(\delta;\,a_t,b_t,0,v_t)=0$ becomes
\begin{align}
    v_t^2
    =\frac{4-\delta^2}{\delta}
     -a_t^2(1+\gamma^2-\delta\gamma),
    \label{eq:v_sq_boundary}
\end{align}
where
\begin{align}
    1+\gamma^2-\delta\gamma
    =\left(\gamma-\frac{\delta}{2}\right)^2
     +1-\frac{\delta^2}{4}
    >0.
\end{align}
Since $v_t^2\ge0$, we have
\begin{align}
    0<a_t^2\le
    \frac{4-\delta^2}
         {\delta(1+\gamma^2-\delta\gamma)}.
    \label{eq:a_sq_upper_fac}
\end{align}
Substituting Eq.~\eqref{eq:v_sq_boundary} and $L_t=1-a_t^2\gamma$ into
Eq.~\eqref{eq:r_reduced_fac} and rearranging yields
\begin{align}
    r_t(\delta)
    =a_t^2\left[
        \frac{(\delta\gamma^2-4\gamma+\delta)^2}
             {\delta(1+\gamma^2-\delta\gamma)}
        +\left(
            \frac{4-\delta^2}
                 {\delta(1+\gamma^2-\delta\gamma)}-a_t^2
         \right)(\delta\gamma-2)^2
      \right]
    \ge0.
    \label{eq:r_fac_square_form}
\end{align}
Here both terms in brackets are nonnegative by
Eq.~\eqref{eq:a_sq_upper_fac}. If $r_t(\delta)=0$, then
$\delta\gamma^2-4\gamma+\delta=0$. This equality is incompatible with
$\delta\gamma-2=0$ when $0<\delta<2$: indeed, the two equalities would give
$\gamma=2/\delta$ and then force $\delta=2$. Thus the second term can vanish
only if equality holds in Eq.~\eqref{eq:a_sq_upper_fac}. Hence
Eq.~\eqref{eq:v_sq_boundary} gives $v_t=0$. Since
$\delta\gamma^2-4\gamma+\delta=0$, the identity
\begin{align}
    \delta(1+\gamma^2-\delta\gamma)
    -(4-\delta^2)\gamma
    =\delta\gamma^2-4\gamma+\delta
\end{align}
gives $\delta(1+\gamma^2-\delta\gamma)=(4-\delta^2)\gamma$. Combining this with
equality in Eq.~\eqref{eq:a_sq_upper_fac} yields $a_t^2=1/\gamma$, and hence
$a_t^2\gamma=1$, or equivalently $L_t=0$.
Together with $u_t=0$, these identities show that the point lies in
$\mathcal M_\fac$.

\paragraph{The case $a_t=0$.}
Here $L_t=1$, and $\ifac(\delta;\,0,b_t,0,v_t)=0$ reads
\begin{align}
    \delta(b_t^2+v_t^2)+\delta^2-4=0.
\end{align}
Consequently,
\begin{align}
    r_t(\delta)
    =4-\delta^2-\delta v_t^2
    =\delta b_t^2\ge0.
\end{align}
Equality holds only if $b_t=0$, in which case
$(a_t,b_t,u_t,v_t)=(0,0,0,v_t)\in\mathcal S_\fac$.

The other endpoint, $v_t=0$, follows by the exchange symmetry. We have thus
proved $r_t(\delta)\ge0$ on $\{\ifac(\delta)=0\}$. If $u_tv_t\ne0$, the first
term in Eq.~\eqref{eq:R_as_minus} is strictly negative. If $u_tv_t=0$, the
endpoint analysis shows that $r_t(\delta)=0$ can occur only on
$\mathcal S_\fac$. Therefore every boundary point outside $\mathcal S_\fac$ satisfies
$R_t^\fac(\delta)<0$, as claimed.
\end{proof}

\subsection{Isotropic-signal rank-1 approximation
(Proposition~\ref{prop:boundary_inward_rank1_approx})}
\label{app:boundary_inward_apx}
We use the normalized isotropic-signal rank-1 approximation model described in
Appendix~\ref{subsec:model_catalog_core}, with $d=n-1\ge2$. Write
\begin{align}
    A_t=\begin{pmatrix}a_t\\u_t\end{pmatrix},
    \qquad
    B_t=\begin{pmatrix}b_t\\v_t\end{pmatrix},
\end{align}
where $a_t,b_t\in\RR^{n-1}$ and $u_t,v_t\in\RR$. Its reduced GD map is the
vector form of Eq.~\eqref{eq:model_catalog_unified_gd}, and its minimizer and
stationary sets are $\mathcal M_\apx$ and $\mathcal S_\apx$ in
Eq.~\eqref{eq:model_catalog_apx_sets}.
Define
\begin{align}
    L_t := 1-\langle a_t,b_t\rangle,
    \quad
    N_t := u_t^2+v_t^2,
\end{align}
and
\begin{align}
    D_t
    := \|A_t\|^2\|B_t\|^2-\langle a_t,b_t\rangle^2,\quad
    D_t^S
    := \|a_t\|^2\|b_t\|^2-\langle a_t,b_t\rangle^2,\quad
    D_t^N
    := u_t^2v_t^2.
\end{align}
Relative to the factorization model, $D_t^S$ is the additional term: it
measures the misalignment of the signal vectors $a_t$ and $b_t$.
By the Cauchy--Schwarz inequality, $D_t^S\ge0$, while $D_t^N\ge0$ is immediate.
Then the certificate satisfies the quotient--remainder decomposition
\begin{align}
    \iapx(\delta;\,A_{t+1},B_{t+1})
    =
    M_t^\apx(\delta)\iapx(\delta;\,A_t,B_t)+R_t^\apx(\delta),
\end{align}
where
\begin{align}
    M_t^\apx(\delta)
    &:=
    1-\eta\delta L_t+\eta^2(L_t^2+D_t),\\
    R_t^\apx(\delta)
    &:=
    \eta(\delta-\eta)\bigl((\delta^2-4)L_t^2-4D_t\bigr)
    +\eta^2\delta^2(D_t^S-D_t^N)
    +\eta\delta(\delta-\eta)N_t.
\end{align}

\begin{proof}[Proof of Proposition~\ref{prop:boundary_inward_rank1_approx}]
Fix $0<\eta<\delta<2$ with
\begin{align}
    q_\eta(\delta):=\eta\delta^2-4\delta+4\eta<0,
\end{align}
and consider a state $(a_t,b_t,u_t,v_t)$ such that
\begin{align*}
    \iapx(\delta;\,A_t,B_t)=0.
\end{align*}
If this point lies in $\mathcal S_\apx$, then the gradient step fixes it, and the
quotient--remainder identity gives $R_t^\apx(\delta)=0$. We prove that
$R_t^\apx(\delta)<0$ at every boundary point outside $\mathcal S_\apx$.

Using
\begin{align}
    D_t
    =D_t^S+\|a_t\|^2v_t^2+\|b_t\|^2u_t^2+u_t^2v_t^2,
\end{align}
we may rewrite the remainder as
\begin{align}
    R_t^\apx(\delta)
    &=-\eta(\delta-\eta)r_t(\delta)
      -\eta^2\delta^2D_t^N
      +\eta q_\eta(\delta)D_t^S,
      \label{eq:R_t_apx_regroup}\\
    r_t(\delta)
    &:=(4-\delta^2)L_t^2
      +4\bigl(\|a_t\|^2v_t^2+\|b_t\|^2u_t^2+u_t^2v_t^2\bigr)
      -\delta N_t.
      \label{eq:r_t_apx_regroup}
\end{align}
Since $D_t^N,D_t^S\ge0$ and $q_\eta(\delta)<0$, the last two terms in
Eq.~\eqref{eq:R_t_apx_regroup} are nonpositive, and their sum vanishes only if
$D_t^N=D_t^S=0$. It remains to prove $r_t(\delta)\ge0$ on
the boundary $\{\iapx(\delta)=0\}$ and to track its equality case.

\paragraph{Reduction to $u_tv_t=0$.}
Fix $(a_t,b_t)$ and $N_t=u_t^2+v_t^2$. As before, the constraint
$\iapx(\delta;\,A_t,B_t)=0$ is unchanged when $(u_t,v_t)$ varies with
$N_t$ fixed. Using
$v_t^2=N_t-u_t^2$, the off-signal expression in $r_t$ becomes
\begin{align}
    \|a_t\|^2(N_t-u_t^2)
    +\|b_t\|^2u_t^2
    +u_t^2(N_t-u_t^2).
\end{align}
Viewed as a function of $u_t^2\in[0,N_t]$, this is a concave quadratic, so its
minimum is attained at $u_t^2=0$ or $u_t^2=N_t$. Thus the minimum of $r_t$ at
fixed $(a_t,b_t,N_t)$ occurs when $u_tv_t=0$. By the exchange symmetry, it
suffices to analyze $u_t=0$, in which case
\begin{align}
    r_t(\delta)
    =(4-\delta^2)L_t^2+(4\|a_t\|^2-\delta)v_t^2.
    \label{eq:r_reduced_apx}
\end{align}

\paragraph{The case $4\|a_t\|^2-\delta>0$.}
Both terms in Eq.~\eqref{eq:r_reduced_apx} are nonnegative. If
$R_t^\apx(\delta)=0$, all three nonpositive terms in
Eq.~\eqref{eq:R_t_apx_regroup} must vanish. In particular,
$r_t(\delta)=D_t^S=0$, so $L_t=v_t=0$ and $a_t,b_t$ are linearly dependent.
Together with $u_t=0$, these conditions place the point in $\mathcal M_\apx$.

\paragraph{The case $4\|a_t\|^2-\delta=0$.}
Here $r_t(\delta)=(4-\delta^2)L_t^2\ge0$. Equality in
Eq.~\eqref{eq:R_t_apx_regroup} would force $L_t=D_t^S=0$. Thus
$\langle a_t,b_t\rangle=1$, and equality in the Cauchy--Schwarz inequality,
together with $\|a_t\|^2=\delta/4$, gives $\|b_t\|^2=4/\delta$.
The identity $\iapx(\delta;\,A_t,B_t)=0$ would then imply
\begin{align}
    v_t^2
    =\frac{4-\delta^2}{\delta}
     -\frac{\delta}{4}-\frac{4}{\delta}+\delta
    =-\frac{\delta}{4},
\end{align}
a contradiction. Hence equality in the remainder cannot occur in this case.

\paragraph{The case $4\|a_t\|^2-\delta<0$ and $\|a_t\|>0$.}
Set
\begin{align}
    \gamma:=\frac{\langle a_t,b_t\rangle}{\|a_t\|^2}.
\end{align}
The Cauchy--Schwarz inequality gives
$\|b_t\|^2/\|a_t\|^2\ge\gamma^2$. Together with
$\iapx(\delta;\,A_t,B_t)=0$, this gives
\begin{align}
    v_t^2
    &=\frac{4-\delta^2}{\delta}
      -\|a_t\|^2
       \left(1+\frac{\|b_t\|^2}{\|a_t\|^2}-\delta\gamma\right)
      \notag\\
    &\le \frac{4-\delta^2}{\delta}
       -\|a_t\|^2(1+\gamma^2-\delta\gamma).
    \label{eq:v_sq_upper_apx}
\end{align}
The denominator below is positive because
\begin{align}
    1+\gamma^2-\delta\gamma
    =\left(\gamma-\frac{\delta}{2}\right)^2
     +1-\frac{\delta^2}{4}
    >0.
\end{align}
Together with $v_t^2\ge0$, this yields
\begin{align}
    0<\|a_t\|^2\le
    \frac{4-\delta^2}
         {\delta(1+\gamma^2-\delta\gamma)}.
    \label{eq:a_norm_upper_apx}
\end{align}
Because $4\|a_t\|^2-\delta<0$, multiplying the upper bound
Eq.~\eqref{eq:v_sq_upper_apx} by this coefficient reverses the inequality.
Using $L_t=1-\|a_t\|^2\gamma$ in Eq.~\eqref{eq:r_reduced_apx}, we obtain
\begin{align}
    r_t(\delta)
    \ge \|a_t\|^2\left[
        \frac{(\delta\gamma^2-4\gamma+\delta)^2}
             {\delta(1+\gamma^2-\delta\gamma)}
        +\left(
            \frac{4-\delta^2}
                 {\delta(1+\gamma^2-\delta\gamma)}-\|a_t\|^2
         \right)(\delta\gamma-2)^2
      \right]
    \ge0,
    \label{eq:r_apx_square_bound}
\end{align}
where the last inequality follows from
Eq.~\eqref{eq:a_norm_upper_apx}.

Suppose now that $R_t^\apx(\delta)=0$. Then
Eq.~\eqref{eq:R_t_apx_regroup} forces $r_t(\delta)=D_t^S=0$.
Equality in Eq.~\eqref{eq:r_apx_square_bound} first gives
$\delta\gamma^2-4\gamma+\delta=0$. As in the factorization case, this cannot
hold together with $\delta\gamma-2=0$, since the two equalities would force
$\delta=2$. Thus equality must also hold in Eq.~\eqref{eq:a_norm_upper_apx}.
The identity
\begin{align}
    \delta(1+\gamma^2-\delta\gamma)
    -(4-\delta^2)\gamma
    =\delta\gamma^2-4\gamma+\delta
\end{align}
and $\delta\gamma^2-4\gamma+\delta=0$ give
$\delta(1+\gamma^2-\delta\gamma)=(4-\delta^2)\gamma$. Combining this with
equality in Eq.~\eqref{eq:a_norm_upper_apx} yields
$\|a_t\|^2=1/\gamma$, and hence $\|a_t\|^2\gamma=1$. Since $D_t^S=0$,
equality holds in Cauchy--Schwarz, so Eq.~\eqref{eq:v_sq_upper_apx} is an
equality; together with
Eq.~\eqref{eq:a_norm_upper_apx}, this gives $v_t=0$. Thus $u_t=v_t=0$,
$D_t^S=0$, and $\langle a_t,b_t\rangle=1$, which places the point in
$\mathcal M_\apx$.

\paragraph{The case $\|a_t\|=0$.}
Here $L_t=1$, and $\iapx(\delta;\,A_t,B_t)=0$ gives
\begin{align}
    \delta(\|b_t\|^2+v_t^2)+\delta^2-4=0.
\end{align}
Consequently,
\begin{align}
    r_t(\delta)
    =4-\delta^2-\delta v_t^2
    =\delta\|b_t\|^2\ge0.
\end{align}
Equality holds only if $b_t=0$, in which case
$(a_t,b_t,u_t,v_t)=(0,0,0,v_t)\in\mathcal S_\apx$.

The other endpoint, $v_t=0$, follows by the exchange symmetry. We have proved
$r_t(\delta)\ge0$ on $\{\iapx(\delta)=0\}$. All three terms in
Eq.~\eqref{eq:R_t_apx_regroup} are therefore nonpositive. If the remainder
vanishes, then $D_t^N=D_t^S=0$, so one of the two endpoint analyses applies,
and the case analysis above places the point in $\mathcal S_\apx$. Hence every
boundary point outside $\mathcal S_\apx$ satisfies $R_t^\apx(\delta)<0$, as claimed.
\end{proof}

%% file: apxF_delta_lt_2.tex

\section{Proof of the abstract convergence theorem and applications}
\label{apdx:abstract_convergence}

In this appendix we prove the abstract convergence theorem
(Theorem~\ref{thm:abstract_state_dependent_convergence}) in five steps directly
from the state-dependent Lyapunov axioms of
Section~\ref{sec:state_dependent_lyapunov},
without reference to any concrete model. We then verify that the three models
of Section~\ref{sec:rank-1 convergence} satisfy those axioms, and apply the
theorem to recover their convergence statements. In the terminal case
$\delta_\ast=2$ the theorem gives convergence to the terminal set and nothing
more; the asymptotics there are determined in
Appendices~\ref{sec:convergence when delta_* = 2}
and~\ref{sec:terminal_inheritance}.

\subsection{Statement of the theorem and auxiliary lemmas}
\label{subsec:theorem5_and_auxiliary_lemmas}

Throughout this appendix $\risk\in C^2(\RR^n)$ has stationary set
$\mathcal S:=\{x\in\RR^n:\nabla\risk(x)=0\}$, and for a nonterminal state
$\delta$ we write
\begin{align}
    \mathcal S_\delta
    :=
    \mathcal S\cap\partial K_\delta
    =
    \mathcal S\cap\{x: I(\delta;\,x)=0\}
    \label{eq:stationary_on_boundary}
\end{align}
for the stationary points on the boundary of $K_\delta$.

\begin{theorem}\label{thm:abstract_state_dependent_convergence}
Fix $\eta>0$, assume $\mathcal S$ has measure zero and that the GD map
$\gd_\eta$ is a submersion almost everywhere, and let $I(\delta;x)$ be a
state-dependent quadratic certificate satisfying
Axioms~\ref{cond:P_is_pd}--\ref{cond:stationarity}.
Then, for almost every $x_0\in\operatorname{int}K_{\delta_{\mathrm{th}}}$,
the sequence $(\delta_t)_{t\ge0}$ of states along the GD trajectory $(x_t)_{t\ge0}$
is strictly increasing, and the limit
$\delta_\ast:=\lim_{t\to\infty}\delta_t\le\overline\delta$ exists.
Moreover:

(1) If $\delta_\ast<\overline\delta$, then
$\dist{x_t}{\mathcal S_{\delta_\ast}}\to0$; in particular,
$\dist{x_t}{\mathcal S}\to0$.

(2) If $\delta_\ast=\overline\delta$, then $\dist{x_t}{K_{\overline\delta}}\to 0$.

(3) Assume that $\mathcal S_\delta$ is finite for every
$\delta\in(\delta_{\mathrm{th}},\overline\delta)$. Then every such $x_0$ for
which case (1) occurs has a trajectory converging to a single stationary point.

(4) Let $\mathcal U\subseteq K_{\delta_{\mathrm{th}}}\cap\mathcal S$ be
closed. If for every $p\in\mathcal U$ the Jacobian $D\gd_\eta(p)$ has an
eigenvalue of modulus strictly greater than one, then the set of
initializations $x_0\in\operatorname{int}K_{\delta_{\mathrm{th}}}$ whose
trajectory converges to a point of $\mathcal U$ has measure zero.

(5) If the finiteness condition in (3) holds and every point of
$K_{\delta_{\mathrm{th}}}\cap\mathcal S$ satisfies the spectral condition in
(4), then, for almost every
$x_0\in\operatorname{int}K_{\delta_{\mathrm{th}}}$, case (1) does not occur;
hence $\delta_\ast=\overline\delta$ and (2) applies.
\end{theorem}

Under the theorem's standing assumptions, conclusions (1) and (2) are the two
halves of a dichotomy and require only the axioms; (3)--(5) carry extra
hypotheses and are invoked separately in
Appendix~\ref{subsec:applications_of_theorem5}. Conclusion (4) is independent
of (3), requiring neither finiteness nor pointwise convergence, and is applied
both to compact subsets of a positive-dimensional stationary set
(Appendix~\ref{subsec:exclude_nonmin_stationary_rank1}) and to compact sets of
unstable global minimizers
(Appendix~\ref{subsec:minimizer_instability_eta_gt_one}).

We first record the elementary fact behind Axiom~\ref{cond:P_is_pd}. It says when
a quadratic sublevel set is a nondegenerate ellipsoid.

\begin{lemma}\label{lem:compact_quadratic_sublevel}
Let $P=P^\top\in\RR^{n\times n}$ and
\begin{align}
    K = \{x \in \RR^n : x^\top P x + 2\,q^\top x + r \le 0\}.
\end{align}
Then $K$ is nonempty and compact if and only if $P\succ 0$ and
$r - q^\top P^{-1}q \le 0$.
\end{lemma}

\begin{proof}
If $P$ has a negative eigenvalue, the corresponding direction makes the
quadratic term tend to $-\infty$, producing an unbounded ray in $K$.
If $P$ is positive semidefinite but singular, pick $v\in\ker P\setminus\{0\}$.
When $q^\top v=0$, translating any point of $K$ along $v$ stays in $K$,
producing a full line; when $q^\top v\ne 0$, choosing the sign of $t$ to
make $2tq^\top v<0$ sends the affine term to $-\infty$, producing an
unbounded ray. In either case $K$ is unbounded. Hence nonempty
compactness forces $P\succ 0$. Completing the square gives
$x^\top Px + 2q^\top x + r = (x+P^{-1}q)^\top P(x+P^{-1}q) + r - q^\top P^{-1}q$,
from which the remaining equivalence is immediate.
\end{proof}

Three auxiliary results are used in the proof of
Theorem~\ref{thm:abstract_state_dependent_convergence}.
The first lemma converts setwise convergence into pointwise convergence
whenever a candidate limit is isolated.

\begin{lemma}\label{lem:finite_limit_set}
Let $(x_t)_{t\ge 0}\subset\RR^d$ be bounded with $\|x_{t+1}-x_t\|\to 0$. If
the set of limit points of $(x_t)_{t\ge 0}$ has an isolated point
$p_\ast$, then $x_t\to p_\ast$.
\end{lemma}
\begin{proof}
Choose $r>0$ such that $p_\ast$ is the only limit point in the closed ball
$\overline B(p_\ast,3r)$. Suppose the limit set is not the singleton
$\{p_\ast\}$, and let $q$ be another limit point, so that
$\|q-p_\ast\|>3r$. Consequently,
$B(q,r)\cap\overline B(p_\ast,2r)=\varnothing$. The sequence visits
$B(p_\ast,r)$ and $B(q,r)$ infinitely often. Once
$\|x_{t+1}-x_t\|<r$, every passage between these balls must visit the compact
annulus $\overline B(p_\ast,2r)\setminus B(p_\ast,r)$; a single step cannot
cross it. There are infinitely many such passages, so compactness yields a
limit point in the annulus. This contradicts the choice of $r$. Thus $p_\ast$
is the unique limit point of the bounded sequence, and $x_t\to p_\ast$.
\end{proof}

The next lemma is the elementary compactness step used in Steps 2 and 3 of the
proof below, and once more in the rank-1 approximation.

\begin{lemma}\label{lem:setwise_from_accumulation}
Let $(x_t)_{t\ge0}$ be a sequence contained in a compact set
$K\subseteq\RR^d$, and let $C\subseteq\RR^d$ be closed. If every accumulation
point of $(x_t)_{t\ge0}$ belongs to $C$, then $\dist{x_t}{C}\to0$.
\end{lemma}
\begin{proof}
If not, there are $\varepsilon>0$ and a subsequence with
$\dist{x_{t_j}}{C}\ge\varepsilon$ for all $j\ge 0$. That subsequence lies in the
compact set $K$, so it has a convergent subsequence, whose limit
$\bar x$ is a limit point of $(x_t)_{t\ge0}$ and hence lies in $C$.
Continuity of $\dist{\cdot}{C}$ then gives
$0=\dist{\bar x}{C}\ge\varepsilon$, a contradiction.
\end{proof}

The last result is the local measure-zero statement behind the exclusion of
unstable limits. At points where the update map is a local
diffeomorphism, it is the local stable-set consequence of the center-stable
manifold theorem \citep[Theorem~III.7]{shub2013global}. The maps we apply it to
may fail to be local diffeomorphisms at finitely many points of the relevant
stationary set, so we state and prove it through the stable-manifold theorem
for pseudo-hyperbolic endomorphisms
\citep[Theorem~5.1]{hirsch1977invariant}, which does not require one.

\begin{corollary}[Measure-zero local nonescaping set near an unstable fixed point]
\label{cor:local_nonescaping_unstable_fixed_point}
Let $g:U\to \mathbb R^d$ be a $C^1$ map on an open set
$U\subseteq \mathbb R^d$, and let $p\in U$ be a fixed point of $g$.
Suppose that $Dg(p)$ has an eigenvalue $\lambda$ with $|\lambda|>1$.
Then there exists $r>0$ such that $B(p,r)\subset U$ and the local
nonescaping set
\begin{align}
    W_p
    :=
    \left\{
        x\in B(p,r):
        g^t(x)\in B(p,r)\ \text{for all }t\ge 0
    \right\}
\end{align}
has measure zero.
\end{corollary}

\begin{proof}
After translating coordinates, we may assume without loss of generality that
$p=0$. Set $T:=Dg(0)$. Choose $\rho$ such that $1<\rho<|\lambda|$ and such
that no eigenvalue of $T$ has modulus exactly $\rho$. Let
$\mathbb R^d=E_1\oplus E_2$ be the real $T$-invariant spectral splitting
associated with eigenvalues outside and inside the circle $|\mu|=\rho$,
respectively. Then $T$ is $\rho$-pseudo-hyperbolic, and
$E_1\neq\{0\}$ because $T$ has the eigenvalue $\lambda$.

By the renorming observation preceding Theorem~5.1 in
\citet[p.~53]{hirsch1977invariant}, we may choose norms $|\cdot|_1$ on $E_1$
and $|\cdot|_2$ on $E_2$ such that
\begin{align}
    m_1(T|_{E_1})
    :={}
    \inf_{x\in E_1\setminus\{0\}}
    \frac{|Tx|_1}{|x|_1}
    >\rho,\quad
    \|T|_{E_2}\|_{2,\mathrm{op}}
    :={}
    \sup_{y\in E_2\setminus\{0\}}
    \frac{|Ty|_2}{|y|_2}
    <\rho,
\end{align}
and equip $\mathbb R^d=E_1\oplus E_2$ with the adapted norm
\begin{align}
    \|(x,y)\|_{\mathrm{ad}}
    :=
    \max\{|x|_1,|y|_2\}.
\end{align}
Since $\mathbb R^d$ is finite-dimensional, $\|\cdot\|_{\mathrm{ad}}$ is
equivalent to the Euclidean norm. Hence Euclidean boundedness implies
$\|\cdot\|_{\mathrm{ad}}$-boundedness, the local little-$o$ estimates can be
expressed in either norm, and the measure zero conclusion is unchanged.

Let $\chi\in C^\infty(\mathbb R^d)$ be a smooth bump function such that $\chi=1$ on $B(0,1)$
and $\operatorname{supp}\chi\subset B(0,2)$, and set
$\chi_r(x):=\chi(x/r)$ (here, the balls are with respect to the Euclidean norm). Choose $r>0$ such that
$\overline{B(0,2r)}\subset U$. Writing $R(x):=g(x)-Tx$ on $U$, define
\begin{align}
    h_r(x)
    :={}
    \begin{cases}
        \chi_r(x)R(x), & x\in U,\\
        0, & x\notin U,
    \end{cases}
    \quad
    f_r(x):=Tx+h_r(x).
\end{align}
Because $\operatorname{supp}\chi_r\subset B(0,2r)\Subset U$, the extension
$h_r$ is globally $C^1$, and hence so is $f_r:\mathbb R^d\to\mathbb R^d$.
Moreover, $f_r(0)=0$ and $f_r=g$ on $B(0,r)$.

Since $g(0)=0$ and $Dg(0)=T$, we have
$\|R(x)\|_{\mathrm{ad}}=o(\|x\|_{\mathrm{ad}})$ and
\begin{align}
    \|DR(x)\|_{\mathrm{ad},\mathrm{op}}
    =
    \|Dg(x)-T\|_{\mathrm{ad},\mathrm{op}}
    \longrightarrow0
    \qquad\text{as }x\to0.
\end{align}
Since $\chi\in C^\infty(\mathbb R^d)$ with compact support, its derivative is bounded in the
dual adapted norm. Thus, with
\begin{align}
    C_\chi
    :=
    \sup_{z\in\mathbb R^d}\|D\chi(z)\|_{\mathrm{ad},*}
    <\infty,
\end{align}
the scaling identity $D\chi_r(x)=r^{-1}D\chi(x/r)$ gives
\begin{align}
    \sup_{x\in\mathbb R^d}
    \|D\chi_r(x)\|_{\mathrm{ad},*}
    \le
    \frac{C_\chi}{r},
\end{align}
where $\|\cdot\|_{\mathrm{ad},*}$ is the dual norm. Combining this bound with
the identity
\begin{align}
    D(\chi_rR) (x)[y]= D\chi_r(x)[y]R(x)  + \chi_rDR(x)[y]
\end{align}
gives
\begin{align}
    L_{\mathrm{ad}}(f_r-T)
    &=L_{\mathrm{ad}}(h_r) \notag\\
    &\le
    \frac{C_\chi}{r}
        \sup_{x \in B(0,2r)}\|g(x)-Tx\|_{\mathrm{ad}}
    +\sup_{x \in B(0,2r)}\|Dg(x)-T\|_{\mathrm{ad},\mathrm{op}}
    \notag\\
    &\longrightarrow 0
    \quad\text{as }r\downarrow0.
\end{align}
Here $L_{\mathrm{ad}}$ denotes the Lipschitz constant induced by
$\|\cdot\|_{\mathrm{ad}}$.
Thus, after decreasing $r$ if necessary, $f_r$ satisfies the small-Lipschitz
hypothesis of \citet[Theorem~5.1]{hirsch1977invariant}.

Let $W_2$ denote the set given by Theorem~5.1 of
\citet{hirsch1977invariant}. By that theorem, $W_2$ is the graph of a $C^1$
map from $E_2$ to $E_1$. Since $E_1\neq\{0\}$, this graph has dimension
$\dim E_2<d$. Therefore $W_2$ has measure zero in $\mathbb R^d$.

It remains to compare $W_p$ with $W_2$. Let $x\in W_p$. Then
$g^t(x)\in B(0,r)$ for all $t\ge 0$. Since $f_r=g$ on $B(0,r)$, induction gives
$f_r^t(x)=g^t(x)$ for every $t\ge 0$. Hence the forward orbit of $x$ under
$f_r$ remains bounded in the Euclidean norm and therefore also in the adapted
norm. Since $\rho>1$, this implies
$\|f_r^t(x)\|_{\mathrm{ad}}/\rho^t\to 0$. By the characterization of $W_2$ in
\citet[Theorem~5.1]{hirsch1977invariant}, we obtain $x\in W_2$. Thus
$W_p\subseteq W_2$.
Since $W_2$ has measure zero, the subset $W_p$ also has measure zero.
\end{proof}

\subsection{Proof of the abstract convergence theorem}
\label{sec:extension_state_dependent}

Throughout the proof we fix the notations and assumptions of the theorem and write
$\gd_\eta$ for the gradient descent map. By definition, $I$ is of class $C^1$
on $(\underline\delta,\overline\delta)\times\RR^n$.

\textbf{Step 1. Monotone state trajectory.}
By the measure-zero assumption on $\mathcal S$ and
Axiom~\ref{cond:end_point_measure_zero}, the set
$\mathcal S\cup K_{\overline\delta}$ has measure zero. Since
$\gd_\eta$ is $C^1$ and a submersion almost everywhere,
Corollary~\ref{cor:gd_preimage_null} shows that
\begin{align}
    E_0
    :=
    \bigcup_{T=0}^\infty
    \gd_\eta^{-T}\bigl(\mathcal S\cup K_{\overline\delta}\bigr)
    \label{eq:abstract_finite_time_exceptional}
\end{align}
has measure zero. Fix
$x_0\in\operatorname{int}K_{\delta_{\mathrm{th}}}\setminus E_0$; then no
finite iterate $x_t$ lies in $\mathcal S$ or in $K_{\overline\delta}$.

Since $x_0\in\operatorname{int}K_{\delta_{\mathrm{th}}}$ we have
$I(\delta_{\mathrm{th}};\,x_0)<0$, and $x_0\notin K_{\overline\delta}$, so
Axiom~\ref{cond:level_set_as_state} assigns to $x_0$ the unique state
$\delta_0\in(\underline\delta,\overline\delta)$ with $I(\delta_0;\,x_0)=0$.
From $I(\delta_{\mathrm{th}};\,x_0)<0=I(\delta_0;\,x_0)$ and
Axiom~\ref{cond:nesting} we get $\delta_0>\delta_{\mathrm{th}}$.

We next check that every iterate has a state, that
$\delta_t\in(\delta_{\mathrm{th}},\overline\delta)$, and that
$(\delta_t)_{t\ge0}$ is nondecreasing. The base case is the previous paragraph,
which gives $\delta_0\in(\delta_{\mathrm{th}},\overline\delta)$. Inductively,
suppose $x_t$ has a state $\delta_t\in(\delta_{\mathrm{th}},\overline\delta)$,
so that $I(\delta_t;\,x_t)=0$. Axiom~\ref{cond:monotonicity} evaluated
at $\delta=\delta_t$ gives $I(\delta_t;\,x_{t+1})\le 0$, that is,
$x_{t+1}\in K_{\delta_t}$. Since $x_{t+1}\notin K_{\overline\delta}$,
Axiom~\ref{cond:level_set_as_state} assigns to it a state
$\delta_{t+1}\in(\underline\delta,\overline\delta)$, and
Axiom~\ref{cond:nesting} gives
$\delta_{t+1}\ge\delta_t>\delta_{\mathrm{th}}$; the induction closes. Each
$x_t$ therefore lies on $\partial K_{\delta_t}$, with
$\delta_{\mathrm{th}}<\delta_0\le\delta_t<\overline\delta$ for all $t\ge0$.

Monotonicity and the bound $\delta_t<\overline\delta$ already give a limit
$\delta_\ast:=\lim_{t\to\infty}\delta_t\le\overline\delta$. The monotonicity is
moreover strict: no iterate is stationary by the choice of $x_0$, so
Axiom~\ref{cond:stationarity} gives
$x_{t+1}=\gd_\eta(x_t)\notin\partial K_{\delta_t}$, while
Axiom~\ref{cond:monotonicity} places $x_{t+1}$ in $K_{\delta_t}$. Hence
$I(\delta_t;\,x_{t+1})<0$; since
$I(\delta_{t+1};\,x_{t+1})=0$, this rules out $\delta_{t+1}=\delta_t$. Hence
$(\delta_t)_{t\ge0}$ is strictly increasing, which proves the assertions of the
theorem preceding~(1).

Finally $x_t\in K_{\delta_t}\subseteq K_{\delta_{\mathrm{th}}}$ by
Axiom~\ref{cond:nesting}, and $K_{\delta_{\mathrm{th}}}$ is compact by
Axiom~\ref{cond:P_is_pd}, so the trajectory is confined to a fixed compact set
for all time.

\textbf{Step 2. The nonterminal case: conclusion (1).}
We write $R_t$ for the signed depth of the inward step from the current level
set:
\begin{align}
    R_t := I(\delta_t;\,x_{t+1}) - I(\delta_t;x_t) =  I(\delta_t;\,x_{t+1}) < 0.
    \label{eq:abstract_strict_inward}
\end{align}
In the concrete models of Section~\ref{sec:rank-1 convergence} this is the
remainder of the quotient--remainder identity, evaluated at the current state
$\delta=\delta_t$.

Assume for the remainder of this step that $\delta_\ast<\overline\delta$. From
$I(\delta_{t+1};\,x_{t+1})=0$, the mean value theorem applied to
$\delta\mapsto I(\delta;\,x_{t+1})$ on $[\delta_t,\delta_{t+1}]$ produces
$\widetilde\delta_t\in(\delta_t,\delta_{t+1})$ with
\begin{align}
    -R_t=\partial_\delta I(\widetilde\delta_t;\,x_{t+1})
    \, (\delta_{t+1}-\delta_t).
\end{align}
Choose $\delta^\sharp\in(\delta_\ast,\overline\delta)$. Since $I$ is $C^1$ on
the compact rectangle
$[\delta_{\mathrm{th}},\delta^\sharp]\times K_{\delta_{\mathrm{th}}}$ and
$x_{t+1}\in K_{\delta_{\mathrm{th}}}$ by Step 1, there is $M>0$ with
$|\partial_\delta I|\le M$ on that rectangle. Moreover,
$\widetilde\delta_t<\delta^\sharp$ for every $t$, so
\begin{align}
        0<-R_t\le M\,(\delta_{t+1}-\delta_t).
    \label{eq:abstract_root_increment}
\end{align}
Summing Eq.~\eqref{eq:abstract_root_increment} telescopes the right-hand side,
and $\delta_\ast<\overline\delta$ makes the total finite:
\begin{align}
        \sum_{t=0}^\infty (-R_t)\le M(\delta_\ast-\delta_0)<\infty,
    \qquad\text{hence}\qquad
    R_t\to 0 .
    \label{eq:abstract_R_to_zero}
\end{align}

We now identify the accumulation points. Let $x_{t_j}\to\bar x$. 
Joint continuity of $I$ and $\delta_{t_j}\to\delta_\ast$
give
\begin{align}
    I(\delta_\ast;\,\bar x)
    =\lim_{j\to\infty}I(\delta_{t_j};\,x_{t_j})
    =0,
    \label{eq:abstract_limit_on_level}
\end{align}
so $\bar x$ lies on the limiting level set. The map
$(\delta,x)\mapsto I(\delta;\,\gd_\eta(x))$ is continuous, and
$R_t=I(\delta_t;\,\gd_\eta(x_t))$, so Eq.~\eqref{eq:abstract_R_to_zero} gives
\begin{align}
    I(\delta_\ast;\,\gd_\eta(\bar x))
    =\lim_{j\to\infty}R_{t_j}
    =0.
    \label{eq:abstract_limit_remainder_zero}
\end{align}
Because $\delta_{\mathrm{th}}<\delta_0\le\delta_\ast<\overline\delta$,
Axioms~\ref{cond:monotonicity} and~\ref{cond:stationarity} apply at
$(\delta_\ast,\bar x)$; Eqs.~\eqref{eq:abstract_limit_on_level}
and~\eqref{eq:abstract_limit_remainder_zero} say that both $\bar x$ and
$\gd_\eta(\bar x)$ lie on $\partial K_{\delta_\ast}$, so
Axiom~\ref{cond:stationarity} yields
$\nabla\risk(\bar x)=0$. Every accumulation point of the trajectory therefore
lies in $\mathcal S_{\delta_\ast}$, the stationary points on the limiting
boundary $\partial K_{\delta_\ast}$.

The trajectory lies in the compact set $K_{\delta_{\mathrm{th}}}$ by Step 1,
and $\mathcal S_{\delta_\ast}$ is closed, so
Lemma~\ref{lem:setwise_from_accumulation} gives
$\dist{x_t}{\mathcal S_{\delta_\ast}}\to0$. Since
$\mathcal S_{\delta_\ast}\subseteq\mathcal S$, this proves~(1).

\textbf{Step 3. The terminal case: conclusion (2).}
Assume instead that $\delta_\ast=\overline\delta$. By Step 1 the trajectory
satisfies $x_t\in K_{\delta_t}$ with $\delta_t\uparrow\overline\delta$, and by
definition
$K_{\overline\delta}=\bigcap_{\underline\delta<\delta<\overline\delta}K_\delta$.
Let $\bar x$ be an accumulation point, say $x_{t_j}\to\bar x$, and fix
$\delta\in(\underline\delta,\overline\delta)$. For all large $j$ we have
$\delta_{t_j}>\delta$, hence $x_{t_j}\in K_{\delta_{t_j}}\subseteq K_\delta$
by Axiom~\ref{cond:nesting}, and $K_\delta$ is closed, so $\bar x\in K_\delta$.
As $\delta$ was arbitrary, $\bar x\in K_{\overline\delta}$. The terminal set is
closed, being an intersection of closed sets, and the trajectory lies in the
compact set $K_{\delta_{\mathrm{th}}}$, so
Lemma~\ref{lem:setwise_from_accumulation} gives
$\dist{x_t}{K_{\overline\delta}}\to 0$, which is~(2).

The conclusion of Step 3 is convergence to the terminal \emph{set}. 
The finer asymptotics in the terminal case $\delta_\ast=\overline\delta$ ---
the reduced recursion
that the gradient step induces on $K_{\overline\delta}$, its post-critical
period-two behavior, and the inheritance of that behavior by a trajectory that
only approaches $K_{\overline\delta}$ --- are treated in
Appendices~\ref{sec:convergence when delta_* = 2}
and~\ref{sec:terminal_inheritance} and are not part of
Theorem~\ref{thm:abstract_state_dependent_convergence}.

\textbf{Step 4. Pointwise upgrade: conclusion (3).}
We first record that the increments vanish. Step~2 gives
$\dist{x_t}{\mathcal S_{\delta_\ast}}\to0$ when
$\delta_\ast<\overline\delta$, and $\mathcal S_{\delta_\ast}$ is compact, so
the distance is attained: there are $y_t\in\mathcal S_{\delta_\ast}$ with
$\|x_t-y_t\|\to0$, and both sequences lie in $K_{\delta_{\mathrm{th}}}$. The
gradient $\nabla\risk$ is continuous and vanishes on $\mathcal S$, so its
uniform continuity on $K_{\delta_{\mathrm{th}}}$ gives
$\|\nabla\risk(x_t)\|=\|\nabla\risk(x_t)-\nabla\risk(y_t)\|\to0$, and
therefore
\begin{align}
    \|x_{t+1}-x_t\|=\eta\|\nabla\risk(x_t)\|\longrightarrow 0 .
    \label{eq:abstract_vanishing_increments}
\end{align}

Assume now the finiteness hypothesis of~(3). Step 1 gives
$\delta_\ast\ge\delta_0>\delta_{\mathrm{th}}$ and case~(1) gives
$\delta_\ast<\overline\delta$, so $\mathcal S_{\delta_\ast}$ is finite; by
Step 2 it contains the limit set of the trajectory, which is nonempty by
compactness. Every point of the limit set is therefore isolated in it, and the
trajectory is bounded with vanishing increments by
Eq.~\eqref{eq:abstract_vanishing_increments}, so
Lemma~\ref{lem:finite_limit_set} gives $x_t\to p$ for a single
$p\in\mathcal S_{\delta_\ast}$. This proves~(3); the argument discards no
initializations beyond $E_0$ of Step 1.

\textbf{Step 5. Exclusion of unstable limits: conclusions (4) and (5).}
Let $\mathcal U\subseteq K_{\delta_{\mathrm{th}}}\cap\mathcal S$ be closed, hence
compact by Axiom~\ref{cond:P_is_pd}, and assume that $D\gd_\eta(p)$ has an eigenvalue of modulus strictly greater than
one for every $p\in\mathcal U$. This step bounds the set of initializations
whose trajectory converges to a point of $\mathcal U$, and it uses neither the
finiteness condition of~(3) nor the pointwise convergence conclusion of~(3).

Since $\risk\in C^2$, the map $\gd_\eta$ is $C^1$, and every $p\in\mathcal U$
is a fixed point of $\gd_\eta$. Corollary~\ref{cor:local_nonescaping_unstable_fixed_point}
therefore applies at each $p\in\mathcal U$ and yields a radius $r_p>0$ and a
local nonescaping set
\begin{align}
    W_p
    =
    \bigl\{
        x\in B(p,r_p):
        \gd_\eta^{\,t}(x)\in B(p,r_p)\ \text{for all }t\ge 0
    \bigr\}
\end{align}
of measure zero. The balls $\{B(p,r_p):p\in\mathcal U\}$ form an open
cover of the compact set $\mathcal U$, so there are finitely many
$p_1,\dots,p_N\in\mathcal U$ with
$\mathcal U\subseteq\bigcup_{i=1}^N B(p_i,r_{p_i})$.

Suppose $x_t\to q$ for some $q\in\mathcal U$. Choose $i$ with
$q\in B(p_i,r_{p_i})$ and then $r>0$ with $B(q,r)\subseteq B(p_i,r_{p_i})$.
There is $m\ge 0$ such that $x_t\in B(q,r)\subseteq B(p_i,r_{p_i})$ for all
$t\ge m$, so the trajectory started at $x_m$ never leaves $B(p_i,r_{p_i})$, and
the defining property of $W_{p_i}$ gives $x_m\in W_{p_i}$. Consequently
\begin{align}
    \bigl\{
        x_0\in\operatorname{int}K_{\delta_{\mathrm{th}}}:
        x_t\to q\ \text{for some } q\in\mathcal U
    \bigr\}
    \subseteq
    \bigcup_{i=1}^N
    \bigcup_{m=0}^\infty
    \gd_\eta^{-m}\bigl(W_{p_i}\bigr).
    \label{eq:abstract_bad_set_unstable_limit}
\end{align}
Each $W_{p_i}$ has measure zero, and $\gd_\eta$ is a submersion almost
everywhere, so Corollary~\ref{cor:gd_preimage_null} makes every iterated
preimage on the right-hand side of
Eq.~\eqref{eq:abstract_bad_set_unstable_limit} of measure zero as well. The
union is countable, hence the exceptional set has measure zero. This
proves~(4).

It remains to prove~(5), which combines~(3) and~(4) as follows: under the
hypotheses of both, an initialization for which case~(1) occurs has a
trajectory converging to a point of $K_{\delta_{\mathrm{th}}}\cap\mathcal S$,
and~(4) confines such initializations to a set of measure zero. So assume the
finiteness condition of~(3), and that every point of
$K_{\delta_{\mathrm{th}}}\cap\mathcal S$ satisfies the spectral condition
of~(4). Since $\mathcal S$ is closed by continuity of $\nabla\risk$, the set
$\mathcal U:=K_{\delta_{\mathrm{th}}}\cap\mathcal S$ is closed, as~(4)
requires.

Two sets of measure zero are discarded: the set $E_0$ of Step 1, outside which
the state trajectory is defined and the conclusion of~(3) is available, and the
set $E_4$ of initializations whose trajectory converges to a point of
$\mathcal U$, which has measure zero by~(4) applied with $\mathcal U$. Suppose
$x_0\in\operatorname{int}K_{\delta_{\mathrm{th}}}\setminus(E_0\cup E_4)$ and
case~(1) occurs. Then~(3) gives $x_t\to p$ for some $p\in\mathcal S$, and the
nesting $x_t\in K_{\delta_t}\subseteq K_{\delta_{\mathrm{th}}}$ together with
closedness of $K_{\delta_{\mathrm{th}}}$ places $p$ in $\mathcal U$; the
trajectory therefore converges to a point of $\mathcal U$, contradicting
$x_0\notin E_4$. Hence the
set of initializations for which case~(1) occurs is contained in
$E_0\cup E_4$ and has measure zero, and for every remaining initialization
$\delta_\ast=\overline\delta$, so that~(2) applies. This proves~(5) and
completes the proof.

\begin{remark}[Scope of the abstract convergence theorem]
Although stated for gradient descent, the argument extends to other autonomous
discrete-time dynamics. For a $C^1$ update map $g$, one replaces stationary
points of $\risk$ by fixed points of $g$ and assumes the corresponding
boundary-inward property (Axiom~\ref{cond:monotonicity}), fixed-point equality
condition (Axiom~\ref{cond:stationarity}), preimage regularity
(Corollary~\ref{cor:gd_preimage_null}), and measure-zero local nonescaping-set
conclusion (Corollary~\ref{cor:local_nonescaping_unstable_fixed_point}). The
only adjustment in the proof is the vanishing-increment argument in Step~4:
if the trajectory lies in a compact set and approaches $\operatorname{Fix}(g)$,
then continuity of $g-\mathrm{id}$ gives
$\|x_{t+1}-x_t\|=\|g(x_t)-x_t\|\to0$. Convergence to the fixed-point set does
not by itself give a single limit, so any pointwise conclusion has to be
checked separately, for instance by requiring each nonterminal fixed-point
slice $\operatorname{Fix}(g)\cap\{x:I(\delta;\,x)=0\}$ to be finite and applying
Lemma~\ref{lem:finite_limit_set}.

For non-autonomous dynamics, such as gradient descent with a varying step size
$\eta_t$, the theorem should not be applied as it stands. Two of its steps use
the autonomous structure directly: the passage in Step 2 from $R_t\to0$ to
$\dist{x_t}{\mathcal S}\to0$, which degenerates when
$\eta_t\to0$, and the exclusion of unstable limits in Step 5, which rests on
the autonomous local nonescaping statement of
Corollary~\ref{cor:local_nonescaping_unstable_fixed_point}. A non-autonomous
extension would therefore require uniform versions of both arguments. We do
not pursue such an extension here.
\end{remark}

\subsection{Verifications for the concrete models}
\label{subsec:concrete_verifications}

The four subsubsections below supply, once each, the four facts that the
applications in Appendix~\ref{subsec:applications_of_theorem5} need in order
to invoke Theorem~\ref{thm:abstract_state_dependent_convergence}: that the
three certificate families satisfy the state-dependent Lyapunov axioms, the
structure of
$\mathcal S_\delta$ at a fixed nonterminal state, the instability of the
non-minimizing stationary points, and the instability of the global minimizers
in the post-critical regime. By the rescaling of
Subsection~\ref{subsec:liang_summary}, we may assume $\sigma=1$ throughout this
subsection without loss of generality.

In all three models the gradient descent map $\gd_\eta$ takes the unified form
\begin{align}
\begin{aligned}
    b_{t+1} &= \bigl(1-\eta(\|a_t\|^2+u_t^2)\bigr)b_t + \eta a_t, &
    v_{t+1} &= \bigl(1-\eta(\|a_t\|^2+u_t^2)\bigr)v_t,\\
    a_{t+1} &= \bigl(1-\eta(\|b_t\|^2+v_t^2)\bigr)a_t + \eta b_t, &
    u_{t+1} &= \bigl(1-\eta(\|b_t\|^2+v_t^2)\bigr)u_t,
\end{aligned}
\label{eq:gd_recalled}
\end{align}
and the certificate is
\begin{align}
    I(\delta;\,a,b,u,v)
    =
    \delta\bigl(\|a\|^2+\|b\|^2+u^2+v^2\bigr)
    -\delta^2\langle a,b\rangle
    +\delta^2-4 ,
    \label{eq:certificate_recalled}
\end{align}
both read with $a,b\in\RR$ and the off-signal coordinates $u,v$ absent for
scalar factorization, with $a,b,u,v\in\RR$ for rank-1 factorization, and with
$a,b\in\RR^{n-1}$, $u,v\in\RR$ for rank-1 approximation. When we differentiate
Eq.~\eqref{eq:gd_recalled} we drop the time index and write
$(a^+,b^+,u^+,v^+):=\gd_\eta(a,b,u,v)$ for the image of one step. We call $(a,b)$
the signal coordinates and $(u,v)$ the off-signal coordinates, and the
\emph{signal block} of a matrix means its diagonal block on $(a,b)$.
Table~\ref{tab:three_models} in Section~\ref{sec:rank-1 convergence}
records the minimizer and stationary sets.

\subsubsection{Verification of the axioms}
\label{subsec:verification_of_axioms}

\begin{lemma}[The three certificates satisfy the axioms]
\label{lem:concrete_axioms}
Let $\sigma=1$, $\underline \delta = 0$, and $\overline\delta=2$.

(1) All three certificates satisfy the structural
Axioms~\ref{cond:P_is_pd}--\ref{cond:level_set_as_state}; no condition on
$\eta$ enters.

(2) For every $\eta\in(0,2)$, the certificates $\isc$ and $\ifac$ satisfy the
dynamical Axioms~\ref{cond:monotonicity}--\ref{cond:stationarity} with
$\delta_{\mathrm{th}}=\eta$.

(3) For every $\eta\in(0,1)$, the certificate $\iapx$ satisfies
Axioms~\ref{cond:monotonicity}--\ref{cond:stationarity} with
$\delta_{\mathrm{th}}=\delta_\mathrm{th}^\apx=2\bigl(1-\sqrt{1-\eta^2}\bigr)/\eta\in(\eta,2)$.

(4) Each of the three certificates is a polynomial in $(\delta,x)$, so its
coefficient functions are polynomials in $\delta$ and are $C^1$ on $(0,2)$;
and in each of the three models $\mathcal S$ has measure zero and the
corresponding GD map is a submersion almost everywhere.
\end{lemma}

\begin{proof}
Write each certificate as $x^\top P(\delta)x+r(\delta)$ with
$r(\delta)=\delta^2-4$. In the coordinate order $(a,b)$ followed by $(u,v)$,
\begin{align}
    P_\sca(\delta)
    =
    \begin{pmatrix}
        \delta & -\delta^2/2\\
        -\delta^2/2 & \delta
    \end{pmatrix},
    \quad
    P_\fac(\delta)=P_\sca(\delta)\oplus\delta I_2,
    \quad
    P_\apx(\delta)=\bigl(P_\sca(\delta)\otimes I_{n-1}\bigr)\oplus\delta I_2 .
\end{align}
We take the structural axioms in order, which gives~(1), and then the dynamical
pair, which gives~(2) and~(3).

The eigenvalues of $P_\sca(\delta)$ are $\delta\pm\delta^2/2$, both positive
for $\delta\in(0,2)$. In $P_\fac$ and $P_\apx$ the off-signal block contributes
the eigenvalue $\delta>0$. With $r(\delta)<0$, this makes
Lemma~\ref{lem:compact_quadratic_sublevel} applicable, and it gives
Axiom~\ref{cond:P_is_pd}.

For Axiom~\ref{cond:nesting}, dividing each certificate by $4-\delta^2>0$, the normalized matrix
has eigenvalues $\delta/\bigl(2(2+\delta)\bigr)$,
$\delta/\bigl(2(2-\delta)\bigr)$, and $\delta/(4-\delta^2)$, with
$\delta$-independent eigenspaces. Each is strictly increasing on $(0,2)$, 
since the derivatives are $(2+\delta)^{-2}$, $(2-\delta)^{-2}$, and
$(4+\delta^2)(4-\delta^2)^{-2}$. Hence
$P(\delta')/(4-\delta'^2)-P(\delta)/(4-\delta^2)\succ0$ for
$\delta<\delta'<2$. The sublevel sets are
$\{x^\top P(\delta)x/(4-\delta^2)\le1\}$, so this gives
$K_{\delta'}\subseteq\operatorname{int}K_\delta$ for $\delta<\delta'<2$. For
$\delta'=2$, pick any $\delta''\in(\delta,2)$; then
$K_2\subseteq K_{\delta''}\subseteq\operatorname{int}K_\delta$.

For Axiom~\ref{cond:end_point_measure_zero}, recall that
$K_2=\bigcap_{0<\delta<2}K_\delta$, so $I(\delta;\,x)\le0$ for every
$\delta<2$ and every $x\in K_2$, whence $I(2;\,x)\le0$ by continuity. A direct
evaluation gives
\begin{align}
\begin{aligned}
    \isc(2;\cdot) &= 2(a-b)^2,\\
    \ifac(2;\cdot) &= 2(a-b)^2+2(u^2+v^2),\\
    \iapx(2;\cdot) &= 2\|a-b\|^2+2(u^2+v^2),
\end{aligned}
\end{align}
all nonnegative, so $K_2$ is contained in the zero set of $I(2;\cdot)$: the
balanced set $\{a=b\}$ in the scalar case and $\{a=b,\ u=v=0\}$ in the other
two. Each of these is a proper linear subspace of the ambient space and
therefore has measure zero. Only this inclusion is needed here;
Proposition~\ref{prop:terminal_sets} identifies $K_2$ exactly.

For Axiom~\ref{cond:level_set_as_state} we check that every point outside $K_2$
has exactly one state. Let $x\notin K_2$. In each of the three
certificates every term except the constant $-4$ carries a factor $\delta$, so
$I(\delta;\,x)\to-4$ as $\delta\downarrow0$. Hence there is
$\delta_-\in(0,2)$ with $I(\delta_-;\,x)<0$.
Since
$x\notin K_2=\bigcap_{0<\delta<2}K_\delta$, there is
$\delta_+\in(0,2)$ with $x \notin K_{\delta_+}$, i.e.,  $I(\delta_+;\,x)>0$, and strict nesting forces $\delta_+>\delta_-$. The map
$\delta\mapsto I(\delta;\,x)$ is continuous on the closed interval
$[\delta_-,\delta_+]\subset(0,2)$ and satisfies $I(\delta_-;\,x)<0<I(\delta_+;\,x)$,
so it has at least one root there. Two distinct
roots $\delta_1<\delta_2$ are impossible: strict nesting would give
$K_{\delta_2}\subseteq\operatorname{int}K_{\delta_1}$, so
$I(\delta_2;\,x)=0$ would force $I(\delta_1;\,x)<0$. The root is therefore
unique, and for $x\in K_2$ the state is defined to be $2$, exactly as in
Axiom~\ref{cond:level_set_as_state}.

Axioms~\ref{cond:monotonicity} and~\ref{cond:stationarity} are the
boundary-inward statements. For $\isc$, the identity
Eqs.~\eqref{eq:certificate_for_scalar}--\eqref{eq:sc_M_t_R_t_def} gives
$\isc(\delta;\,\gd_\eta^\sca(x))=\eta(\delta-\eta)(\delta^2-4)L^2$
on the level set, with $L=1-ab$; for $\eta<\delta<2$ this is nonpositive, and
it vanishes exactly when $L=0$. On that level set $L=0$ is equivalent to
$\nabla\risk_\sca(x)=0$: the scalar stationary set is
$\{ab=1\}\cup\{(0,0)\}$, and the origin satisfies $\isc(\delta;0,0)=\delta^2-4<0$
for $\delta<2$, so it is not on the level set. For $\ifac$ and
$\iapx$, the corresponding statements are
Propositions~\ref{prop:boundary_inward_rank1}
and~\ref{prop:boundary_inward_rank1_approx}, proved in
Appendix~\ref{sec:boundary_inward_proofs}; in the approximation case the
proposition additionally requires $q_\eta(\delta)<0$, which for
$\eta\in(0,1)$ holds exactly on $(\delta_\mathrm{th}^\apx,2]$, so the
threshold is $\delta_\mathrm{th}^\apx$. In all three cases a stationary point
is fixed by the gradient step, so equality holds there, which is the remaining
direction of Axiom~\ref{cond:stationarity}.

For~(4), the polynomiality is immediate from the displayed formulas, each of
which is polynomial in $\delta$ and in the coordinates; and $\mathcal S$ is a
finite union of algebraic sets of dimension strictly less than the ambient
dimension in each model, hence of measure zero, and the almost-everywhere
submersion property of $\gd_\eta^\sca$, $\gd_\eta^\fac$, and $\gd_\eta^\apx$
follows from Proposition~\ref{prop:gd_submersion_ae} through the direct
specializations recorded in Appendix~\ref{apdx:gd_submersion}.
\end{proof}

\subsubsection{Stationary points on a nonterminal boundary}
\label{subsec:levelwise_stationary_slices}

\begin{lemma}[The sets $\mathcal S_\delta$ at a nonterminal state]
\label{lem:levelwise_stationary_slices}
Fix $\delta\in(0,2)$ and set $\tau_\delta:=\sqrt{(4-\delta^2)/\delta}>0$.

(1) For scalar factorization,
\begin{align}
    \mathcal S\cap\{\isc(\delta;\cdot)=0\}
    =
    \bigl\{(a,b)\in\RR^2: ab=1,\ a^2+b^2=\tfrac4\delta\bigr\},
\end{align}
which consists of exactly four points, all of them global minimizers.

(2) For rank-1 factorization,
\begin{align}
    \mathcal S\cap\{\ifac(\delta;\cdot)=0\}
    &=
    \bigl\{(a,b,0,0): ab=1,\ a^2+b^2=\tfrac4\delta\bigr\}
    \notag\\
    &\qquad\cup
    \bigl\{(0,0,\pm\tau_\delta,0),\,(0,0,0,\pm\tau_\delta)\bigr\},
    \label{eq:fac_slice}
\end{align}
which consists of exactly eight points, four of them global minimizers.

(3) For rank-1 approximation with $n\ge3$ the minimizing and the
non-minimizing parts must be recorded separately, because only the latter is
finite. First,
\begin{align}
    \mathcal M\cap\{\iapx(\delta;\cdot)=0\}
    =
    \bigl\{(a,\,a/\|a\|^2,\,0,0):\ \|a\|^2+\|a\|^{-2}=4/\delta\bigr\},
    \label{eq:apx_minimizer_slice}
\end{align}
a compact set of dimension $n-2\ge1$ and hence uncountable, whereas
\begin{align}
    (\mathcal S\setminus\mathcal M)\cap\{\iapx(\delta;\cdot)=0\}
    =
    \bigl\{(0,0,\pm\tau_\delta,0),\,(0,0,0,\pm\tau_\delta)\bigr\}
    \label{eq:apx_nonmin_slice}
\end{align}
consists of exactly four points.

(4) In cases (2) and (3), every non-minimizing point of $\mathcal S_\delta$ is at
Euclidean distance at least $\sqrt2$ from $\mathcal M$.
\end{lemma}

\begin{proof}
(1) The origin is not on the level set, because
$\isc(\delta;0,0)=\delta^2-4<0$ for $\delta<2$. On the minimizer set
$\{ab=1\}$, the certificate reduces to $\delta(a^2+b^2)-4=0$, that is
$a^2+b^2=4/\delta$. The
pair $(a+b,a-b)$ then satisfies $(a+b)^2=4/\delta+2$ and
$(a-b)^2=4/\delta-2$, both strictly positive for $\delta<2$, so there are
exactly four solutions, all global minimizers.

(2) On $\mathcal M$ we have
$u=v=0$, so the level condition is the scalar one and (1) gives the first four
points. On the non-minimizing stationary set $\{a=b=0,\ uv=0\}$, the level
condition reads
$\delta^2-4+\delta(u^2+v^2)=0$, that is $u^2+v^2=\tau_\delta^2>0$; combined
with $uv=0$ this gives the four listed points. The origin satisfies
$\ifac(\delta;0)=\delta^2-4<0$ and is therefore not on the level set.

(3) On $\mathcal M$ we have $a\parallel b$, $\langle a,b\rangle=1$, and
$u=v=0$, so $b=a/\|a\|^2$ and the level condition reads
$\|a\|^2+\|a\|^{-2}=4/\delta$, that is $\|a\|^4-(4/\delta)\|a\|^2+1=0$, whose
roots in $\|a\|^2$ are the two mutually reciprocal positive numbers
$\tfrac2\delta\bigl(1\pm\sqrt{1-\delta^2/4}\bigr)$, distinct because
$\delta<2$. The set is therefore the union of the two corresponding spheres in
$\RR^{n-1}$, transported by
$a\mapsto(a,a/\|a\|^2,0,0)$. It is compact, and each sphere has dimension
$n-2\ge1$ because $n\ge3$, so the set is positive-dimensional and
uncountable. The computation on the non-minimizing stationary set
$\{a=b=0,\ uv=0\}$ is identical to the one in (2).

(4) Let $p=(0,0,u,v)$ be a non-minimizing point of $\mathcal S_\delta$ and let
$(a,b,0,0)\in\mathcal M$. On $\mathcal M$ we have
$\|a\|^2+\|b\|^2\ge2\|a\|\,\|b\|\ge2\langle a,b\rangle=2$, using the
Cauchy--Schwarz inequality in the approximation case and $ab=1$ in the
factorization case. Hence
\begin{align}
    \bigl\|(a,b,0,0)-(0,0,u,v)\bigr\|^2
    =
    \|a\|^2+\|b\|^2+u^2+v^2
    \ge 2,
\end{align}
which is the claim.
\end{proof}

\subsubsection{Instability of non-minimizing stationary points}
\label{subsec:exclude_nonmin_stationary_rank1}

The three models share one $2\times2$ matrix in the signal block at a
non-minimizing stationary point, and the whole verification reduces to its
spectrum. For $\tau\in\RR$ set
\begin{align}
    J(\tau)
    :=
    \begin{pmatrix}
        1 & \eta\\
        \eta & 1-\eta\tau^2
    \end{pmatrix},
    \qquad
    \lambda_\pm(\tau)
    :=
    1-\frac{\eta\tau^2}{2}
    \pm
    \frac{\eta}{2}\sqrt{\tau^4+4}.
    \label{eq:nonmin_block}
\end{align}
A direct computation gives $\operatorname{tr}J(\tau)=2-\eta\tau^2$ and
$\det J(\tau)=1-\eta\tau^2-\eta^2$, so $\lambda_\pm(\tau)$ are the eigenvalues of
$J(\tau)$. Since $\sqrt{\tau^4+4}>\tau^2$ for every $\tau\in\RR$,
\begin{align}
    \lambda_+(\tau)>1
    \qquad\text{for every }\tau\in\RR .
    \label{eq:nonmin_expanding}
\end{align}
The matrix obtained from $J(\tau)$ by exchanging the two coordinates has the
same characteristic polynomial and hence the same eigenvalues.

\begin{lemma}[Every non-minimizing stationary point has an expanding direction]
\label{lem:nonmin_unstable}
Let $\eta>0$ and assume $\sigma=1$.

(1) For scalar factorization, $\mathcal S\setminus\mathcal M=\{(0,0)\}$ and
$D\gd_\eta^\sca(0,0)=J(0)$, whose eigenvalues are $1\pm\eta$.

(2) For rank-1 factorization, every $p\in\mathcal S\setminus\mathcal M$ has
the form $p=(0,0,u,v)$ with $uv=0$; writing $\tau^2:=u^2+v^2$, the Jacobian
$D\gd_\eta^\fac(p)$ is block diagonal for the splitting $(a,b)\oplus(u,v)$, and
its signal block is $J(\tau)$ up to a coordinate exchange.

(3) For rank-1 approximation, every $p\in\mathcal S\setminus\mathcal M$ has
the same form, and the signal block of $D\gd_\eta^\apx(p)$ is
$J(\tau)\otimes I_{n-1}$, again up to a coordinate exchange.

In all three cases $D\gd_\eta(p)$ has the eigenvalue $\lambda_+(\tau)>1$.
\end{lemma}

\begin{proof}
(1) With $\gd_\eta^\sca(a,b)=\bigl(a-\eta(ab-1)b,\ b-\eta(ab-1)a\bigr)$, the
partial derivatives at the origin are $\partial_a a^+=1-\eta b^2=1$,
$\partial_b a^+=-\eta(ab-1)-\eta ab=\eta$, and symmetrically for $b^+$, so
$D\gd_\eta^\sca(0,0)=J(0)$ and $\lambda_\pm(0)=1\pm\eta$.

(2) Differentiate Eq.~\eqref{eq:gd_recalled} at
$p=(0,0,\tau,0)$. Using $a=b=0$ we get $\partial_a a^+=1$,
$\partial_b a^+=\eta-2\eta ba=\eta$, $\partial_b b^+=1-\eta(a^2+u^2)=1-\eta\tau^2$,
$\partial_a b^+=\eta-2\eta ab=\eta$, while all mixed derivatives between the
$(a,b)$ and $(u,v)$ coordinates vanish because each of them carries a factor
$a$, $b$, or $v$. The off-signal block is $\diag(1,\,1-\eta\tau^2)$. Hence
\begin{align}
    D\gd_\eta^\fac(0,0,\tau,0)
    =
    \begin{pmatrix}
        1 & \eta & 0 & 0\\
        \eta & 1-\eta\tau^2 & 0 & 0\\
        0 & 0 & 1 & 0\\
        0 & 0 & 0 & 1-\eta\tau^2
    \end{pmatrix},
    \label{eq:jacobian_nonmin_u_axis}
\end{align}
whose signal block is $J(\tau)$. At $p=(0,0,0,\tau)$ the roles of $a$ and $b$
are exchanged, giving the signal block
$\begin{psmallmatrix}1-\eta\tau^2&\eta\\ \eta&1\end{psmallmatrix}$ with the
same eigenvalues. Both Jacobians depend on $\tau$ only through $\tau^2$, so the
sign of $\tau$ is irrelevant. The origin is the case $\tau=0$.

(3) Differentiate Eq.~\eqref{eq:gd_recalled} in its vector reading at
$p=(0,0,\tau,0)$ with $a,b\in\RR^{n-1}$. The same cancellations give
$\partial_a a^+=I_{n-1}$, $\partial_b a^+=\eta I_{n-1}$,
$\partial_a b^+=\eta I_{n-1}$, and
$\partial_b b^+=(1-\eta\tau^2)I_{n-1}$, with vanishing mixed blocks, so the
signal block is $J(\tau)\otimes I_{n-1}$ and each $\lambda_\pm(\tau)$ occurs
with multiplicity $n-1$. The point $p=(0,0,0,\tau)$ is again handled by the
coordinate exchange.

In every case Eq.~\eqref{eq:nonmin_expanding} gives an eigenvalue
$\lambda_+(\tau)>1$.
\end{proof}

\begin{corollary}\label{cor:nonmin_compact_unstable}
For each of the three models and any threshold
$\delta_{\mathrm{th}}\in(0,2)$, the set
$(\mathcal S\setminus\mathcal M)\cap K_{\delta_{\mathrm{th}}}$ is compact and
every one of its points satisfies the spectral condition of
Theorem~\ref{thm:abstract_state_dependent_convergence}(4). It may therefore be
taken as $\mathcal U$ in that conclusion.
\end{corollary}

\begin{proof}
The set $\mathcal S\setminus\mathcal M$ is closed in each model. For scalar
factorization it is the single point $(0,0)$. In the other two, $\mathcal M$ is
disjoint from $\{a=b=0\}$ because $\langle a,b\rangle=1$ there, so
$\mathcal S\setminus\mathcal M$ is the entire non-minimizing stationary set
$\{a=b=0,\ uv=0\}$, the union of the two coordinate subspaces
$\{a=b=0,\ v=0\}$ and
$\{a=b=0,\ u=0\}$. Since $K_{\delta_{\mathrm{th}}}$ is compact, the
intersection is compact. It is contained in $\mathcal S$, and
Lemma~\ref{lem:nonmin_unstable} supplies the expanding eigenvalue at each of
its points.
\end{proof}

\subsubsection{Instability of minimizers in the post-critical regime}
\label{subsec:minimizer_instability_eta_gt_one}

Scalar and rank-1 factorization share a second $2\times2$ block, this time at a
global minimizer, and one calculation covers both.

\begin{lemma}[Minimizer instability for $\eta>1$]
\label{lem:minimizer_instability}
Assume $\sigma=1$ and let $x_\ast$ be a global
minimizer with signal coordinates $(a_\ast,b_\ast)$, so that
$a_\ast b_\ast=1$. Define the signal block
\begin{align}
    J_\ast
    :=
    \begin{pmatrix}
        1-\eta b_\ast^2 & -\eta\\
        -\eta & 1-\eta a_\ast^2
    \end{pmatrix},
    \label{eq:minimizer_signal_block}
\end{align}
whose eigenvalues are $1$ and $1-\eta(a_\ast^2+b_\ast^2)$.

(1) For scalar factorization, $D\gd_\eta^\sca(a_\ast,b_\ast)=J_\ast$.

(2) For rank-1 factorization,
$D\gd_\eta^\fac(a_\ast,b_\ast,0,0)=J_\ast\oplus\diag(1-\eta b_\ast^2,\,1-\eta a_\ast^2)$.

(3) If $\eta>1$, then
$\bigl|1-\eta(a_\ast^2+b_\ast^2)\bigr|\ge2\eta-1>1$, so in both models every
global minimizer is an unstable fixed point of the GD map. Moreover, for any
threshold $\delta_{\mathrm{th}}\in(0,2)$, the set
$\mathcal M\cap K_{\delta_{\mathrm{th}}}$ is compact, and it is therefore an
admissible choice of $\mathcal U$ in
Theorem~\ref{thm:abstract_state_dependent_convergence}(4).
\end{lemma}

\begin{proof}
For the scalar map, $\partial_a a^+=1-\eta b^2$ and
$\partial_b a^+=-\eta(ab-1)-\eta ab$, which equals $-\eta$ when $ab=1$; the
derivatives of $b^+$ are symmetric. This gives (1). The trace and determinant
of $J_\ast$ are $2-\eta(a_\ast^2+b_\ast^2)$ and
\begin{align}
    (1-\eta b_\ast^2)(1-\eta a_\ast^2)-\eta^2
    =
    1-\eta(a_\ast^2+b_\ast^2)+\eta^2a_\ast^2b_\ast^2-\eta^2
    =
    1-\eta(a_\ast^2+b_\ast^2),
\end{align}
using $a_\ast b_\ast=1$, so the eigenvalues are $1$ and
$1-\eta(a_\ast^2+b_\ast^2)$.

For (2), differentiate Eq.~\eqref{eq:gd_recalled} at
$(a_\ast,b_\ast,0,0)$. The signal block reproduces the scalar computation,
the mixed blocks vanish because $u_\ast=v_\ast=0$, and the off-signal block is
$\diag(1-\eta b_\ast^2,\,1-\eta a_\ast^2)$.

For (3), $a_\ast b_\ast=1$ gives $a_\ast^2+b_\ast^2\ge2a_\ast b_\ast=2$, so
for $\eta>1$ we get
$\bigl|1-\eta(a_\ast^2+b_\ast^2)\bigr|=\eta(a_\ast^2+b_\ast^2)-1\ge2\eta-1>1$.
Finally $\mathcal M$ is closed and $K_{\delta_{\mathrm{th}}}$ is compact, so
$\mathcal M\cap K_{\delta_{\mathrm{th}}}$ is compact.
\end{proof}

The rank-1 approximation theorem covers the pre-critical range $\eta<1$ only,
so no post-critical minimizer statement is claimed for it here.

\subsection{Applications of the abstract convergence theorem}
\label{subsec:applications_of_theorem5}

We now apply Theorem~\ref{thm:abstract_state_dependent_convergence} using the
axiom, stationary-slice, and instability verifications above. We work with
$\sigma=1$; the rescaling of Subsection~\ref{subsec:liang_summary} gives the
general statements. The terminal upgrades used below are proved later in
Corollary~\ref{cor:subcritical_terminal_inheritance}.

\paragraph{Scalar factorization.}
We prove Proposition~\ref{prop:scalar_convergence}. The same argument also
recovers Theorem~\ref{thm:convergence_region}(1). By
Lemma~\ref{lem:concrete_axioms}(1),(2),(4),
Theorem~\ref{thm:abstract_state_dependent_convergence} applies with
$\delta_{\mathrm{th}}=\eta$ and $\overline\delta=2$, and its initialization
domain $\operatorname{int}K_\eta=\{\isc(\eta;\cdot)<0\}$ is the certified region
at $\sigma=1$. By
Lemma~\ref{lem:levelwise_stationary_slices}(1) the set $\mathcal S_\delta$ is
finite for every $\delta\in[\eta,2)$, so conclusion~(3) is available.

Let $\eta\in(0,1)$. If $\delta_\ast<2$, conclusion~(3) gives $x_t\to p$ for a
single $p\in\mathcal S_{\delta_\ast}$, and that set
consists of global minimizers. If $\delta_\ast=2$, conclusion~(2) gives
$\dist{x_t}{K_2^\sca}\to0$, and
Corollary~\ref{cor:subcritical_terminal_inheritance}, proved later, upgrades
this to convergence to a global minimizer. Thus the dynamics converge to a
global minimizer, recovering Theorem~\ref{thm:convergence_region}(1) from the
certificate.

Let $\eta\in(1,2)$. The set $K_\eta\cap\mathcal S$ is the union of
$\mathcal M\cap K_\eta$ and the single point $(0,0)$. Every point of the
first set has an expanding eigenvalue by
Lemma~\ref{lem:minimizer_instability}(1),(3), and the origin has the
eigenvalue $1+\eta>1$ by Lemma~\ref{lem:nonmin_unstable}(1). Conclusion~(5)
therefore applies: for almost every certified initialization the nonterminal
case does not occur, so $\delta_\ast=2$, and conclusion~(2) gives
$\dist{x_t}{K_2^\sca}\to0$. This is assertion~(1) of the proposition.

\paragraph{Rank-1 matrix factorization.}
We prove Theorem~\ref{thm:convergence_region_rank_1}. By
Lemma~\ref{lem:concrete_axioms}(1),(2),(4) the theorem applies with
$\delta_{\mathrm{th}}=\eta$ and initialization domain
$\{\ifac(\eta;\cdot)<0\}$, and by
Lemma~\ref{lem:levelwise_stationary_slices}(2) the set $\mathcal S_\delta$ is
finite on $[\eta,2)$.

Let $\eta\in(0,1)$. If $\delta_\ast<2$, conclusion~(3) gives convergence to a
single point of the eight-point set in Eq.~\eqref{eq:fac_slice}, four of which
are global minimizers. By Corollary~\ref{cor:nonmin_compact_unstable}, the set
$\mathcal U=(\mathcal S\setminus\mathcal M)\cap K_\eta$ is admissible in
conclusion~(4), which rules out the four non-minimizing points for almost every
certified initialization. Thus the limit is a global minimizer. If
$\delta_\ast=2$, conclusion~(2) gives
$\dist{x_t}{K_2^\fac}\to0$, and
Corollary~\ref{cor:subcritical_terminal_inheritance}, proved later, upgrades
this to convergence to a global minimizer. This proves assertion (1).

Let $\eta\in(1,2)$. Applying conclusion~(4) with
$\mathcal U=\mathcal M\cap K_\eta$, admissible by
Lemma~\ref{lem:minimizer_instability}(3), shows that for almost every
certified initialization the dynamics do not converge to a global minimizer.
The set $K_\eta\cap\mathcal S$ is the union of $\mathcal M\cap K_\eta$ and
$(\mathcal S\setminus\mathcal M)\cap K_\eta$. The points of these two sets have
expanding eigenvalues by Lemmas~\ref{lem:minimizer_instability}(3)
and~\ref{lem:nonmin_unstable}(2), respectively. Conclusion~(5) therefore gives
$\delta_\ast=2$ for almost every certified initialization, and conclusion~(2)
yields $\dist{x_t}{K_2^\fac}\to0$. This proves assertion (2).

\paragraph{Isotropic-signal rank-1 approximation.}
We prove Theorem~\ref{thm:convergence_region_rank_1_approx} for a fixed
$\eta\in(0,1)$. By Lemma~\ref{lem:concrete_axioms}(1),(3),(4) the theorem
applies with $\delta_{\mathrm{th}}=\delta_\mathrm{th}^\apx$ and initialization
domain $\{\iapx(\delta_\mathrm{th}^\apx;\cdot)<0\}$. Conclusion~(3) is not
available here, because the minimizer slice
in Eq.~\eqref{eq:apx_minimizer_slice} is positive-dimensional, and hence
infinite, for $n\ge3$.

Suppose $\delta_\ast<2$. Conclusion~(1) places every limit point of the
trajectory in the set
$\mathcal S\cap\{\iapx(\delta_\ast;\cdot)=0\}$, which by
Lemma~\ref{lem:levelwise_stationary_slices}(3) is the union of the
minimizers of Eq.~\eqref{eq:apx_minimizer_slice} and the four points of
Eq.~\eqref{eq:apx_nonmin_slice}. Conclusion~(4), applied with
$\mathcal U=(\mathcal S\setminus\mathcal M)\cap K_{\delta_\mathrm{th}^\apx}$
and admissible by Corollary~\ref{cor:nonmin_compact_unstable}, rules out
convergence to any of those four points for almost every certified
initialization. Any non-minimizing limit point would be isolated in the limit
set by Lemma~\ref{lem:levelwise_stationary_slices}(4). Since the increments
vanish, Lemma~\ref{lem:finite_limit_set} would then force the entire trajectory
to converge to that point, which conclusion~(4) excludes almost surely. Hence
every limit point belongs to $\mathcal M$, and
Lemma~\ref{lem:setwise_from_accumulation} gives $\dist{x_t}{\mathcal M}\to0$.

Suppose instead $\delta_\ast=2$. Corollary~\ref{cor:subcritical_terminal_inheritance},
proved later, gives $\dist{x_t}{\mathcal M}\to0$ in this case.

Thus $\dist{x_t}{\mathcal M}\to0$ for almost every certified initialization.
No pointwise convergence is claimed.

%% file: apxG_terminal_manifold.tex

\section{Terminal geometry and the reduced cubic dynamics}
\label{sec:convergence when delta_* = 2}

Theorem~\ref{thm:abstract_state_dependent_convergence} leaves open what happens
in the terminal case $\delta_\ast=\overline\delta$: it gives convergence to
the terminal set and nothing more. This appendix answers the next question for
the three concrete models. In the terminal case $\delta_\ast=2$, the trajectory
approaches an explicit compact set $K_2$. We first identify $K_2$ and show that
it is forward invariant. The restricted gradient step then reduces to a common
one-dimensional cubic map, whose orbits we classify in both step-size
regimes. How the asymptotics of this reduced map are inherited by a full
trajectory that only approaches $K_2$ is a separate question, treated in
Appendix~\ref{sec:terminal_inheritance}.

\subsection{The terminal set \texorpdfstring{$K_2$}{K2} and its forward invariance}
\label{appx:K_2_description}

By Theorem~\ref{thm:abstract_state_dependent_convergence}(2), if
$\delta_\ast=2$ then $\dist{x_t}{K_2}\to0$, where
\begin{align}
    K_2 := \bigcap_{0<\delta<2} K_\delta.
\end{align}
For the rank-$1$ approximation model, we use the decomposition
$A=(a,u)$ and $B=(b,v)$, where $a,b\in\RR^{n-1}$ and $u,v\in\RR$.
The following proposition computes $K_2$ for all three models.

\begin{proposition}[The terminal sets]
\label{prop:terminal_sets}
In the normalized setting $\sigma=1$,
\begin{align}
    K_2^\sca &= \{(a,b)\in\RR^2 : a=b,\ a^2\le 2\},\\
    K_2^\fac &= \{(a,b,u,v)\in\RR^4 : a=b,\ u=v=0,\ a^2\le 2\},\\
    K_2^\apx &= \{(A,B)\in\RR^{2n} : a=b,\ u=v=0,\ \norm{a}^2\le 2\}.
\end{align}
\end{proposition}

\begin{proof}
The three certificates share the structure of the argument, so we run it once.
Let $I$ be any of $\isc,\ifac,\iapx$ and let $x\in K_2$. Then $I(\delta;\,x)\le0$
for every $\delta<2$, so $I(2;\,x)\le0$ by continuity. As recorded in the proof
of Lemma~\ref{lem:concrete_axioms},
\begin{align}
    \isc(2;\,a,b)&=2(a-b)^2,
    \notag\\
    \ifac(2;a,b,u,v)&=2(a-b)^2+2(u^2+v^2),
    \notag\\
    \iapx(2;a,b,u,v)&=2\norm{a-b}^2+2(u^2+v^2),
\end{align}
each of which is nonnegative and vanishes exactly on the balanced set
$\{a=b,\ u=v=0\}$ (with the off-signal conditions absent in the scalar case).
Hence every point of $K_2$ is balanced.

On the balanced set all three certificates collapse to the same expression.
Writing $a=b$ and $u=v=0$, and using $\langle a,b\rangle=\norm{a}^2$ in the
approximation case,
\begin{align}
    I(\delta;\,x)
    =
    (2\delta-\delta^2)\norm{a}^2+\delta^2-4 .
    \label{eq:balanced_certificate}
\end{align}
Since $2\delta-\delta^2>0$ on $(0,2)$, the condition $I(\delta;\,x)\le0$ is
equivalent to
\begin{align}
    \norm{a}^2
    \le
    \frac{4-\delta^2}{2\delta-\delta^2}
    =
    \frac{2+\delta}{\delta},
    \label{eq:balanced_ineq}
\end{align}
and $\delta\mapsto(2+\delta)/\delta$ is decreasing on $(0,2)$ with infimum $2$
attained in the limit $\delta\uparrow2$. Therefore a balanced point lies in
$K_2$ if and only if $\norm{a}^2\le2$, which is the assertion.
\end{proof}

Forward invariance of the terminal set is a property of the certificate
framework.

\begin{lemma}
\label{lem:terminal_forward_invariance}
Let $I(\delta;\,x)$ satisfy
Axioms~\ref{cond:nesting}--\ref{cond:monotonicity}. Then
$\gd_\eta(K_{\overline\delta})\subseteq K_{\overline\delta}$.
\end{lemma}

\begin{proof}
By Axioms~\ref{cond:nesting} and~\ref{cond:level_set_as_state}, a point
lies in $K_{\delta}$ if and only if its state is at least $\delta$.
Suppose $x\in K_{\overline\delta}$ but
$\gd_\eta(x)\notin K_{\overline\delta}$. Since
$K_{\overline\delta}=\bigcap_{\delta<\overline\delta}K_\delta$, there
exists $\delta_\circ<\overline\delta$ with
$\gd_\eta(x)\notin K_{\delta_\circ}$; note that this step does not require
$\gd_\eta(x)$ to carry a state. Fix
$\delta'\in(\max(\delta_\circ,\delta_{\mathrm{th}}),\overline\delta)$;
by Axiom~\ref{cond:nesting}, $K_{\delta'}\subseteq K_{\delta_\circ}$, so
$\gd_\eta(x)\notin K_{\delta'}$. Since $K_{\delta'}$ is closed and
$\gd_\eta$ is continuous, there is a neighborhood $U$ of $x$ with
$\gd_\eta(U)\cap K_{\delta'}=\varnothing$. Shrinking $U$, we may also assume
$U\subseteq K_{\delta'}$, since
$x\in K_{\delta''}\subseteq\operatorname{int}K_{\delta'}$ for any
$\delta''\in(\delta',\overline\delta)$. By
Axiom~\ref{cond:end_point_measure_zero}, we may pick
$x'\in U\setminus K_{\overline\delta}$, whose state $\delta(x')$ then lies
in $[\delta',\overline\delta)$. Axiom~\ref{cond:monotonicity} gives
$\gd_\eta(x')\in K_{\delta(x')}\subseteq K_{\delta'}$, contradicting
$\gd_\eta(U)\cap K_{\delta'}=\varnothing$.
\end{proof}

\subsection{Product reduction and the normalized cubic family}
\label{subsec:terminal_reduced_map}

Proposition~\ref{prop:terminal_sets} shows that every point of $K_2$ is
balanced, so a single signal variable describes it. Write
\begin{align}
    w :=
    \begin{cases}
        a=b\in\RR, & \text{for } \isc \text{ and } \ifac,\\
        a=b\in\RR^{n-1}, & \text{for } \iapx,
    \end{cases}
\end{align}
with the off-signal coordinates equal to zero. Substituting $a=b=w$ and
$u=v=0$ into the reduced GD dynamics of
Section~\ref{sec:rank-1 convergence} gives, in all three models,
\begin{align}
    w^+ = \bigl(1-\eta\norm{w}^2\bigr)w+\eta w
        = \bigl(1+\eta-\eta\norm{w}^2\bigr)w,
    \label{eq:balanced_update}
\end{align}
and the off-signal coordinates stay at zero. The squared signal norm therefore
evolves autonomously according to
\begin{align}
    s := \norm{w}^2\in[0,2],
    \qquad
    s^+ = s\,(1+\eta-\eta s)^2 .
    \label{eq:s_update}
\end{align}
We call $s$ the \emph{product coordinate}: it is $ab$ in the scalar and
factorization models and $\langle a,b\rangle$ in the approximation model, so that
$s=0$ is the origin and $s=1$ is the global-minimizer fiber inside $K_2$.

For the one-dimensional dynamics it is convenient to rescale $s$ so that the
recursion becomes a standard cubic. Put
\begin{align}
    m:=(1+\eta)^2,
    \qquad
    y:=\frac{\eta}{1+\eta}\,s,
    \qquad
    J_\eta:=\Bigl[0,\tfrac{2\eta}{1+\eta}\Bigr],
    \label{eq:y_definition}
\end{align}
so that $s\in[0,2]$ corresponds to $y\in J_\eta$. A direct substitution turns
Eq.~\eqref{eq:s_update} into
\begin{align}
    y^+ = h_\eta(y) := m\,y\,(1-y)^2 .
    \label{eq:h_eta}
\end{align}
The global-minimizer fiber $s=1$ becomes the point
\begin{align}
    y_\star:=\frac{\eta}{1+\eta},
\end{align}
and the post-critical window of interest is
\begin{align}
    1<\eta<\sqrt5-1
    \quad\Longleftrightarrow\quad
    4<m<5 .
\end{align}
This is the window in which $h_\eta$ has an attracting $2$-cycle; see
Subsection~\ref{subsec:stable_2_period}.

We record the reduction as a semiconjugacy, since that is the form in which
Appendix~\ref{sec:terminal_inheritance} uses it. Define
\begin{align}
    \pi_\eta(x):=\frac{\eta}{1+\eta}\,s(x),
    \qquad
    s(x):=
    \begin{cases}
        ab, & \text{scalar and rank-1 factorization},\\
        \langle a,b\rangle, & \text{rank-1 approximation}.
    \end{cases}
    \label{eq:pi_eta}
\end{align}
On $K_2$ the function $s(\cdot)$ agrees with the coordinate $s$ of
Eq.~\eqref{eq:s_update}, of which it is the extension to the whole space.

\begin{proposition}[Terminal semiconjugacy]
\label{prop:terminal_cubic_semiconjugacy}
For each of the three models and every $0<\eta<2$,
\begin{align}
    \pi_\eta\circ\gd_\eta = h_\eta\circ\pi_\eta
    \qquad\text{on } K_2,
\end{align}
and $\pi_\eta(K_2)=J_\eta$.
\end{proposition}

\begin{proof}
Let $x\in K_2$. By Proposition~\ref{prop:terminal_sets}, $x$ is balanced with
signal variable $w$ and $\norm{w}^2\le2$, so $s(x)=\norm{w}^2$ and
$\pi_\eta(x)=\eta\norm{w}^2/(1+\eta)\in J_\eta$; conversely every value in
$J_\eta$ is attained.

Equation~\eqref{eq:balanced_update} shows directly that $\gd_\eta(x)$ is
balanced, with signal variable
$\bigl(1+\eta-\eta s(x)\bigr)w$ and zero off-signal coordinates. Hence
$s(\gd_\eta(x))=s(x)\bigl(1+\eta-\eta\,s(x)\bigr)^2$, and
\begin{align}
    \pi_\eta(\gd_\eta(x))
    =
    \frac{\eta}{1+\eta}\,s(x)\bigl(1+\eta-\eta\,s(x)\bigr)^2
    =
    m\cdot\frac{\eta\,s(x)}{1+\eta}\Bigl(1-\frac{\eta\,s(x)}{1+\eta}\Bigr)^2
    =
    h_\eta(\pi_\eta(x)),
\end{align}
where the middle equality multiplies and divides by $(1+\eta)^2=m$.
\end{proof}

The next lemma records the elementary geometry of $h_\eta$ that the period
analysis needs. Note that $J_\eta$ contains $[0,1]$ exactly when $\eta\ge1$, so
in the post-critical window the invariant interval sticks out past $1$; the
lemma shows that this excess region is dynamically irrelevant.

\begin{lemma}[Invariant interval and location of periodic points]
\label{lem:normalized_cubic_range}
Let $0<\eta<2$. Then:

(1) $h_\eta'(y)=m(1-y)(1-3y)$, so the only interior critical point of $h_\eta$
in $[0,1]$ is $c=1/3$, with $h_\eta(c)=4m/27$ and $h_\eta''(c)=-2m\ne0$;

(2) $h_\eta(J_\eta)\subseteq J_\eta$;

(3) if moreover $\eta<\sqrt5-1$, then $h_\eta([0,1])\subseteq[0,1]$, every
$y\in J_\eta$ with $y>1$ satisfies $h_\eta(y)<1$, and consequently every
periodic point of $h_\eta$ in $J_\eta$ lies in $[0,1)$;

(4) the fixed points of $h_\eta$ in $J_\eta$ are exactly $0$ and $y_\star$,
with multipliers
\begin{align}
    h_\eta'(0)=m=(1+\eta)^2>1,
    \qquad
    h_\eta'(y_\star)=1-2\eta ;
    \label{eq:fixed_point_multipliers}
\end{align}

(5) $Sh_\eta(y)<0$ for $y\in(0,1)\setminus\{1/3\}$, where $S$ denotes the
Schwarzian derivative. In particular $h_\eta$ is $S$-unimodal on $[0,1]$ with a
non-flat critical point.
\end{lemma}

\begin{proof}
(1) Differentiating Eq.~\eqref{eq:h_eta} gives
$h_\eta'(y)=m(1-y)^2-2my(1-y)=m(1-y)(1-3y)$, which vanishes on $[0,1]$ at
$y=1/3$ and at the endpoint $y=1$; and $h_\eta(1/3)=m\cdot\frac13\cdot\frac49=4m/27$,
$h_\eta''(y)=2m(3y-2)$.

(2) The map $h_\eta$ and the interval $J_\eta$ are the same in all three models,
so it is enough to argue in the scalar model. There $\isc$ satisfies the
axioms with $\delta_{\mathrm{th}}=\eta$ for every $\eta\in(0,2)$ by
Lemma~\ref{lem:concrete_axioms}(1)--(2). Then
\begin{align}
    h_\eta(J_\eta)
    =
    h_\eta\bigl(\pi_\eta(K_2)\bigr)
    =
    \pi_\eta\bigl(\gd_\eta(K_2)\bigr)
    \subseteq
    \pi_\eta(K_2)
    =
    J_\eta ,
\end{align}
using Proposition~\ref{prop:terminal_cubic_semiconjugacy} for the three
equalities and Lemma~\ref{lem:terminal_forward_invariance} for the inclusion.

(3) For $m<5$ the maximum of $h_\eta$ on $[0,1]$ is $4m/27<20/27<1$, and
$h_\eta\ge0$ there, so $h_\eta([0,1])\subseteq[0,1]$. The set $J_\eta\cap(1,\infty)$ is
nonempty only for $\eta>1$, and there both factors of $h_\eta'$ are negative, so
$h_\eta$ is increasing on $[1,2\eta/(1+\eta)]$ and attains its maximum at the
right endpoint of $J_\eta$.
Since $1<\eta<\sqrt5-1<\tfrac54$, we have
$(\eta-1)^2<\tfrac1{16}$ and $2\eta<\tfrac52$, whence
\begin{align}
    h_\eta\Bigl(\frac{2\eta}{1+\eta}\Bigr)
    =\frac{2\eta(\eta-1)^2}{1+\eta}
    <2\eta(\eta-1)^2
    <\frac52\cdot\frac1{16}
    =\frac5{32}<1,
\end{align}
Hence any $y\in J_\eta$ with
$y>1$ is mapped into $[0,1)$ and, by the first sentence, never returns above
$1$; such a $y$ cannot be periodic. Since $h_\eta(1)=0$, the point $1$ is not
periodic either.

(4) $h_\eta(y)=y$ with $y\ne0$ forces $m(1-y)^2=1$, that is $y=1\mp1/(1+\eta)$.
The root $y=\eta/(1+\eta)=y_\star$ lies in $J_\eta$, while
$y=(2+\eta)/(1+\eta)$ exceeds $2\eta/(1+\eta)$ because $\eta<2$. The
multipliers follow from (1): $h_\eta'(0)=m$ and
$h_\eta'(y_\star)=m(1-y_\star)(1-3y_\star)=(1+\eta)\cdot\frac{1-2\eta}{1+\eta}=1-2\eta$.

(5) With $h_\eta'''\equiv6m$ and the derivatives from (1),
\begin{align}
    Sh_\eta(y)
    =
    \frac{h_\eta'''(y)}{h_\eta'(y)}
    -\frac32\Bigl(\frac{h_\eta''(y)}{h_\eta'(y)}\Bigr)^2
    =
    -\frac{6(6y^2-8y+3)}{(1-4y+3y^2)^2},
\end{align}
and $6y^2-8y+3$ has discriminant $-8<0$, hence is strictly positive. The
critical point is non-flat because $h_\eta''(1/3)=-2m\ne0$.
\end{proof}

\subsection{Period-two structure of the normalized cubic family}
\label{subsec:stable_2_period}

Everything about periods $1$ and $2$ follows from one factorization. What
separates the two regimes is visible already in
Lemma~\ref{lem:normalized_cubic_range}(4): the multiplier at $y_\star$ is
$1-2\eta$, which passes through $-1$ at the critical step size $\eta=1$, where
$h_\eta$ undergoes a period doubling. With $m=(1+\eta)^2$, expanding both sides
gives
\begin{align}
    h_\eta^2(y)-y
    =
    y\,
    \bigl((1+\eta)y-\eta\bigr)\,
    \bigl((1+\eta)y-(2+\eta)\bigr)\,
    \bigl(1-my(1-y)\bigr)\,
    Q_m(y),
    \label{eq:h2_factorization}
\end{align}
where
\begin{align}
    Q_m(y):=1+m(1-y)-m^2y(1-y)^3 .
    \label{eq:Q_m}
\end{align}

The first three factors of Eq.~\eqref{eq:h2_factorization} encode the three
fixed points of $h_\eta$, namely $0$, $y_\star=\eta/(1+\eta)$, and
$(2+\eta)/(1+\eta)$. The quadratic factor $1-my(1-y)$ and the quartic $Q_m$
contain the remaining possible roots of $h_\eta^2(y)-y$; the next lemma rules
out the quartic factor on the interval relevant below.

\begin{lemma}
\label{lem:Q_m_positive}
Let $0<\eta<\sqrt5-1$, so that $m<5$. Then $Q_m(y)\ge1$ for $y\in[0,1]$, with
equality only at $y=1$.
\end{lemma}

\begin{proof}
Substituting $t:=1-y\in[0,1]$ into Eq.~\eqref{eq:Q_m} gives
\begin{align}
    Q_m(1-t)
    =
    1+mt-m^2(1-t)t^3
    =
    1+mt\bigl[1-m(1-t)t^2\bigr].
\end{align}
On $[0,1]$ the function $t\mapsto(1-t)t^2$ has maximum $4/27$, attained at
$t=2/3$. Hence $m(1-t)t^2<5\cdot\frac4{27}=\frac{20}{27}<1$, so the bracket is
positive and $Q_m(1-t)\ge1$, with equality exactly when $t=0$.
\end{proof}

\begin{proposition}
\label{prop:subcritical_no_period_two}
Let $0<\eta<1$. Then $h_\eta$ has no point of minimal period $2$ in $J_\eta$.
\end{proposition}

\begin{proof}
For $\eta<1$ we have $J_\eta\subset[0,1)$, so
Lemma~\ref{lem:Q_m_positive} gives $Q_m>0$ on $J_\eta$, and
Lemma~\ref{lem:normalized_cubic_range}(4) puts the root $(2+\eta)/(1+\eta)$ of
the third factor outside $J_\eta$. By Eq.~\eqref{eq:h2_factorization}, a point
of $J_\eta$ fixed by $h_\eta^2$ is therefore either a fixed point of $h_\eta$
or a root of $1-my(1-y)=my^2-my+1$. The latter has discriminant
$m^2-4m=m(m-4)$, which is negative because $m=(1+\eta)^2<4$. Hence no such
root is real.
\end{proof}

\begin{theorem}[\cite{coppel1955solution}]
\label{thm:coppel}
If a continuous map $f\colon[a,b]\to[a,b]$ has no point of minimal period $2$,
then $f^t(z)$ converges to a fixed point of $f$ for every $z\in[a,b]$.
\end{theorem}

\begin{corollary}[Reduced asymptotics in the pre-critical regime]
\label{cor:subcritical_reduced_convergence}
Let $0<\eta<1$. Then every orbit of $h_\eta$ in $J_\eta$ converges to $0$ or to
$y_\star$. Moreover $0$ is a repelling fixed point and $y_\star$ is attracting.
\end{corollary}

\begin{proof}
Lemma~\ref{lem:normalized_cubic_range}(2) makes $h_\eta$ a continuous self-map
of the compact interval $J_\eta$, and
Proposition~\ref{prop:subcritical_no_period_two} removes the least-period-$2$
points, so Theorem~\ref{thm:coppel} applies and every orbit converges to a
fixed point; by Lemma~\ref{lem:normalized_cubic_range}(4) the fixed points in
$J_\eta$ are $0$ and $y_\star$. The multipliers in
Eq.~\eqref{eq:fixed_point_multipliers} are $m>1$ at $0$ and $|1-2\eta|<1$ at
$y_\star$ for $\eta\in(0,1)$.
\end{proof}

\begin{proposition}[Unique attracting $2$-cycle in the post-critical regime]
\label{prop:postcritical_unique_period_two}
Let $1<\eta<\sqrt5-1$. Then $h_\eta$ has exactly one orbit of minimal period $2$
in $J_\eta$, namely $\Lambda_\eta:=\{y_-,y_+\}$ with
\begin{align}
    y_\pm
    =
    \frac12
    \Bigl(
    1\pm\frac{\sqrt{(\eta-1)(\eta+3)}}{1+\eta}
    \Bigr),
    \qquad
    0<y_-<\tfrac12<y_+<1 ,
    \label{eq:y_pm}
\end{align}
or, in the product coordinate,
\begin{align}
    s_\pm=\frac{1+\eta}{\eta}\,y_\pm
    =\frac{1+\eta\pm\sqrt{(\eta-1)(\eta+3)}}{2\eta},
    \qquad
    \tfrac12<s_-<s_+<2 .
    \label{eq:s_pm}
\end{align}
Its multiplier is
\begin{align}
    \mu_\eta:=(h_\eta^2)'(y_\pm)=7-4\eta-2\eta^2,
    \label{eq:two_cycle_multiplier}
\end{align}
and $|\mu_\eta|<1$ precisely on $1<\eta<\sqrt5-1$, so $\Lambda_\eta$ is
attracting.
\end{proposition}

\begin{proof}
By Lemma~\ref{lem:normalized_cubic_range}(3) every periodic point lies in
$[0,1)$, where Lemma~\ref{lem:Q_m_positive} gives $Q_m>0$ and the third factor
of Eq.~\eqref{eq:h2_factorization} does not vanish. So a point of minimal period
$2$ must be a root of $my^2-my+1$, whose roots are
$\frac12\bigl(1\pm\sqrt{1-4/m}\bigr)$; substituting $m=(1+\eta)^2$ gives
Eq.~\eqref{eq:y_pm}, and both roots are real and distinct because $m>4$. They
lie in $(0,1)$ because $(\eta-1)(\eta+3)<(1+\eta)^2$. Since
$my_\pm(1-y_\pm)=1$, we have $h_\eta(y_\pm)=1-y_\pm=y_\mp$, so the two roots
form a single $2$-cycle. Equation~\eqref{eq:s_pm} follows by inverting the substitution
$y=\eta s/(1+\eta)$ of Eq.~\eqref{eq:y_definition}.
The bound $s_->\tfrac12$ is equivalent to $\sqrt{(\eta-1)(\eta+3)}<1$, that is
$\eta<\sqrt5-1$; and $s_+<2$ because $y_+<1$ gives
$s_+=(1+\eta)y_+/\eta<(1+\eta)/\eta\le2$ for $\eta\ge1$.

For the multiplier, $h_\eta'(y)=m(1-y)(1-3y)$ and
$(h_\eta^2)'(y_-)=h_\eta'(y_-)h_\eta'(y_+)$; expanding this product gives
Eq.~\eqref{eq:two_cycle_multiplier}. Finally $\mu_\eta$ is strictly decreasing in
$\eta$, since $\mu_\eta'=-4-4\eta<0$, with $\mu_1=1$ and
$\mu_{\sqrt5-1}=-1$; hence $|\mu_\eta|<1$ exactly on the stated window.
\end{proof}

\begin{remark}[Reduced and full-state cycles]
The proposition identifies a $2$-cycle of the product coordinate, not yet of
the dynamics on $K_2$. Indeed, $w\mapsto\norm{w}^2$ is not injective, so each
of $s_\pm$ has a fiber of more than one point: the pair $\pm w$ for $\isc$ and
$\ifac$, and a sphere for $\iapx$. The action of $\gd_\eta$ on these fibers is
described in Appendix~\ref{sec:terminal_inheritance},
Lemma~\ref{lem:signed_cycle_split}.
\end{remark}

\subsection{Absence of period four and the exact reduced asymptotics}
\label{subsec:no_period_four}

In this subsection we verify that $h_\eta$ has no orbit of minimal period $4$.
Through Corollary~\ref{cor:postcritical_periods}, this exclusion is used in
Appendix~\ref{sec:terminal_inheritance} to rule out homoclinic points; see
Corollary~\ref{cor:no_homoclinic_reduced}.

\begin{lemma}
\label{lem:unique_nonrepelling_cycle}
Let $1<\eta<\sqrt5-1$. Then every periodic orbit of $h_\eta$ in $J_\eta$ other
than $\Lambda_\eta$ is hyperbolic repelling.
\end{lemma}

\begin{proof}
By Lemma~\ref{lem:normalized_cubic_range}(3) and~(5), $h_\eta$ restricts to a
$C^3$ self-map of $[0,1]$ with $Sh_\eta<0$ and the single interior critical
point $c=1/3$, and all periodic orbits in $J_\eta$ lie in $[0,1)$. For such
maps, Singer's theorem \citep[Theorem II.6.1]{demelo1993onedimensional} gives
two facts: (1) the
immediate basin of an attracting periodic orbit contains a critical point or a
boundary point of the interval, and (2) every neutral periodic point is attracting.

The orbit $\Lambda_\eta$ is attracting by
Proposition~\ref{prop:postcritical_unique_period_two}. Its immediate basin
cannot contain a boundary point of $[0,1]$: the point $0$ is a repelling fixed
point by Eq.~\eqref{eq:fixed_point_multipliers}, and $h_\eta(1)=0$, so neither
$0$ nor $1$ is attracted to $\Lambda_\eta$. Hence the immediate basin of
$\Lambda_\eta$ contains the unique interior critical point $c=1/3$.

Suppose some periodic orbit $\mathcal O\ne\Lambda_\eta$ were not repelling. A
neutral orbit is attracting by (2), so $\mathcal O$ is attracting in either
case, and (1) places a critical point or a boundary point in its immediate
basin. Boundary points are excluded as above, so $c$ would lie in
the immediate basin of $\mathcal O$ as well. Immediate basins of distinct
attracting orbits are disjoint, a contradiction. Hence every periodic orbit in
$J_\eta$ other than $\Lambda_\eta$ is hyperbolic repelling. In particular, at
every point $y\in J_\eta$ of minimal period $4$,
$\partial_y h_\eta^4(y)\ne 1$.
\end{proof}

\begin{proposition}
\label{prop:no_period_four}
For every $1<\eta<\sqrt5-1$, the map $h_\eta$ has no point of minimal period $4$
in $J_\eta$.
\end{proposition}

\begin{proof}
Let
\begin{align}
    \mathcal P
    :=
    \bigl\{\eta\in(1,\sqrt5-1) :
    h_\eta \text{ has a point of minimal period } 4 \text{ in }J_\eta\bigr\}.
\end{align}
We show that $\mathcal P$ is open and closed in the connected interval
$(1,\sqrt5-1)$ and that it is not everything.

\emph{(1) $\mathcal P$ is open}. Let $\eta_0\in\mathcal P$ with a least-period-$4$
point $y_0\in J_{\eta_0}$. By
Lemma~\ref{lem:unique_nonrepelling_cycle} that orbit is hyperbolic repelling, so
$\partial_y h_\eta^4(y)\ne1$ at $(y_0,\eta_0)$, and the implicit
function theorem continues $y_0$ to a nearby solution
$y(\eta)$ of $h_\eta^4(y)=y$ for $\eta$ near $\eta_0$. By
Lemma~\ref{lem:normalized_cubic_range}(3) and the minimal-period assumption,
$y_0\in(0,1)$; after shrinking the parameter neighborhood inside
$(1,\sqrt5-1)$, continuity gives $y(\eta)\in(0,1)\subset J_\eta$. Because
$h_{\eta_0}^j(y_0)\ne y_0$ for $j=1,2,3$, these inequalities persist by
continuity, so $y(\eta)$ has minimal period $4$ for $\eta$ close to $\eta_0$.

\emph{(2) $\mathcal P$ is closed in $(1,\sqrt5-1)$}. Let $\eta_k\in\mathcal P$ with
$\eta_k\to\eta_\ast\in(1,\sqrt5-1)$, and let $y_k\in J_{\eta_k}$ be a
least-period-$4$ point of $h_{\eta_k}$. By
Lemma~\ref{lem:normalized_cubic_range}(3), every $y_k$ lies in the compact set
$[0,1]$, so after passing to a subsequence $y_k\to y_\ast$ with
$h_{\eta_\ast}^4(y_\ast)=y_\ast$. If $y_\ast$ has minimal period $4$, then
$\eta_\ast\in\mathcal P$ and we are done. Otherwise $y_\ast$ is a fixed point of
$h_{\eta_\ast}^2$, hence, by Eq.~\eqref{eq:h2_factorization} together with
Lemma~\ref{lem:Q_m_positive} and
Proposition~\ref{prop:postcritical_unique_period_two}, one of
$0$, $y_\star$, $y_-$, $y_+$. Each of these is a hyperbolic fixed point
of $h_{\eta_\ast}^4$, by
Lemma~\ref{lem:normalized_cubic_range}(4) at $0$ and $y_\star$ and by
Proposition~\ref{prop:postcritical_unique_period_two} at $y_\pm$.
Each of the four candidates extends to a branch that is smooth in $\eta$ near
$\eta_\ast$ --- the constant $0$, $y_\star(\eta)=\eta/(1+\eta)$, and
$y_\pm(\eta)$ of Eq.~\eqref{eq:y_pm} --- and every point of such a branch is
fixed by $h_\eta^2$, hence by $h_\eta^4$. By hyperbolicity, the implicit
function theorem applied to
$h_\eta^4(y)-y$ at $(y_\ast,\eta_\ast)$ makes that branch the unique solution of
$h_\eta^4(y)=y$ in a neighborhood of $(y_\ast,\eta_\ast)$. For large $k$ the
pair $(y_k,\eta_k)$ lies in that neighborhood, so $y_k$ lies on the branch and
has minimal period at most $2$, a contradiction. Hence $\eta_\ast\in\mathcal P$.

Now, we show that $\mathcal P\ne(1,\sqrt5-1)$. At the rational parameter $\eta_0=11/10 \in (1,\sqrt5-1)$,
the polynomial $h_{\eta_0}^2(y)-y$ divides $h_{\eta_0}^4(y)-y$ exactly, and the
quotient
\begin{align}
    R(y):=\frac{h_{\eta_0}^4(y)-y}{h_{\eta_0}^2(y)-y}
    \label{eq:period4_quotient}
\end{align}
has degree $72$, and Sturm's theorem
\citep[Theorem~2.50]{basu2006algorithms}, applied in exact rational
arithmetic,
shows that $R$ has no real root in $J_{\eta_0}$. Every least-period-$4$ point in
$J_{\eta_0}$ would be a root of $R$ in $[0,1)\subset J_{\eta_0}$, so
$\eta_0\notin\mathcal P$.

Being open, closed, and proper in a connected interval, $\mathcal P$ is empty.
\end{proof}

\begin{corollary}
\label{cor:postcritical_periods}
For every $1<\eta<\sqrt5-1$, the set of minimal periods of $h_\eta$ on $J_\eta$
is contained in $\{1,2\}$.
\end{corollary}

\begin{proof}
In the Sharkovsky ordering \citep[Theorem 2.8]{elaydi2007discrete}, every integer
other than $1$, $2$, and $4$ precedes $4$. Thus a point of minimal period
$n\notin\{1,2,4\}$ would force a point of minimal period $4$, while $n=4$ is
excluded directly by Proposition~\ref{prop:no_period_four}.
\end{proof}

\begin{corollary}[Exact reduced asymptotics in the stable post-critical window]
\label{cor:postcritical_exact_orbits}
Let $1<\eta<\sqrt5-1$. Then the $\omega$-limit set of every orbit of $h_\eta$ in
$J_\eta$ is $\{0\}$, $\{y_\star\}$, or $\Lambda_\eta$.
\end{corollary}

\begin{proof}
By Lemma~\ref{lem:normalized_cubic_range}(3) every orbit enters $[0,1]$ after at
most one step and remains there, so we may work with
$h_\eta^2\colon[0,1]\to[0,1]$. A point of minimal period $2$ for $h_\eta^2$ would
have minimal period $4$ for $h_\eta$, which Corollary~\ref{cor:postcritical_periods}
excludes; hence Theorem~\ref{thm:coppel} applies to $h_\eta^2$ and every orbit
of $h_\eta^2$ converges to a fixed point of $h_\eta^2$. By
Eq.~\eqref{eq:h2_factorization}, Lemma~\ref{lem:Q_m_positive}, and
Proposition~\ref{prop:postcritical_unique_period_two}, those fixed points are
$0$, $y_\star$, $y_-$, and $y_+$. If $h_\eta^{2t}(y)\to p$ then
$h_\eta^{2t+1}(y)\to h_\eta(p)$, so the $h_\eta$-orbit accumulates exactly on
$\{p,h_\eta(p)\}$, which is $\{0\}$, $\{y_\star\}$, or $\Lambda_\eta$.
\end{proof}

%% file: apxH_terminal_inheritance.tex

\section{Inheritance of the terminal dynamics}
\label{sec:terminal_inheritance}

Appendix~\ref{sec:convergence when delta_* = 2} classifies the asymptotic
behavior of the reduced map $h_\eta$ on the terminal set $K_2$. This
classification does not immediately transfer to the full gradient-descent
dynamics: when $\delta_\ast=2$, a full trajectory only satisfies
$\dist{x_t}{K_2}\to0$ and need not lie on $K_2$ at any finite time. Instead of
tracking finite-time orbits, we pass to the $\omega$-limit set. By Lemma~2.1 of
\citet{hirsch2001chain}, the $\omega$-limit set of a precompact forward orbit is
internally chain transitive. Since every accumulation point lies in $K_2$, the
exact semiconjugacy on $K_2$ projects this set to a compact invariant internally
chain transitive set of $h_\eta$. Thus the inheritance problem reduces to
classifying the possible internally chain transitive sets of the reduced map.

We carry this out in four steps. First, we show that internal chain transitivity
passes through the terminal semiconjugacy. Second, we study the structure of
invariant sets near the repelling fixed points of $h_\eta$. Third, we combine
these structural properties with the orbit classification of
Appendix~\ref{sec:convergence when delta_* = 2} to classify all compact
invariant internally chain transitive sets of $h_\eta$. Fourth, we lift the
resulting reduced alternatives back to the full gradient-descent dynamics.

Theorems~\ref{thm:reduced_ict_classification_subcritical}
and~\ref{thm:reduced_ict_classification_postcritical} give this reduced
classification. Conditional on
$\delta_\ast=2$, Proposition~\ref{prop:terminal_omega_classification} lifts this
reduced classification to the full $\omega$-limit set for every initialization
in the full-measure set furnished by
Theorem~\ref{thm:abstract_state_dependent_convergence}.
Corollary~\ref{cor:subcritical_terminal_inheritance} and
Theorem~\ref{thm:terminal_period_two_inheritance} then use
Theorem~\ref{thm:abstract_state_dependent_convergence}(4) to discard the
stationary alternatives outside null sets. In the post-critical scalar and
rank-$1$ factorization models,
Theorem~\ref{thm:abstract_state_dependent_convergence}(5) further gives
$\delta_\ast=2$ for almost every certified initialization, yielding
Corollary~\ref{cor:theorem_five_period_two_inheritance}. Rank-$1$ approximation
has a narrower scope: its certificate framework is verified only for
$0<\eta<1$, and its non-discrete fibers mean that internal chain transitivity
alone need not select a single full-state point or signal direction; see
Remark~\ref{rmk:apx_no_postcritical_inheritance}. We close by computing the
Hessian sharpness along the inherited period-$2$ orbit and comparing it with
$2/\eta$.

\subsection{Limit sets, internal chain transitivity, and projection}
\label{subsec:ict_pushdown}

\begin{definition}[Internal chain transitivity]
\label{def:ict}
Let $f\colon X\to X$ be continuous on a metric space and let $M\subseteq X$ be
compact and invariant, i.e., $f(M)=M$. Given $p,q\in M$ and $\varepsilon>0$, an
\emph{internal $\varepsilon$-chain} from $p$ to $q$ is a finite sequence
$z_0=p,z_1,\dots,z_N=q$ with $z_i\in M$, $N\ge1$, and
$d(f(z_i),z_{i+1})<\varepsilon$ for $0\le i<N$. The set $M$ is
\emph{internally chain transitive} if such a chain exists for all $p,q\in M$
and all $\varepsilon>0$.
\end{definition}

\begin{lemma}[Lemma 2.1 of \citep{hirsch2001chain}]
\label{lem:omega_limit_ict}
The $\omega$-limit set of a precompact forward orbit of a continuous map is
nonempty, compact, and invariant, and it is internally chain transitive.
\end{lemma}

We apply this to the continuous full gradient descent map and the forward orbit
of $x_0$.

\begin{lemma}[The $\omega$-limit set lies in $K_2$]
\label{lem:omega_in_K2}
Let $x_0\in\operatorname{int}K_{\delta_{\mathrm{th}}}$ belong to the
full-measure set on which
Theorem~\ref{thm:abstract_state_dependent_convergence} holds with
$\delta_\ast=2$. Then $\Omega:=\omega(x_0)$ is nonempty, compact, invariant,
and internally chain transitive, and $\Omega\subseteq K_2$.
\end{lemma}

\begin{proof}
By Step 1 of the proof of
Theorem~\ref{thm:abstract_state_dependent_convergence}, the trajectory remains
in the compact set $K_{\delta_{\mathrm{th}}}$, so the orbit is precompact and
Lemma~\ref{lem:omega_limit_ict} gives the first assertion. By
Theorem~\ref{thm:abstract_state_dependent_convergence}(2),
$\dist{x_t}{K_2}\to0$; since $K_2$ is closed, every accumulation point of the
trajectory lies in $K_2$, that is $\Omega\subseteq K_2$.
\end{proof}

The projection now transports all of this structure to the reduced map.

\begin{lemma}[Semiconjugacy preserves internal chain transitivity]
\label{lem:ict_pushdown_semiconjugacy}
Let $\Omega\subseteq K_2$ be nonempty, compact, $\gd_\eta$-invariant, and
internally chain transitive. Then $M:=\pi_\eta(\Omega)\subseteq J_\eta$ is
nonempty, compact, $h_\eta$-invariant, and internally chain transitive.
\end{lemma}

\begin{proof}
Nonemptiness and compactness are immediate from continuity of $\pi_\eta$, and
$M\subseteq\pi_\eta(K_2)=J_\eta$ by
Proposition~\ref{prop:terminal_cubic_semiconjugacy}. Invariance follows from
the same proposition:
\begin{align}
    h_\eta(M)=h_\eta(\pi_\eta(\Omega))=\pi_\eta(\gd_\eta(\Omega))=\pi_\eta(\Omega)=M,
\end{align}
using $\gd_\eta(\Omega)=\Omega$.

For internal chain transitivity, fix $p',q'\in M$ and $\varepsilon>0$, and
choose $p,q\in\Omega$ with $\pi_\eta(p)=p'$ and $\pi_\eta(q)=q'$. The map
$\pi_\eta$ is uniformly continuous on the compact set $\Omega$, so there is
$\varepsilon'>0$ such that $z,z''\in\Omega$ with $\norm{z-z''}<\varepsilon'$
implies $\lvert \pi_\eta(z)-\pi_\eta(z'')\rvert<\varepsilon$. Take an internal
$\varepsilon'$-chain $z_0=p,\dots,z_N=q$ in $\Omega$. Then
$\pi_\eta(z_0),\dots,\pi_\eta(z_N)$ lies in $M$, and for each $i$,
\begin{align}
    \lvert h_\eta(\pi_\eta(z_i))-\pi_\eta(z_{i+1})\rvert
    =
    \lvert \pi_\eta(\gd_\eta(z_i))-\pi_\eta(z_{i+1})\rvert
    <\varepsilon,
\end{align}
where the equality is the semiconjugacy. This is an internal
$\varepsilon$-chain from $p'$ to $q'$ in $M$.
\end{proof}

\subsection{Structure of invariant sets near repelling fixed points}
\label{subsec:repellers_homoclinic}

The orbit classification of
Appendix~\ref{sec:convergence when delta_* = 2} describes the forward
asymptotics of individual points of $h_\eta$. To classify compact invariant
sets, however, we need additional information about how such sets may contain
the repelling fixed points. At the origin, convergence to the repeller forces a
finite-time hit, while invariance rules out every nontrivial return to the
origin. In the post-critical regime, the analogous issue is subtler at the other
repelling fixed point $y_\star$, the reduced point corresponding to the
global-minimizer fiber: a nontrivial eventual preimage of $y_\star$ near
$y_\star$ would generate a backward orbit accumulating at $y_\star$. We exclude
precisely this possibility through the absence of homoclinic points.

\paragraph{The origin.}
This rests on a range computation. Because $h_\eta$
vanishes only at $0$ and $1$, and its range misses $1$, no nonzero point of an
invariant set can reach $0$.

\begin{lemma}
\label{lem:no_return_to_zero}
Let $0<\eta<\sqrt5-1$ and let $M\subseteq J_\eta$ satisfy $h_\eta(M)=M$. If
$w\in M$ and $h_\eta^N(w)=0$ for some $N\ge0$, then $w=0$.
\end{lemma}

\begin{proof}
By Lemma~\ref{lem:normalized_cubic_range}(1) and~(3),
\begin{align}
    h_\eta(J_\eta)\subseteq[0,1).
\end{align}
Since $M=h_\eta(M)$ and $M\subseteq J_\eta$,
\begin{align}
    M=h_\eta(M)\subseteq h_\eta(J_\eta)\subseteq[0,1),
\end{align}
so $1\notin M$.

Now $h_\eta(y)=my(1-y)^2$ vanishes exactly at $y\in\{0,1\}$, so
$h_\eta^{-1}(0)=\{0,1\}$. Suppose $h_\eta^N(w)=0$ with $N\ge1$ chosen minimal.
Then $h_\eta^{N-1}(w)\in h_\eta^{-1}(0)\cap M=\{0,1\}\cap M=\{0\}$, so
$h_\eta^{N-1}(w)=0$, contradicting minimality. Hence $N=0$ and $w=0$.
\end{proof}

The preceding lemma rules out nontrivial finite-time returns to the origin. The
following elementary fact converts convergence to a repelling fixed point into
precisely such a finite-return statement.

\begin{lemma}
\label{lem:repeller_exact_hit}
Let $f\colon\RR\to\RR$ be $C^1$ and let $p$ be a fixed point with
$\lvert f'(p)\rvert>1$. If $f^t(z)\to p$, then $f^N(z)=p$ for some finite
$N$.
\end{lemma}

\begin{proof}
Choose $r>0$ and $q>1$ with $\lvert f'\rvert\ge q$ on $(p-r,p+r)$.
By the convergence of the sequence, we can find $T$ such that $f^t(z)\in(p-r,p+r)$ for all $t\ge T$, and define
$d_t:=\lvert f^t(z)-p\rvert$. For $t\ge T$ the mean value theorem gives
$d_{t+1}=\lvert f(f^t(z))-f(p)\rvert\ge q\,d_t$. If $d_T>0$ then $d_t\ge q^{t-T}d_T\to\infty$,
contradicting $d_t\to0$. Hence $d_T=0$.
\end{proof}

Combining the two gives the reusable consequence for invariant sets on which
every orbit converges to the origin.

\begin{lemma}[Invariant sets converging to the origin]
\label{lem:invariant_set_origin}
Let $0<\eta<\sqrt5-1$ and let $M\subseteq J_\eta$ be nonempty and satisfy
$h_\eta(M)=M$. Suppose that $h_\eta^t(w)\to0$ for every $w\in M$. Then
$M=\{0\}$.
\end{lemma}

\begin{proof}
Fix $w\in M$. Since $0$ is a repelling fixed point of $h_\eta$,
Lemma~\ref{lem:repeller_exact_hit} gives $h_\eta^N(w)=0$ for some finite $N$.
By Lemma~\ref{lem:no_return_to_zero}, $w=0$. Thus every point of $M$ is the
origin, so $M=\{0\}$.
\end{proof}

\paragraph{The nonzero repelling fixed point.}
Unlike the origin, $y_\star$ may have nontrivial preimages in $J_\eta$, so the
preceding range argument does not apply. Following \citet{block1978homoclinic},
for a fixed point $p$ of an interval map $f$, set
\begin{align}
    W^u(p,f)
    :=
    \bigcap_{V\ni p}\ \bigcup_{n\ge1} f^n(V),
\end{align}
where the intersection is over neighborhoods of $p$ in the interval. A point
$z\ne p$ is a \emph{homoclinic point} of $p$ if $z\in W^u(p,f)$ and
$f^N(z)=p$ for some $N\ge1$. In particular, a backward orbit
$z=z_0,z_{-1},z_{-2},\ldots$ satisfying
$f(z_{-(k+1)})=z_{-k}$ and $z_{-k}\to p$ places $z$ in $W^u(p,f)$.

\begin{corollary}
\label{cor:no_homoclinic_reduced}
Let $1<\eta<\sqrt5-1$. Then $h_\eta$ has no homoclinic point in $J_\eta$.
\end{corollary}

\begin{proof}
By Theorem A of \citet{block1978homoclinic}, a continuous self-map of a closed
interval has a homoclinic point if and only if it has a periodic point whose
minimal period is not a power of $2$. By
Corollary~\ref{cor:postcritical_periods} the minimal periods of $h_\eta$ are
contained in $\{1,2\}$, all of which are powers of $2$; the contrapositive gives
the claim.
\end{proof}

These facts together isolate the repelling fixed point $y_\star$ inside a
compact invariant set that avoids the attracting $2$-cycle.

\begin{lemma}
\label{lem:ystar_isolated_invariant_set}
Let $1<\eta<\sqrt5-1$ and let $M\subseteq J_\eta$ be compact with
$h_\eta(M)=M$, $M\cap\Lambda_\eta=\varnothing$, and $y_\star\in M$. Then
$y_\star$ is isolated in $M$.
\end{lemma}

\begin{proof}
Since $\lvert h_\eta'(y_\star)\rvert=2\eta-1>1$ and
$h_\eta'(y_\star)\ne0$, the inverse function theorem provides a local inverse
branch $\varphi$ of $h_\eta$ near $y_\star$ with
$\varphi(y_\star)=y_\star$ and
$\lvert\varphi'(y_\star)\rvert=1/(2\eta-1)<1$. By continuity of $\varphi'$, we
may choose $r>0$ and $q>1$ such that
\begin{align}
    V_\star:=(y_\star-r,y_\star+r)\subseteq\operatorname{int}J_\eta,
    \qquad 0\notin V_\star,
    \qquad \sup_{y\in V_\star}\lvert\varphi'(y)\rvert\le q^{-1},
\end{align}
and $\varphi$ is defined on $V_\star$. Since $\varphi(y_\star)=y_\star$, the mean
value theorem gives
$\lvert\varphi(y)-y_\star\rvert\le q^{-1}\lvert y-y_\star\rvert<r$ for every
$y\in V_\star$. Thus $\varphi(V_\star)\subseteq V_\star$, and $\varphi$ is a
contraction on $V_\star$.
If $y_\star$ is not isolated in $M$, we can find $w\in M$ satisfying
$0<\lvert w-y_\star\rvert<r$.

The forward orbit of $w$ stays in $M$, and by
Corollary~\ref{cor:postcritical_exact_orbits} its $\omega$-limit set is
$\{0\}$, $\{y_\star\}$, or $\Lambda_\eta$. The last possibility is impossible,
since the $\omega$-limit set lies in the closed set $M$ and
$M\cap\Lambda_\eta=\varnothing$. If the $\omega$-limit set were $\{0\}$, then
$h_\eta^t(w)\to0$, so Lemma~\ref{lem:repeller_exact_hit} applied at $0$ would
give $h_\eta^N(w)=0$ for some $N$, and Lemma~\ref{lem:no_return_to_zero} would
force $w=0$, contradicting $0\notin V_\star$. Hence the
$\omega$-limit set must be $\{y_\star\}$, i.e., $h_\eta^t(w)\to y_\star$, and
Lemma~\ref{lem:repeller_exact_hit} gives $h_\eta^N(w)=y_\star$ for some finite
$N$.

Because $\varphi$ is a contraction fixing $y_\star$,
$\varphi^k(w)\to y_\star$.
Since $\varphi$ is a local inverse branch of $h_\eta$,
$\{\varphi^k(w)\}_{k\ge0}$ is a backward orbit of $w$ converging to $y_\star$,
i.e., $w\in W^u(y_\star,h_\eta)$.
Together with $h_\eta^N(w)=y_\star$, this makes $w$ a homoclinic point of
$y_\star$, contradicting Corollary~\ref{cor:no_homoclinic_reduced}. This
contradiction proves the claim.
\end{proof}

\subsection{Classification of compact invariant internally chain transitive sets}
\label{subsec:ict_classification}

We now combine the pointwise orbit classification from
Appendix~\ref{sec:convergence when delta_* = 2} with the invariant-set
structure established in the previous subsection. Internal chain transitivity
supplies two additional mechanisms. An attracting invariant object admits a
trapping neighborhood, which confines any internally chain transitive set that
meets it. Conversely, once a fixed point is isolated inside the invariant set,
internal chains cannot leave that point. We record these two elementary
principles before classifying the possible reduced limit sets.

\begin{lemma}
\label{lem:frozen_chain}
Let $f\colon X\to X$ be continuous on a metric space $(X,d)$, let $M\subseteq X$
be compact, invariant, and internally chain transitive, and let $p\in M$ be a
fixed point of $f$ that is isolated in $M$. Then
$M=\{p\}$.
\end{lemma}

\begin{proof}
Since $p$ is isolated in $M$, there is $r>0$ such that $M\cap B(p,r)=\{p\}$.
Let $z\in M$ and take an internal $\varepsilon$-chain
$p=z_0,z_1,\dots,z_N=z$ in $M$ with $\varepsilon<r$. Since $f(p)=p$, the first
step gives $d(z_1,p)=d(z_1,f(z_0))<\varepsilon<r$, so
$z_1\in M\cap B(p,r)=\{p\}$, that is $z_1=p$. Inducting along the chain gives
$z_N=p$, so $z=p$.
\end{proof}

\begin{lemma}
\label{lem:trapping_confines_chains}
Let $f\colon X\to X$ be continuous on a metric space $(X,d)$, let $U\subseteq X$
be open, precompact, and $f(\overline U)\subseteq U$, and let
$M\subseteq X$ be compact, invariant, and internally chain transitive with
$M\cap U\ne\varnothing$. Then $M\subseteq U$.
\end{lemma}

\begin{proof}
If $U=X$, the conclusion is immediate. Otherwise set
$\gamma:=\dist{f(\overline U)}{X\setminus U}>0$, which is positive because
$f(\overline U)$ is compact and contained in the open set $U$. Fix
$p\in M\cap U$ and $z\in M$, and take an internal $\varepsilon$-chain
$p=z_0,\dots,z_N=z$ in $M$ with $\varepsilon<\gamma$. If $z_i\in U$, then
$f(z_i)\in f(\overline U)$ and $d(z_{i+1},f(z_i))<\gamma$, so $z_{i+1}\in U$.
By induction $z=z_N\in U$.
\end{proof}

\begin{theorem}[Reduced classification in the pre-critical regime]
\label{thm:reduced_ict_classification_subcritical}
Let $0<\eta<1$ and let $M\subseteq J_\eta$ be nonempty, compact,
$h_\eta$-invariant, and internally chain transitive. Then
$M=\{0\}$ or $M=\{y_\star\}$.
\end{theorem}

\begin{proof}
\emph{Case 1: $y_\star\in M$.} By
Corollary~\ref{cor:subcritical_reduced_convergence} the fixed point $y_\star$ is
attracting, so for every $\rho>0$ there is an open interval
$U\subseteq(y_\star-\rho,y_\star+\rho)$ containing $y_\star$ with
$h_\eta(\overline U)\subseteq U$. Lemma~\ref{lem:trapping_confines_chains} gives
$M\subseteq U$, and letting $\rho\downarrow0$ gives $M=\{y_\star\}$.

\emph{Case 2: $y_\star\notin M$.} Every $w\in M$ has its forward orbit in $M$, and
by Corollary~\ref{cor:subcritical_reduced_convergence} that orbit converges to
$0$ or to $y_\star$. Since the limit lies in the closed set $M$ and
$y_\star\notin M$, every orbit in $M$ converges to $0$.
Lemma~\ref{lem:invariant_set_origin} therefore gives $M=\{0\}$.
\end{proof}

\begin{theorem}[Reduced classification in the post-critical regime]
\label{thm:reduced_ict_classification_postcritical}
Let $1<\eta<\sqrt5-1$ and let $M\subseteq J_\eta$ be nonempty, compact,
$h_\eta$-invariant, and internally chain transitive. Then
$M=\{0\}$, $M=\{y_\star\}$, or $M=\Lambda_\eta$.
\end{theorem}

\begin{proof}

\emph{Case 1: $M\cap\Lambda_\eta\ne\varnothing$.} The orbit $\Lambda_\eta$ is
attracting by Proposition~\ref{prop:postcritical_unique_period_two}. Hence for
every $\rho>0$ there are disjoint open intervals $U_-\ni y_-$ and
$U_+\ni y_+$ with compact closures such that
\begin{align}
    \overline U_-\cup\overline U_+
    \subseteq
    \{y\in J_\eta:\dist{y}{\Lambda_\eta}<\rho\},
    \quad
    h_\eta(\overline U_-)\subseteq U_+,
    \quad
    h_\eta(\overline U_+)\subseteq U_-.
\end{align}
Thus $U:=U_-\cup U_+$ has compact closure,
$h_\eta(\overline U)\subseteq U$, and $M\cap U\ne\varnothing$.
Lemma~\ref{lem:trapping_confines_chains} gives $M\subseteq U$, and letting
$\rho\downarrow0$ gives $M\subseteq\Lambda_\eta$. A nonempty invariant subset
of a $2$-cycle is the whole cycle, so $M=\Lambda_\eta$.

\emph{Case 2: $M\cap\Lambda_\eta=\varnothing$.} Every $w\in M$ has its forward
orbit in $M$, and by Corollary~\ref{cor:postcritical_exact_orbits} its
$\omega$-limit set is $\{0\}$, $\{y_\star\}$, or $\Lambda_\eta$. The last is
impossible, since the $\omega$-limit set is contained in the closed set $M$ and
$M\cap\Lambda_\eta=\varnothing$.
Hence every orbit in $M$ converges to $0$ or to
$y_\star$.

\emph{Subcase 2a: $y_\star\in M$.} By
Lemma~\ref{lem:ystar_isolated_invariant_set}, $y_\star$ is isolated in $M$, and
Lemma~\ref{lem:frozen_chain} gives $M=\{y_\star\}$.

\emph{Subcase 2b: $y_\star\notin M$.} Then every orbit in $M$ converges to $0$,
so Lemma~\ref{lem:invariant_set_origin} gives $M=\{0\}$.
\end{proof}

\subsection{Lift to the full state and terminal inheritance}
\label{subsec:full_state_lift}

We now return to the three normalized models covered by
Lemma~\ref{lem:concrete_axioms}: scalar factorization and rank-$1$
factorization, for which $\delta_{\mathrm{th}}=\eta$ when $0<\eta<2$, and
rank-$1$ approximation, for which
$\delta_{\mathrm{th}}=2(1-\sqrt{1-\eta^2})/\eta$ when $0<\eta<1$. In all
three cases the terminal state is $\overline\delta=2$.

For the classification below, fix any initialization
$x_0\in\operatorname{int}K_{\delta_{\mathrm{th}}}$ for which
Theorem~\ref{thm:abstract_state_dependent_convergence}(2) holds (i.e.,
$\delta_\ast=\overline\delta=2$), and let $\Omega:=\omega(x_0)$. By
Lemma~\ref{lem:omega_in_K2} and
Lemma~\ref{lem:ict_pushdown_semiconjugacy}, $\Omega$ is nonempty, compact,
invariant, internally chain transitive, and contained in $K_2$, and
$M:=\pi_\eta(\Omega)$ has the same four properties for $h_\eta$. We determine
the possible sets $\Omega$ by combining
Theorems~\ref{thm:reduced_ict_classification_subcritical}
and~\ref{thm:reduced_ict_classification_postcritical} with explicit descriptions
of the fibers of $\pi_\eta$ over the resulting candidate sets.

Two fibers are immediate from Proposition~\ref{prop:terminal_sets}. The value
$y=0$ corresponds to $s=0$, hence to $w=0$, so
\begin{align}
    \pi_\eta^{-1}(0)\cap K_2=\{0\},
    \label{eq:fiber_origin}
\end{align}
the origin of the full space. The value $y=y_\star$ corresponds to $s=1$, so
\begin{align}
    \pi_\eta^{-1}(y_\star)\cap K_2^\sca &= \{(1,1),(-1,-1)\},
    \notag\\
    \pi_\eta^{-1}(y_\star)\cap K_2^\fac &= \{(1,1,0,0),(-1,-1,0,0)\},
    \label{eq:fiber_minimizer}\\
    \pi_\eta^{-1}(y_\star)\cap K_2^\apx &= \{(a,a,0,0):\norm{a}=1\}.
    \notag
\end{align}
In the first two cases the fiber consists of two fixed points of $\gd_\eta$; in
the third it is a sphere, and it is contained in the global-minimizer set
$\mathcal M$, since $a=b$ gives $a\parallel b$ and
$\langle a,b\rangle=\norm{a}^2=1$.

The remaining reduced candidate, which occurs only in the post-critical regime,
is $\Lambda_\eta$. Its full fiber has more than one invariant component, so
lifting the reduced classification requires determining which components are
compatible with internal chain transitivity.

\begin{lemma}
\label{lem:signed_cycle_split}
Let $1<\eta<\sqrt5-1$ and set $w_\pm:=\sqrt{s_\pm}$ with $s_\pm$ as in
Eq.~\eqref{eq:s_pm}. For scalar factorization,
\begin{align}
    \pi_\eta^{-1}(\Lambda_\eta)\cap K_2^\sca
    =
    \{(w_-,w_-),(w_+,w_+),(-w_-,-w_-),(-w_+,-w_+)\},
    \label{eq:four_point_fiber}
\end{align}
and the same set with $u=v=0$ appended for rank-1 factorization. In the scalar
model, this four-point set is the disjoint union of the two $2$-cycles
\begin{align}
    C_\sca^+:=\{(w_-,w_-),(w_+,w_+)\},
    \qquad
    C_\sca^-:=-C_\sca^+.
\end{align}
For rank-$1$ factorization, define
\begin{align}
    C_\fac^\pm
    :=
    \{(a,b,0,0):(a,b)\in C_\sca^\pm\}.
\end{align}
In either model, these are $2$-cycles of $\gd_\eta$, and their union is not
internally chain transitive. Below, $C^\pm$ denotes $C_\sca^\pm$ or
$C_\fac^\pm$ according to the model under consideration.
\end{lemma}

\begin{proof}
On $K_2$, a point is determined by its balanced signal variable $w\in\RR$, and
$\pi_\eta(x)\in\Lambda_\eta$ means $w^2=s(x)\in\{s_-,s_+\}$, which gives the
four listed points; they lie in $K_2$ because $s_\pm<2$ by
Eq.~\eqref{eq:s_pm}. By Eq.~\eqref{eq:balanced_update} the gradient step acts on
the signal variable as
\begin{align}
    w^+=(1+\eta-\eta s)\,w=(1+\eta)(1-y)\,w ,
\end{align}
using $s=(1+\eta)y/\eta$. Since $0<y_\pm<1$, the factor $(1+\eta)(1-y_\pm)$ is
strictly positive, so the sign of $w$ is preserved. Combined with
$h_\eta(y_\mp)=y_\pm$, this shows that $\gd_\eta$ maps $(w_-,w_-)\mapsto(w_+,w_+)\mapsto(w_-,w_-)$
and likewise on the negative branch, so $C^+$ and $C^-$ are $2$-cycles and the
four-point set is their disjoint union.

Finally, let $\varepsilon>0$ be smaller than $\dist{C^+}{C^-}=2\sqrt2\,w_->0$.
Every point of $C^+\cup C^-$ is mapped by $\gd_\eta$ to a point of its own
cycle, so any internal $\varepsilon$-chain starting in $C^+$ has each successive
point within $\varepsilon$ of a point of $C^+$, hence in $C^+$; no internal
$\varepsilon$-chain reaches $C^-$. Therefore $C^+\cup C^-$ is not internally
chain transitive.
\end{proof}

\begin{proposition}
\label{prop:terminal_omega_classification}
Let $x_0\in\operatorname{int}K_{\delta_{\mathrm{th}}}$ be an initialization for
which the conclusions of
Theorem~\ref{thm:abstract_state_dependent_convergence} hold. Suppose that
$\delta_\ast=2$, and let $\Omega=\omega(x_0)$.

(1) If $0<\eta<1$, then either $\Omega=\{0\}$, or $\Omega$ is contained in the
fiber Eq.~\eqref{eq:fiber_minimizer}. In the scalar and rank-1 factorization
cases, the latter forces $\Omega$ to be a single balanced global minimizer, so
$x_t$ converges to it; in the rank-1 approximation case it gives
$\dist{x_t}{\mathcal M}\to0$.

(2) If $1<\eta<\sqrt5-1$, then for scalar factorization and rank-1
factorization exactly one of the following holds: $\Omega=\{0\}$; $\Omega$ is a
single balanced global minimizer; or $\Omega=C^+$ or $\Omega=C^-$.
\end{proposition}

\begin{proof}
(1) By Theorem~\ref{thm:reduced_ict_classification_subcritical},
$M=\{0\}$ or $M=\{y_\star\}$. If $M=\{0\}$ then
Eq.~\eqref{eq:fiber_origin} gives $\Omega=\{0\}$. If $M=\{y_\star\}$ then
$\Omega$ lies in the fiber Eq.~\eqref{eq:fiber_minimizer}.

In the scalar and rank-$1$ factorization cases, this fiber consists of two fixed
points at positive distance. Choose $p\in\Omega$. Then $p$ is an isolated fixed
point in $\Omega$, so Lemma~\ref{lem:frozen_chain} gives $\Omega=\{p\}$. A
precompact trajectory with a singleton $\omega$-limit set converges to that
point; hence $x_t\to p$.

In the approximation case, the fiber is contained in $\mathcal M$. Thus every
accumulation point of the trajectory lies in $\mathcal M$, and
Lemma~\ref{lem:setwise_from_accumulation} gives
$\dist{x_t}{\mathcal M}\to0$.

(2) By Theorem~\ref{thm:reduced_ict_classification_postcritical},
$M\in\{\{0\},\{y_\star\},\Lambda_\eta\}$. The first two cases are as in (1). If
$M=\Lambda_\eta$, then $\Omega$ is a nonempty invariant internally chain
transitive subset of the four-point set Eq.~\eqref{eq:four_point_fiber}. Its
nonempty invariant subsets are $C^+$, $C^-$, and $C^+\cup C^-$, and the union is
excluded by Lemma~\ref{lem:signed_cycle_split}. Hence $\Omega=C^+$ or
$\Omega=C^-$.
\end{proof}

In the pre-critical regime, Proposition~\ref{prop:terminal_omega_classification}(1)
already establishes the concrete convergence conclusions for trajectories with
$\delta_\ast=2$, uniformly across the three models. We record it separately,
since this is the statement cited from the applications in
Appendix~\ref{subsec:applications_of_theorem5}.

\begin{corollary}[The terminal case in the pre-critical regime]
\label{cor:subcritical_terminal_inheritance}
Let $0<\eta<1$. For each of the three models, the following holds for almost
every $x_0\in\operatorname{int}K_{\delta_{\mathrm{th}}}$ satisfying
$\delta_\ast=2$: in the scalar and rank-$1$ factorization models the trajectory
converges to a global minimizer, and in the rank-$1$ approximation model
$\dist{x_t}{\mathcal M}\to0$.
\end{corollary}

\begin{proof}
Fix one of the three models, and let $N_0$ be the measure-zero exceptional set
outside which Theorem~\ref{thm:abstract_state_dependent_convergence} holds.
For $x_0\notin N_0$ with $\delta_\ast=2$,
Proposition~\ref{prop:terminal_omega_classification}(1) shows that the only
alternative to the stated conclusions is $\Omega=\{0\}$, that is, convergence
to the origin.
The origin is a stationary point at which $D\gd_\eta$ has the eigenvalue
$1+\eta>1$, by Lemma~\ref{lem:nonmin_unstable}, and it belongs to the compact
set $(\mathcal S\setminus\mathcal M)\cap K_{\delta_{\mathrm{th}}}$, so
Theorem~\ref{thm:abstract_state_dependent_convergence}(4) applied with
$\mathcal U=\{0\}$ shows that the set $N_1$ of initializations whose trajectory
converges to it has measure zero. The conclusion holds with
$N:=N_0\cup N_1$.
\end{proof}

\begin{theorem}[Period-two inheritance in the terminal case]
\label{thm:terminal_period_two_inheritance}
Fix $1<\eta<\sqrt5-1$ and consider scalar factorization or rank-1
factorization in the normalized setting $\sigma=1$. For almost every
$x_0\in\operatorname{int}K_\eta$ satisfying $\delta_\ast=2$, the $\omega$-limit
set of the full gradient descent trajectory is exactly one of the two signed
balanced period-$2$ orbits,
\begin{align}
    \omega(x_0)=C^+
    \quad\text{or}\quad
    \omega(x_0)=C^- .
\end{align}
In particular, if $C\in\{C^+,C^-\}$ denotes the selected cycle, then
$\dist{x_t}{C}\to0$, and the even and odd subsequences converge to its two
distinct points.
\end{theorem}

\begin{proof}
Let $N_0$ be the measure-zero exceptional set outside which
Theorem~\ref{thm:abstract_state_dependent_convergence} holds.
For $x_0\notin N_0$ with $\delta_\ast=2$,
Proposition~\ref{prop:terminal_omega_classification}(2) shows that the only
alternatives to the two signed cycles are $\Omega=\{0\}$ and $\Omega$ a single
balanced global minimizer. Both are convergence to a stationary point at which
$D\gd_\eta$ has an eigenvalue of modulus greater than one: the origin by
Lemma~\ref{lem:nonmin_unstable}, with eigenvalue $1+\eta$, and each global minimizer by
Lemma~\ref{lem:minimizer_instability}, with eigenvalue
$1-\eta(a_\ast^2+b_\ast^2)$ of modulus at least $2\eta-1>1$ because $\eta>1$.
Thus
\begin{align}
    \mathcal U:=\{0\}\cup(\mathcal M\cap K_\eta)
\end{align}
is a compact subset of $K_\eta\cap\mathcal S$, and every point of $\mathcal U$
has an expanding Jacobian direction. By
Theorem~\ref{thm:abstract_state_dependent_convergence}(4), the set $N_1$ of
initializations whose trajectory converges to a point of $\mathcal U$ has
measure zero. Let $N:=N_0\cup N_1$. For $x_0\notin N$ with $\delta_\ast=2$,
both stationary alternatives are excluded and
$\Omega\in\{C^+,C^-\}$.

For the final sentence, let $C\in\{C^+,C^-\}$ be the selected cycle.
Then $\omega(x_0)=C$ and precompactness of the orbit give
$\dist{x_t}{C}\to0$. Write $C=\{p_-,p_+\}$ with
$\gd_\eta(p_\mp)=p_\pm$. Choose disjoint open neighborhoods $V_-\ni p_-$ and
$V_+\ni p_+$ whose closures meet $C$ only at their respective points. By
continuity of $\gd_\eta$, there are open neighborhoods $U_\pm\ni p_\pm$ with
$\overline U_\pm\subset V_\pm$ such that
$\gd_\eta(U_-)\subseteq V_+$ and $\gd_\eta(U_+)\subseteq V_-$. Since
$\dist{x_t}{C}\to0$, there is $T$ with $x_t\in U_-\cup U_+$ for all $t\ge T$.
If $x_t\in U_-$, then $x_{t+1}\in V_+\cap(U_-\cup U_+)=U_+$; similarly,
$x_t\in U_+$ implies $x_{t+1}\in U_-$. Thus the trajectory alternates between
$U_-$ and $U_+$ from time $T$ on, and its parity is eventually fixed: one of
the even and odd subsequences lies in $U_-$ and the other in $U_+$ for $t\ge T$.
Each subsequence is precompact, and all of its accumulation points lie both in
$C$ and in the closure of its assigned neighborhood. That intersection is
a singleton, so the two subsequences converge to the two respective points of
the cycle.
\end{proof}

\begin{corollary}[Inheritance for almost every initialization]
\label{cor:theorem_five_period_two_inheritance}
Fix $1<\eta<\sqrt5-1$ and consider scalar factorization or rank-1
factorization with $\sigma=1$. By
Theorem~\ref{thm:abstract_state_dependent_convergence}(5), whose hypotheses were
verified for these two models in
Appendix~\ref{subsec:applications_of_theorem5}, almost every
$x_0\in\operatorname{int}K_\eta$ satisfies $\delta_\ast=2$. Combining this with
Theorem~\ref{thm:terminal_period_two_inheritance}, for almost every
$x_0\in\operatorname{int}K_\eta$ the trajectory has $\omega$-limit set equal to
one of the two signed balanced period-$2$ orbits.
\end{corollary}

\begin{proof}
Both statements hold outside subsets of
$\operatorname{int}K_\eta$ of measure zero; discard the union of the two.
\end{proof}

\begin{remark}[Limitations for isotropic-signal rank-1 approximation]
\label{rmk:apx_no_postcritical_inheritance}
Theorem~\ref{thm:terminal_period_two_inheritance} is asserted for scalar and
rank-1 factorization only. For rank-1 approximation,
Lemma~\ref{lem:concrete_axioms} verifies the certificate framework only for
$0<\eta<1$, so it provides no post-critical terminal localization. There is
also a non-discrete-fiber limitation already visible in the pre-critical
regime: the fiber over the reduced fixed point $y_\star$ is a sphere and, when
$n\ge3$, can itself be internally chain transitive. Thus internal chain
transitivity need not select a single full-state point or signal direction.
Accordingly, Corollary~\ref{cor:subcritical_terminal_inheritance} gives only a
setwise conclusion for rank-1 approximation, and we do not assert a
post-critical full-state cycle inheritance result for that model.
\end{remark}

\subsection{Sharpness of the inherited period-two orbit}
\label{subsec:terminal_sharpness}

Theorem~\ref{thm:terminal_period_two_inheritance} identifies the full-state
limits $C^\pm$ and shows that the even and odd subsequences converge to their
two points. We now characterize these inherited cycles in terms of the Hessian
of the original loss. Here the \emph{sharpness} at a point means the largest
eigenvalue of that Hessian. We compare it with $2/\eta$, the linear-stability
threshold for a positive Hessian eigenmode. The calculation below shows that
the two limiting sharpness values lie on opposite sides of this threshold.
Hence, by continuity of the Hessian spectrum, the sharpness along the even and
odd subsequences alternates asymptotically across $2/\eta$. By
Corollary~\ref{cor:theorem_five_period_two_inheritance}, this conclusion holds
for almost every initialization in $\operatorname{int}K_\eta$. For further
discussion in relation to prior work, see Appendix~\ref{apx:related_work}.

\begin{corollary}[Sharpness alternates across $2/\eta$]
\label{cor:period_two_sharpness}
Let $1<\eta<\sqrt5-1$. For scalar factorization and rank-1 factorization, the
two points of either signed balanced period-$2$ orbit $C^\pm$ have product
coordinates $s_-$ and $s_+$. Their sharpness values are
\begin{align}
    S_\pm := 3s_\pm-1
    = \frac{\eta+3\pm3\sqrt{(\eta-1)(\eta+3)}}{2\eta},
    \qquad
    S_-<\frac2\eta<S_+ .
\end{align}
\end{corollary}

\begin{proof}
For $\risk_\sca(a,b)=\tfrac12(ab-1)^2$ the Hessian at a balanced point $a=b=w$
with $s=w^2$ is
\begin{align}
    \nabla^2\risk_\sca(w,w)
    =
    \begin{pmatrix}
        s & 2s-1\\
        2s-1 & s
    \end{pmatrix},
\end{align}
whose eigenvalues are $s\pm(2s-1)$, that is $3s-1$ and $1-s$. Since
$s_\pm>\tfrac12$ by Eq.~\eqref{eq:s_pm}, we have $3s_\pm-1>1-s_\pm$, so the
sharpness is $3s_\pm-1$; substituting Eq.~\eqref{eq:s_pm} gives the displayed
values. The inequality $S_-<2/\eta<S_+$ is equivalent to
$3\sqrt{(\eta-1)(\eta+3)}>\eta-1$, that is $9(\eta+3)>\eta-1$, which holds for
every $\eta>1$.

For rank-1 factorization, the Hessian of $\risk_\fac$ at $(w,w,0,0)$ is block
diagonal with the same signal block and the off-signal block $s\,I_2$, so its
eigenvalues are $1-s$, $3s-1$, and $s$ twice. Again $s_\pm>\tfrac12$ gives
$3s_\pm-1>s_\pm$, so the additional eigenvalue does not change the maximum. The
orthogonal reduction of Subsection~\ref{subsec:gd_transformation} is an isometry
and preserves the Hessian spectrum. In the original coordinates, direct
differentiation at a balanced point gives the same $2\times2$ signal block and
the scalar operator $sI$ on each off-signal subspace, with no mixed terms.
Thus the additional eigenvalue $s$ only acquires multiplicity, and the sharpness
of the original problem is the same.
\end{proof}

%% file: apxI_local_obstruction.tex
\section{A local obstruction to fixed quadratic Lyapunov functions}
\label{appendix:state-dependent-lyapunov}

In this appendix, we make precise the obstruction to using fixed quadratic Lyapunov
functions for scalar factorization. Consider the scalar loss
\begin{align}
    \risk_\sca(a,b) := \tfrac{1}{2}(1-ab)^2,
\end{align}
with one gradient descent step $\gd_\eta^\sca$ of step size $\eta > 0$, and a fixed
quadratic candidate
\begin{align}
    V(x) := (x-x_c)^\top P\,(x-x_c),
\end{align}
where $P=P^\top\in \RR^{2\times 2}$ is positive definite and
$x_c \in \RR^2$ is fixed.

\begin{proposition}\label{prop:fixed_quadratic_alignment}
Let $x_* = (a_*, b_*)^\top$ be a global minimizer of $\risk_\sca$, so that
$a_* b_* = 1$. Assume there exists a neighborhood $U$ of $x_*$ such that for every
$x \in U$ with $V(x) \neq V(x_*)$,
\begin{align}
    \left|
    \frac{V(\gd_\eta^\sca(x)) - V(x_*)}{V(x) - V(x_*)}
    \right| \le 1.
\end{align}
Set $w := P(x_* - x_c)$. If $w \neq 0$, then
\begin{align}
    \nabla^2 \risk_\sca(x_*)\, w = \lambda w
\end{align}
for some $\lambda \in[0,2/\eta]$; equivalently, $w$ must be an eigenvector of
the Hessian at $x_*$ whose gradient-descent multiplier has modulus at most one.
In particular, if $w \notin \ker \nabla^2 \risk_\sca(x_*)$, then $w$ is
parallel to the unique positive-eigenvalue direction $\binom{1}{a_*^2}$, and
$a_*^2+b_*^2\le 2/\eta$.
\end{proposition}

\begin{proof}
Write $x = x_* + \epsilon$ with $\epsilon = (\epsilon_1,\epsilon_2)^\top$ small.
Since $\nabla \risk_\sca(x_*) = 0$,
\begin{align}
    \nabla \risk_\sca(x_*+\epsilon)
    = H_* \epsilon + O(\norm{\epsilon}^2),
    \quad
    H_* := \nabla^2 \risk_\sca(x_*)
    = \begin{pmatrix} b_*^2 & 1\\ 1 & a_*^2 \end{pmatrix}.
\end{align}
Using symmetry of $P$,
\begin{align}
    V(x_*+\epsilon) - V(x_*)
    &= 2 \epsilon^\top w + \epsilon^\top P \epsilon
    = 2 \epsilon^\top w + O(\norm{\epsilon}^2),\\
    \gd_\eta^\sca(x_*+\epsilon)
    &= x_* + \epsilon - \eta H_* \epsilon + O(\norm{\epsilon}^2),
\end{align}
so
\begin{align}
    V(\gd_\eta^\sca(x_*+\epsilon)) - V(x_*)
    &= 2\epsilon^\top w - 2\eta\, \epsilon^\top H_* w + O(\norm{\epsilon}^2).
\end{align}
Setting
$\mathrm{den}(\epsilon) := V(x_*+\epsilon) - V(x_*)$ and
$\mathrm{num}(\epsilon) := V(\gd_\eta^\sca(x_*+\epsilon)) - V(x_*)$,
the assumption is $|\mathrm{num}(\epsilon)|\le|\mathrm{den}(\epsilon)|$
whenever $\mathrm{den}(\epsilon)\ne0$. Substituting the expansions and
squaring gives
\begin{align}
    \bigl(\epsilon^\top(I-\eta H_*)w\bigr)^2
    \le
    (\epsilon^\top w)^2+o(\norm{\epsilon}^2).
\end{align}
Setting $\epsilon = t z$ with nonzero $z \in \RR^2$ fixed and letting
$t \downarrow 0$, the denominator is nonzero for all sufficiently small
$t\ne0$: this follows
from its linear term if $z^\top w\ne0$, while if $z^\top w=0$ then it equals
$t^2z^\top Pz>0$. Division by $t^2$ therefore gives
\begin{align}\label{eq:local_absolute_deviation_monotonicity}
    \bigl|z^\top(I-\eta H_*)w\bigr|\le |z^\top w|
    \quad\text{for all } z\in \RR^2.
\end{align}

Set $u:=(I-\eta H_*)w$.
Equation~\eqref{eq:local_absolute_deviation_monotonicity} shows that
$z^\top u=0$ whenever $z^\top w=0$. Hence $u=\rho w$ for some scalar $\rho$,
and another application of
Eq.~\eqref{eq:local_absolute_deviation_monotonicity} gives $|\rho|\le1$.
Therefore
\begin{align}
    H_*w=\frac{1-\rho}{\eta}\,w=:\lambda w,
    \qquad
    0\le\lambda\le\frac2\eta.
\end{align}

Finally, $a_* b_* = 1$ gives $\det H_* = a_*^2 b_*^2 - 1 = 0$, so $H_*$ has
eigenvalues $0$ and $a_*^2+b_*^2$, with
\begin{align}
    \ker H_* = \Span{\binom{a_*^2}{-1}},
    \quad
    \operatorname{im} H_* = \Span{\binom{1}{a_*^2}}.
\end{align}
Therefore, if $w \notin \ker H_*$, then necessarily
$\lambda=a_*^2+b_*^2>0$ and
$w \in \Span{\binom{1}{a_*^2}}$, the unique positive-eigenvalue direction;
the bound on $\lambda$ also gives $a_*^2+b_*^2\le2/\eta$.
\end{proof}

Proposition~\ref{prop:fixed_quadratic_alignment} shows that any fixed quadratic
candidate for which the absolute deviation $|V(x)-V(x_*)|$ is locally
nonincreasing under gradient descent must be aligned with a Hessian
eigendirection satisfying $|1-\eta\lambda|\le1$. The required direction depends
on the selected minimizer through $a_*^2$. Thus, guaranteeing this local
monotonicity for trajectories initialized in a neighborhood of their limiting
minimizer requires a minimizer-dependent compatibility that a fixed quadratic
candidate does not generically possess. This is the local obstruction that
motivates allowing the quadratic form to vary with the state.

%% file: apxJ_uniqueness_proof.tex
\section{Proof of Theorem~\ref{thm:uniqueness_state_dependent_lyapunov}:
Uniqueness up to equivalence}
\label{appendix:uniqueness_state_dependent_lyapunov}

We prove the uniqueness statement for the scalar factorization problem
\begin{align}
    \risk_\sca(a, b) = \tfrac{1}{2}(1 - ab)^2,
\end{align}
whose gradient descent map is given by Eq.~\eqref{eq:model_catalog_scalar_gd}.

We first record what the structural axioms give for an origin-centered
certificate. This is the specialization of
Lemma~\ref{lem:compact_quadratic_sublevel} of
Appendix~\ref{subsec:theorem5_and_auxiliary_lemmas} to $q\equiv0$.

\begin{corollary}\label{cor:origin_centered_pd}
For the origin-centered subclass $q\equiv 0$, Axioms~\ref{cond:P_is_pd}
and~\ref{cond:nesting} give $P(\delta)\succ 0$ and $r(\delta)<0$ at every
nonterminal $\delta$.
\end{corollary}

\begin{proof}
Applying Lemma~\ref{lem:compact_quadratic_sublevel} to
$K_\delta=\{x:x^\top P(\delta)x+r(\delta)\le0\}$ turns
Axiom~\ref{cond:P_is_pd} into $P(\delta)\succ0$ and $r(\delta)\le0$. If
$r(\delta)=0$ then $K_\delta=\{0\}$ has empty interior, so
Axiom~\ref{cond:nesting} would force $K_{\delta'}=\varnothing$ for every
$\delta'\in(\delta,\overline\delta)$, contradicting
Axiom~\ref{cond:P_is_pd} at $\delta'$.
\end{proof}

\begin{lemma}[Step-size downgrade]
\label{lem:step_size_downgrade}
Suppose $\mathcal K_I$ satisfies the structural axioms and
Axiom~\ref{cond:monotonicity} holds for $\gd_{\eta_0}$ above
$\delta_{\mathrm{th}}$. Then
$(\mathcal K_I,\gd_\eta,\delta_{\mathrm{th}})$ is a state-dependent
Lyapunov family for every $\eta\in(0,\eta_0)$.
\end{lemma}

\begin{proof}
Fix $\delta\in(\delta_{\mathrm{th}},\overline\delta)$ and
$x\in\partial K_\delta$. With $\theta:=\eta/\eta_0\in(0,1)$,
\begin{align}
    \gd_\eta(x)=(1-\theta)x+\theta\gd_{\eta_0}(x).
\end{align}
$x$ already lies in $K_\delta$, and Axiom~\ref{cond:monotonicity} places
$\gd_{\eta_0}(x)$ there as well. If $x$ is nonstationary, the two points are
distinct, so strict convexity of the nondegenerate ellipsoid $K_\delta$ places
their convex combination in
$\operatorname{int}K_\delta$. If $x$ is stationary, both maps fix it. Thus
Axioms~\ref{cond:monotonicity} and~\ref{cond:stationarity} hold at $\eta$ with
the same threshold; the structural axioms are unchanged.
\end{proof}

We first normalize the scale. Fix a nonterminal
$\delta\in(\delta_{\mathrm{th}},\overline\delta)$. By
Corollary~\ref{cor:origin_centered_pd} we have $r(\delta)<0$, so dividing
$I(\delta;\cdot)$ by $-r(\delta)>0$ leaves the sublevel set $K_\delta$
unchanged and, after replacing $P(\delta)$ by $P(\delta)/(-r(\delta))$, puts
the certificate in the form $x^\top P(\delta)\,x-1$. The rescaled coefficients
are $C^1$ on the nonterminal interval, and the rescaled certificate induces
the same level sets and the same state-dependent Lyapunov triplet. For the
rest of the proof we use this representation and write
\begin{align}
    I(\delta;\, x) = x^\top P(\delta)\, x - 1,
    \quad
    x = (a, b)^\top.
\end{align}

By Axiom~\ref{cond:symmetry} we may write
\begin{align}
    P(\delta)
    =
    \begin{pmatrix}
        c(\delta) & d(\delta) \\
        d(\delta) & c(\delta)
    \end{pmatrix}.
\end{align}

For fixed $\delta \in (\delta_{\mathrm{th}}, \overline\delta)$, define the
one-step decrement function
\begin{align}
    \Phi_\delta(x)
    :=
    x^\top P(\delta)\, \nabla \risk_\sca(x)
    - \frac{\eta}{2}\,
      \nabla \risk_\sca(x)^\top P(\delta)\, \nabla \risk_\sca(x),
\end{align}
so that, 
\begin{align}
    I(\delta;\, \gd_\eta(x)) - I(\delta;\, x) = -2\eta\, \Phi_\delta(x).
\end{align}
Axiom~\ref{cond:monotonicity} therefore gives
\begin{align}
    \Phi_\delta(x) \ge 0
    \quad
    \text{on } \{x : I(\delta;\, x) = 0\},
\end{align}
while Axiom~\ref{cond:stationarity} gives $\Phi_\delta(x) = 0$ at every
stationary point of $\risk_\sca$ on this level set. Hence every such
stationary point is a constrained minimizer of $\Phi_\delta$ subject to
$I(\delta;\, x) = 0$.

Let $x_*=(a_*,b_*)^\top$ be a nonzero stationary point on this level set.
Then $\nabla \risk_\sca(x_*)=0$ gives $a_*b_*=1$, and
\begin{align}
    x_*^\top P(\delta)\,x_* = 1.
    \label{eq:appendix_normalization_constraint}
\end{align}
The calculation below treats the case $a_*\ne b_*$. We show afterward that
every admissible nonterminal level has such an unbalanced stationary contact,
so this local calculation applies to the entire nonterminal family.

\textbf{First-order optimality condition.}
Writing
\begin{align}
    g(x) := \nabla \risk_\sca(x)
    =
    \begin{pmatrix}
        (ab - 1)b \\
        (ab - 1)a
    \end{pmatrix}
\end{align}
and
\begin{align}
    H_* := \nabla^2 \risk_\sca(x_*)
    =
    \begin{pmatrix}
        b_*^2 & 1 \\
        1 & a_*^2
    \end{pmatrix},
\end{align}
and using $g(x_*) = 0$, differentiation of $\Phi_\delta$ at $x_*$ yields
\begin{align}
    \nabla \Phi_\delta(x_*) = H_*\, P(\delta)\, x_*.
\end{align}
Since $\nabla_x I(\delta;\, x_*) = 2\, P(\delta)\, x_*$, the Lagrange
multiplier condition reads
\begin{align}
    H_*\, P(\delta)\, x_* = 2\lambda\, P(\delta)\, x_*
    \label{eq:appendix_lagrange_eigenvector}
\end{align}
for some $\lambda \in \RR$; that is, $P(\delta)\, x_*$ is an eigenvector of
$H_*$.

Because $a_* b_* = 1$,
$\det H_* = a_*^2 b_*^2 - 1 = 0$,
so the eigenvalues of $H_*$ are
\begin{align}
    \lambda_1 = 0,
    \quad
    \lambda_2 = a_*^2 + b_*^2,
\end{align}
with corresponding eigenvectors
\begin{align}
    v_1 =
    \begin{pmatrix}
        a_*^2 \\
        -1
    \end{pmatrix},
    \quad
    v_2 =
    \begin{pmatrix}
        1 \\
        a_*^2
    \end{pmatrix}.
\end{align}
We consider the two possibilities for $P(\delta)\, x_*$ in turn.

\textbf{Case 1: $P(\delta)\, x_* \parallel v_2$.}
Suppose
\begin{align}
    P(\delta)\, x_* = t \begin{pmatrix} 1 \\ a_*^2 \end{pmatrix}
\end{align}
for some $t \in \RR$. Using $b_* = 1/a_*$,
\begin{align}
    P(\delta)\, x_*
    =
    \begin{pmatrix}
        c\, a_* + d/a_* \\
        d\, a_* + c/a_*
    \end{pmatrix},
\end{align}
so
\begin{align}
    c\, a_* + \frac{d}{a_*} = t,
    \quad
    d\, a_* + \frac{c}{a_*} = t\, a_*^2.
\end{align}
Eliminating $t$ gives
\begin{align}
    d\, a_* + \frac{c}{a_*}
    = a_*^2 \Bigl(c\, a_* + \frac{d}{a_*}\Bigr),
\end{align}
which simplifies to $c/a_* = c\, a_*^3$. By the assumption that $a_* \ne b_*$, we have $a_*^2 \neq 1$, and the
identity above forces $c = 0$.
Eq.~\eqref{eq:appendix_normalization_constraint} then gives
\begin{align}
    1 = x_*^\top P(\delta)\, x_* = 2d,
\end{align}
so
\begin{align}
    P(\delta)
    =
    \begin{pmatrix}
        0 & 1/2 \\
        1/2 & 0
    \end{pmatrix}.
\end{align}
This matrix is not positive definite, contradicting Axiom~\ref{cond:P_is_pd}. Case~1
is therefore impossible.

\textbf{Case 2: $P(\delta)\, x_* \parallel v_1$.}
Suppose
\begin{align}
    P(\delta)\, x_* = t \begin{pmatrix} a_*^2 \\ -1 \end{pmatrix}
\end{align}
for some $t \in \RR$. Then
\begin{align}
    c\, a_* + \frac{d}{a_*} = t\, a_*^2,
    \quad
    d\, a_* + \frac{c}{a_*} = -t.
\end{align}
Eliminating $t$ gives
\begin{align}
    c\, a_* + \frac{d}{a_*}
    = -a_*^2\Bigl(d\, a_* + \frac{c}{a_*}\Bigr)
    = -d\, a_*^3 - c\, a_*,
\end{align}
hence
\begin{align}
    2c\, a_*^2 + d\, (1 + a_*^4) = 0, \quad\text{i.e.,}\quad
    d = -\frac{2c\, a_*^2}{1 + a_*^4}.
    \label{eq:appendix_d_in_terms_of_c}
\end{align}
Substituting into
Eq.~\eqref{eq:appendix_normalization_constraint} yields
\begin{align}
    1
    = c\Bigl(a_*^2 + \tfrac{1}{a_*^2}\Bigr) + 2d
    = c\,\frac{a_*^4 + 1}{a_*^2}
      - \frac{4c\, a_*^2}{1 + a_*^4},
\end{align}
which simplifies to
\begin{align}
    c = \frac{a_*^2(a_*^4 + 1)}{(a_*^4 - 1)^2},
    \quad
    d = -\frac{2\, a_*^4}{(a_*^4 - 1)^2}.
    \label{eq:appendix_c_d_from_astar}
\end{align}

\textbf{Reparametrization by the scalar state.}
Define
\begin{align}
    \zeta := \frac{4\, a_*^2}{a_*^4 + 1}.
    \label{eq:appendix_zeta_from_astar}
\end{align}
Since $a_*^2 > 0$, $\zeta > 0$; moreover,
\begin{align}
    2 - \zeta
    = \frac{2(a_*^2 - 1)^2}{a_*^4 + 1} \ge 0,
\end{align}
with equality if and only if $a_*^2 = 1$, so $\zeta \in (0, 2)$ in the
present regime. A direct computation gives
\begin{align}
    4 - \zeta^2 = \frac{4(a_*^4 - 1)^2}{(a_*^4 + 1)^2},
\end{align}
and hence
\begin{align}
    \frac{\zeta}{4 - \zeta^2}
    &= \frac{4\, a_*^2}{a_*^4 + 1}
       \cdot
       \frac{(a_*^4 + 1)^2}{4(a_*^4 - 1)^2}
    = \frac{a_*^2(a_*^4 + 1)}{(a_*^4 - 1)^2}
    = c,
    \\
    -\frac{\zeta^2}{2(4 - \zeta^2)}
    &= -\frac{16\, a_*^4}{2(a_*^4 + 1)^2}
       \cdot
       \frac{(a_*^4 + 1)^2}{4(a_*^4 - 1)^2}
    = -\frac{2\, a_*^4}{(a_*^4 - 1)^2}
    = d.
\end{align}
Therefore
\begin{align}\label{eq:unique_P}
    P(\delta)
    = \widetilde P(\zeta)
    :=
    \frac{1}{4 - \zeta^2}
    \begin{pmatrix}
        \zeta & -\zeta^2/2 \\
        -\zeta^2/2 & \zeta
    \end{pmatrix},
\end{align}
which is a member of the normalized quadratic family underlying $\isc$.
Consequently,
\begin{align}
    I(\delta;\, a, b)
    = x^\top P(\delta)\, x - 1
    = \frac{\zeta(a^2 + b^2) - \zeta^2\, ab}{4 - \zeta^2} - 1
    = \frac{\isc(\zeta;\,a,b)}{4-\zeta^2}.
\end{align}
Hence the positive-definite matrix is uniquely determined by the unbalanced
stationary contact and belongs to the scalar certificate family used in
Section~\ref{sec:rank-1 convergence}.

It remains to explain why the preceding local calculation determines the
whole state-dependent family on the original admissible nonterminal state
space.  For this purpose, choose an auxiliary step size
\begin{align}
    \eta_0:=\eta \quad\text{if }\eta<1,
    \qquad
    \eta_0\in(0,1) \quad\text{if }\eta\in[1,2).
\end{align}
If $\eta_0<\eta$, Lemma~\ref{lem:step_size_downgrade} shows that the original
set family, paired with $\gd_{\eta_0}$ and the same threshold, remains a
state-dependent Lyapunov family; if $\eta_0=\eta$, this is immediate. We may
therefore use the $\eta_0<1$ convergence result of
Theorem~\ref{thm:convergence_region} in the global argument below.

Let $K_\delta=\{x^\top P(\delta)x\le 1\}$ be an arbitrary family satisfying
the axioms, with original state parameter
$\delta\in(\delta_{\mathrm{th}},\bar\delta)$. We first show that $K_\delta$
is forward invariant under $\gd_{\eta_0}$. If
$x\in K_\delta\setminus K_{\bar\delta}$, then its state satisfies
$\delta(x)\ge \delta>\delta_{\mathrm{th}}$, and
Axiom~\ref{cond:monotonicity} at step size $\eta_0$ gives
\begin{align}
\gd_{\eta_0}(x)\in K_{\delta(x)}\subset K_\delta .
\end{align}
If $x\in K_{\bar\delta}$, take any $\delta_j\uparrow\bar\delta$ with
$\delta_j>\delta$. Terminal negligibility and nesting imply
$\dist{x}{\partial K_{\delta_j}}\to0$; otherwise a fixed ball about $x$ would
lie in $K_{\bar\delta}$. Thus there are
$x_j\in\partial K_{\delta_j}$ with $x_j\to x$. Boundary inwardness gives
$\gd_{\eta_0}(x_j)\in K_{\delta_j}\subset K_\delta$, so continuity of
$\gd_{\eta_0}$ and closedness of $K_\delta$ give
$\gd_{\eta_0}(x)\in K_\delta$.

We now show that $K_\delta$ contains a global minimizer. Choose
$x_0\in\operatorname{int}K_\delta$ outside the level set
$\{\isc(\eta_0;\cdot)=0\}$ and the exceptional measure-zero set of
Theorem~\ref{thm:convergence_region}. Forward invariance confines its orbit to
the compact set $K_\delta$. If $\isc(\eta_0;\,x_0)>0$, equivalently
$Q(1;\,a_0,b_0)>8/\eta_0$, Theorem~\ref{thm:convergence_region}(2) instead
gives divergence, a contradiction. Hence $\isc(\eta_0;\,x_0)<0$, equivalently
$Q(1;\,a_0,b_0)<8/\eta_0$, and
Theorem~\ref{thm:convergence_region}(1) gives convergence to a global
minimizer. Since $K_\delta$ is closed, the limiting minimizer belongs to
$K_\delta$. Hence
\begin{align}
K_\delta\cap\mathcal M\neq\varnothing,
\quad
\mathcal M:=\{(a,b)\in\mathbb R^2:ab=1\}.
\end{align}

We next show that $K_\delta$ must in fact contain an unbalanced global
minimizer on its boundary. By Axiom~\ref{cond:symmetry}, write
\begin{align}
P(\delta)
=
\begin{pmatrix}
c(\delta) & d(\delta) \\
d(\delta) & c(\delta)
\end{pmatrix}.
\end{align}
On the positive minimizer branch $(a,1/a)$, define
\begin{align}
h_\delta(a)
:=
(a,1/a)^\top P(\delta)(a,1/a)
=
c(\delta)\left(a^2+\frac{1}{a^2}\right)+2d(\delta),
\quad a>0.
\end{align}
Since $P(\delta)$ is positive definite, $c(\delta)>0$, and therefore
$h_\delta(a)\to\infty$ as $a\downarrow 0$ or $a\to\infty$. The set
\begin{align}
A_\delta
:=
\{a>0:h_\delta(a)\le 1\}
\end{align}
is therefore compact. Since $K_\delta$ contains a global minimizer, either
$A_\delta$ is nonempty or the analogous set on the negative minimizer branch
is nonempty. By the central symmetry of the origin-centered quadratic level
set, the two cases are identical, so we work on the positive branch.

We claim that $A_\delta$ cannot collapse to the singleton $\{1\}$ for any
nonterminal state $\delta\in(\delta_{\mathrm{th}},\bar\delta)$. Indeed,
if $A_\delta=\{1\}$, then the only global minimizers in $K_\delta$ are the
balanced minimizers $(1,1)$ and $(-1,-1)$. Moreover, they lie on
$\partial K_\delta$, because $h_\delta(1)=1$. Now take any
$\delta'\in(\delta,\bar\delta)$. By strict nesting,
\begin{align}
K_{\delta'}\subset \operatorname{int}K_\delta .
\end{align}
Since the only global minimizers in $K_\delta$ lie on $\partial K_\delta$, it
follows that
\begin{align}
K_{\delta'}\cap\mathcal M=\varnothing .
\end{align}
This contradicts the preceding argument applied to $K_{\delta'}$, which shows
that every admissible nonterminal level set must contain a global minimizer.
Therefore $A_\delta$ is not the singleton $\{1\}$.

Consequently, $A_\delta$ contains some point $a_\ast\neq 1$. Since
$A_\delta$ is compact and bounded away from both $0$ and $\infty$, it has a
boundary point $a_\ast\neq 1$. The corresponding point
$x_\ast=(a_\ast,1/a_\ast)^\top$ is an unbalanced global minimizer satisfying
\begin{align}
x_\ast\in\partial K_\delta,
\quad
a_\ast b_\ast=1,
\quad
a_\ast\neq b_\ast .
\end{align}

Applying the local Lagrange-multiplier calculation at this boundary minimizer,
the positive-definite quadratic form on this level set is forced to be
\begin{align}
P(\delta)=\widetilde P(\zeta_\ast),
\qquad
\zeta_\ast:=\frac{4a_\ast^2}{a_\ast^4+1}.
\end{align}
Thus
\begin{align}
K_\delta=\widetilde K_{\zeta_\ast},
\quad
\widetilde K_\zeta
:=
\{x\in\mathbb R^2:x^\top \widetilde P(\zeta)x\le 1\}.
\end{align}
Hence every original admissible nonterminal level set is one of the
normalized level sets $\widetilde K_\zeta$.

It remains to identify the normalized range. Since the original level sets
are strictly nested in $\delta$, and the normalized level sets
$\widetilde K_\zeta$ are strictly nested in $\zeta$, the correspondence
$\delta\mapsto\zeta(\delta)$ is strictly increasing. Therefore the
one-sided limit
\begin{align}
    \zeta_{\mathrm{th}}
    :=
    \lim_{\delta\downarrow\delta_{\mathrm{th}}} \zeta(\delta)
\end{align}
exists in $[0,2)$. In fact, $\zeta_{\mathrm{th}}\ge\eta$ (here we return to the
original step size $\eta$, at which Axiom~\ref{cond:monotonicity} is assumed). Suppose then that some
admissible level set had normalized parameter $\zeta(\delta)<\eta$. Picking a non-stationary boundary point $x$ of
$\widetilde K_{\zeta(\delta)}$ (the ellipse
$\partial\widetilde K_{\zeta(\delta)}$ meets the stationary set in at most
four points) and evaluating the quotient--remainder identity
Eq.~\eqref{eq:certificate_for_scalar} on the level set gives
\begin{align}
\isc\bigl(\zeta(\delta);\,\gd_\eta(x)\bigr)
=\eta\,\bigl(\zeta(\delta)-\eta\bigr)\bigl(\zeta(\delta)^2-4\bigr)L^2>0,
\qquad L:=1-ab,
\end{align}
so the GD step maps $x$ strictly outside
$K_\delta=\widetilde K_{\zeta(\delta)}$, contradicting
Axiom~\ref{cond:monotonicity}.

The upper endpoint of the normalized range must be $2$. If instead
$\zeta(\delta)$ converged to some value $\zeta_\infty<2$ as
$\delta\uparrow\bar\delta$, then the terminal set
\begin{align}
K_{\bar\delta}
=
\bigcap_{\delta<\bar\delta}K_\delta
\end{align}
would contain the nondegenerate ellipsoid $\widetilde K_{\zeta_\infty}$ and
therefore would have positive measure, contradicting the terminal
negligibility axiom. Hence
\begin{align}
\lim_{\delta\uparrow\bar\delta}\zeta(\delta)=2.
\end{align}

Next, we verify that no intermediate value is omitted, so that the range of
$\zeta$ is the full interval $(\zeta_{\mathrm{th}},2)$. Suppose, for
contradiction, that some $\zeta_\circ\in(\zeta_{\mathrm{th}},2)$ is
not attained. By the two endpoint limits and strict monotonicity, there
exist $\delta_1<\delta_2$ in
$(\delta_{\mathrm{th}},\overline\delta)$ with
$\zeta(\delta_1)<\zeta_\circ<\zeta(\delta_2)$. Pick any point $x$ on the
boundary $\partial\widetilde K_{\zeta_\circ}$. By the strict nesting of the
normalized family,
\begin{align}
x
\in
\operatorname{int}\widetilde K_{\zeta(\delta_1)}
\setminus
\widetilde K_{\zeta(\delta_2)}
=
\operatorname{int}K_{\delta_1}\setminus K_{\delta_2}.
\end{align}
In particular, $x\notin K_{\overline\delta}$, so
Axiom~\ref{cond:level_set_as_state} provides a state $\delta(x)$ with
$I(\delta(x);\,x)=0$, that is, $x\in\partial K_{\delta(x)}$.
Axiom~\ref{cond:nesting} forces $\delta_1<\delta(x)<\delta_2$: if
$\delta(x)\le\delta_1$, then
$\operatorname{int}K_{\delta_1}\subseteq\operatorname{int}K_{\delta(x)}$,
so $x\in\operatorname{int}K_{\delta(x)}$, contradicting
$x\in\partial K_{\delta(x)}$; if $\delta(x)\ge\delta_2$, then
$x\in K_{\delta(x)}\subseteq K_{\delta_2}$, contradicting $x \notin K_{\delta_2}$. Hence $\delta(x)$ lies in the admissible interval, so
$K_{\delta(x)}=\widetilde K_{\zeta(\delta(x))}$ and
$x\in\partial\widetilde K_{\zeta(\delta(x))}$. Since distinct members of a
strictly nested family have disjoint boundaries, this forces
$\zeta(\delta(x))=\zeta_\circ$, contradicting the assumption that
$\zeta_\circ$ is not attained.

Consequently, $\delta\mapsto\zeta(\delta)$ is a strictly increasing
surjection from $(\delta_{\mathrm{th}},\overline\delta)$ onto the
interval $(\zeta_{\mathrm{th}},2)$. A strictly monotone surjection
onto an interval has no jumps, so $\zeta$ is a continuous bijection. This
is the reparameterization in the statement of
Theorem~\ref{thm:uniqueness_state_dependent_lyapunov}.

Finally, the terminal set is determined after this identification by the
nested limiting intersection
\begin{align}
K_{\bar\delta}
=
\bigcap_{\zeta_{\mathrm{th}}<\zeta<2}
\widetilde K_\zeta .
\end{align}
Thus the local calculation determines the entire nonterminal quadratic family
on the original admissible state space, and the terminal set is then uniquely
determined by the nesting axiom.

\begin{remark}\label{rmk:ext_to_quartic}
The local Lagrange analysis above is not specific to the particular scalar loss
$\risk_{\mathrm{sc}}(a,b)=\frac12(1-ab)^2$. Indeed, it applies to any $C^2$
two-variable objective $\widetilde{\risk}(a,b)$ whose gradient vanishes on the
scalar factorization minimizer manifold,
\begin{align}
    \nabla \widetilde{\risk}(a,b)=0
    \quad\text{for every } (a,b)\in\RR^2 \text{ with } ab=1,
\end{align}
which does not preclude stationary points off that manifold, 
and whose Hessian at every stationary point $(a_*,b_*)$ with $a_*b_* = 1$ agrees with that of
$\risk_{\mathrm{sc}}$, namely
\begin{align}
    \nabla^2 \widetilde{\risk}(a_*,b_*)
    =
    \begin{pmatrix}
        b_*^2 & 1\\
        1 & a_*^2
    \end{pmatrix}.
\end{align}
Then the same constrained-minimum/Lagrange-multiplier argument applies verbatim,
since the derivation only uses that $g(x_*)=0$ and the first-order expansion
\begin{align}
    \nabla \Phi_\delta(x_*) = \nabla^2 \widetilde{\risk}(x_*)\, P\, x_*.
\end{align}
Consequently, within the origin-centered subclass and under the same symmetry
axiom $P_{11}(\delta)=P_{22}(\delta)$, the quadratic form at any unbalanced
stationary boundary contact is characterized exactly as in the scalar
factorization problem. In particular, this applies to
\begin{align}
    \widetilde{\risk}(a,b)
    =
    \frac12(1-ab)^2 + \mu(1-ab)^4,
    \quad \mu \ge 0,
\end{align}
because its global minimizer set is still $\{ab=1\}$, and at every stationary point
$(a_*,b_*)$ with $a_*b_*=1$, its Hessian is again
\begin{align}
    \nabla^2 \widetilde{\risk}(a_*,b_*)
    =
    \begin{pmatrix}
        b_*^2 & 1\\
        1 & a_*^2
    \end{pmatrix}.
\end{align}
Thus, the local uniqueness mechanism for the quadratic state-dependent Lyapunov
family is unchanged for this quartic-augmented scalar factorization loss.
\end{remark}

%% file: apxK_matrix_construction.tex
%

\section{Constructing matrix certificates beyond the proved cases}
\label{appendix:state_dependent_higher_dim}

Appendix~\ref{appendix:uniqueness_state_dependent_lyapunov} shows that inside
the scalar origin-centered exchange-symmetric subclass the certificate is
unique. This appendix asks what the same local stationary-crossing principle
produces when the target contains a second nonzero singular mode. The answer
is constructive: for $X=\diag(1,\sigma)$ the local conditions cut the
admissible quadratics down to a two-parameter family, and matching the
diagonal coefficients of the two blocks selects an explicit one-parameter
branch in closed form.

Throughout, consider an origin-centered quadratic family
\begin{align}
    I(\delta; x) = x^\top P(\delta) x + r(\delta),
\end{align}
for a rank-1 problem with state $x=(a,b,u,v)^\top \in \RR^4$, compatible with
the invariant slice $u=v=0$, on which it restricts to the scalar
factorization family. We take the state space of the higher-dimensional
construction to be the scalar state space $(0,2]$, so that the scalar
uniqueness result supplies the coordinate $\delta$. The local analysis below
determines how the signal block extends to the noise variables.

\subsection{Warm-up: the case \texorpdfstring{$X=\diag(1,0)$}{X=diag(1,0)}}
\label{appendix:diag10_state_dependent_structure}

Write
\begin{align}
    A=\binom{a}{u}, \quad B=\binom{b}{v}, \quad x=(a,b,u,v)^\top \in \RR^4,
\end{align}
and consider the loss
\begin{align}
    \risk_\fac(a,b,u,v)
    =
    \frac12\bigl((ab-1)^2+b^2u^2+a^2v^2+u^2v^2\bigr),
\end{align}
whose global minimizers are $x_*=(a_*,b_*,0,0)^\top$ with $a_*b_*=1$. We look
for a normalized quadratic candidate $I(\delta; x)=x^\top P(\delta) x - 1$
with $P(\delta)\succ0$. As in the scalar case, the local Lagrange condition at
a minimizer is
\begin{align}
    \nabla^2 \risk_\fac(x_*)\,P(\delta) x_* = \lambda\, P(\delta) x_*
\end{align}
for some scalar $\lambda$. The point of this subsection is that the minimizer
slice alone fixes the signal block and leaves the off-signal block
underdetermined.

\begin{proposition}
\label{prop:diag10_structure}
Let $P$ be symmetric positive definite, and let
$x_*=(a_*,b_*,0,0)^\top$ with $a_*b_*=1$ and $a_* \ne b_*$. Suppose that
\begin{align}
    x_*^\top P x_*=1,
    \qquad
    \nabla^2\risk_\fac(x_*)P x_*=\lambda_s P x_*
\end{align}
for some $\lambda_s\in\RR$, and that $x\mapsto x^\top P x$ is invariant under
\begin{align}
    (a,b,u,v)\mapsto (b,a,v,u),
    \qquad
    (a,b,u,v)\mapsto (a,b,-u,v).
\end{align}
Then
\begin{align}
    P=
    \begin{pmatrix}
        c_1 & d   & 0   & 0\\
        d   & c_1 & 0   & 0\\
        0   & 0   & c_2 & 0\\
        0   & 0   & 0   & c_2
    \end{pmatrix},
    \qquad
    c_1=\frac{\delta}{4-\delta^2},
    \quad
    d=-\frac{\delta^2}{2(4-\delta^2)},
    \quad
    \delta := \frac{4a_*^2}{a_*^4+1},
\end{align}
and the coefficient $c_2>0$ is not determined by these conditions.
\end{proposition}

\begin{proof}
Exchange symmetry gives
\begin{align}
    P=
    \begin{pmatrix}
        c_1 & d   & e_1 & e_2\\
        d   & c_1 & e_2 & e_1\\
        e_1 & e_2 & c_2 & f\\
        e_2 & e_1 & f   & c_2
    \end{pmatrix},
\end{align}
and invariance under the sign flip of $u$ forces $e_1=e_2=f=0$, removing every
signal--noise coupling. On the slice $u=v=0$ the normalization and alignment
conditions reduce to their scalar counterparts, whose solution is
\begin{align}
    \begin{pmatrix}
        c_1 & d\\
        d & c_1
    \end{pmatrix}
    =
    \frac{1}{4-\delta^2}
    \begin{pmatrix}
        \delta & -\delta^2/2\\
        -\delta^2/2 & \delta
    \end{pmatrix},
    \qquad
    \delta=\frac{4a_*^2}{a_*^4+1}\in(0,2).
\end{align}
Neither condition involves $c_2$ at $x_*$, and positive definiteness imposes
only $c_2>0$.
\end{proof}

\subsection{Local derivation for \texorpdfstring{$X=\diag(1,\sigma)$}{X=diag(1,sigma)}}
\label{appendix:diag1sigma_state_dependent_structure}

Now take $X=\diag(1,\sigma)$ with $\sigma\in(0,1)$ and the same rank-1
factors. Its objective, GD map, and stationary geometry are catalogued in
Appendix~\ref{subsec:model_catalog_two_mode}. In particular,
\begin{align}
    \risk_\sigma(a,b,u,v)
    =
    \frac12\Bigl((ab-1)^2+b^2u^2+a^2v^2+(uv-\sigma)^2\Bigr).
\end{align}
The global minimizers are still $x_*=(a_*,b_*,0,0)^\top$ with $a_*b_*=1$. In
addition, the points
\begin{align}
    z_*=(0,0,u_*,v_*)^\top,\quad u_*v_*=\sigma,
\end{align}
are stationary. They are not global minimizers, but they are the stationary
slice on which the noise block of a candidate certificate must satisfy the
same local Lagrange condition.

\begin{proposition}
\label{prop:diag1sigma_structure}
Let $P$ be symmetric positive definite, and consider an unbalanced signal
minimizer and an unbalanced noise stationary point,
\begin{align}
    x_*=(a_*,b_*,0,0)^\top,
    \quad a_*b_*=1,
    \quad a_*\ne b_*;
    \qquad
    z_*=(0,0,u_*,v_*)^\top,
    \quad u_*v_*=\sigma,
    \quad u_*\ne v_*.
\end{align}
Suppose that both lie on the normalized level set and satisfy the local
alignment conditions,
\begin{align}
    x_*^\top P x_*=z_*^\top Pz_*=1,
    \qquad
    \nabla^2\risk_\sigma(x_*)Px_*=\lambda_s Px_*,
    \qquad
    \nabla^2\risk_\sigma(z_*)Pz_*=\lambda_n Pz_*
\end{align}
for some $\lambda_s,\lambda_n\in\RR$, and that $x\mapsto x^\top P x$ is
invariant under $(a,b,u,v)\mapsto (b,a,v,u)$. Then
\begin{align}
    P=
    \begin{pmatrix}
        c_1 & d   & 0   & 0\\
        d   & c_1 & 0   & 0\\
        0   & 0   & c_2 & f\\
        0   & 0   & f   & c_2
    \end{pmatrix},
\end{align}
where
\begin{align}
    c_1=\frac{\delta}{4-\delta^2},
    \quad
    d=-\frac{\delta^2}{2(4-\delta^2)},
    \qquad
    c_2=\frac{\xi}{4-\xi^2\sigma^2},
    \quad
    f=-\frac{\xi^2\sigma}{2(4-\xi^2\sigma^2)},
\end{align}
with
\begin{align}
    \delta := \frac{4a_*^2}{a_*^4+1},
    \quad
    \xi := \frac{4u_*^2}{u_*^4+\sigma^2}.
\end{align}
\end{proposition}

\begin{proof}
At $x_*$ the Hessian of $\risk_\sigma$ is block diagonal,
\begin{align}
    H_*=
    \begin{pmatrix}
        b_*^2 & 1 & 0 & 0\\
        1 & a_*^2 & 0 & 0\\
        0 & 0 & b_*^2 & -\sigma\\
        0 & 0 & -\sigma & a_*^2
    \end{pmatrix},
\end{align}
with signal eigenvalues $0$ and $a_*^2+b_*^2$, and noise characteristic
polynomial $\lambda^2-(a_*^2+b_*^2)\lambda+(1-\sigma^2)$. Since
$\sigma\in(0,1)$ we have $1-\sigma^2>0$, so neither $0$ nor $a_*^2+b_*^2$ is a
noise eigenvalue and the two spectra are disjoint.

Write $P$ in the exchange-symmetric form of
Proposition~\ref{prop:diag10_structure}. Alignment requires $Px_*$ to be an
eigenvector of $H_*$; disjointness of the spectra then forbids both components
of $Px_*$ from being nonzero, and its signal component cannot vanish because
$x_*^\top Px_*=1$. Hence
\begin{align}
    e_1a_*+e_2b_*=0,
    \qquad
    e_2a_*+e_1b_*=0,
\end{align}
and $a_*^2\ne b_*^2$ gives $e_1=e_2=0$, so $P$ is block diagonal.

The two blocks are now determined separately. On $u=v=0$ the conditions at
$x_*$ are the scalar ones, giving $c_1$ and $d$ as claimed. At $z_*$, block
diagonality puts $Pz_*$ in the noise subspace, so alignment reduces to
\begin{align}
    \begin{pmatrix}
        v_*^2 & \sigma\\
        \sigma & u_*^2
    \end{pmatrix}
    \begin{pmatrix}
        c_2 & f\\
        f & c_2
    \end{pmatrix}
    \binom{u_*}{v_*}
    =
    \lambda_n
    \begin{pmatrix}
        c_2 & f\\
        f & c_2
    \end{pmatrix}
    \binom{u_*}{v_*},
\end{align}
which together with $c_2(u_*^2+v_*^2)+2fu_*v_*=1$ is the scalar target-$\sigma$
calculation and yields $c_2$ and $f$.
\end{proof}

Proposition~\ref{prop:diag1sigma_structure} says that any normalized quadratic
candidate compatible with the local Lagrange analysis is
\begin{align}
    I(\delta,\xi; a,b,u,v)
    =
    \frac{\delta(a^2+b^2)-\delta^2ab}{4-\delta^2}
    +
    \frac{\xi(u^2+v^2)-\xi^2\sigma uv}{4-\xi^2\sigma^2}
    -1,
    \label{eq:apxK_two_parameter_family}
\end{align}
a two-parameter family: $\delta$ fixes the signal block and $\xi$ the noise
block. A one-parameter state-dependent quadratic certificate requires a
relation $\xi=\xi(\delta)$. The endpoint limits recover the two certificates
of Section~\ref{sec:rank-1 convergence}: at $\sigma=1$, the choice $\xi=\delta$
gives the normalized $\iapx$ certificate, whereas at $\sigma=0$, the choice
$\xi(\delta)=4\delta/(4-\delta^2)$ fixes the otherwise undetermined diagonal
noise-block coefficient of
Subsection~\ref{appendix:diag10_state_dependent_structure} and gives the
normalized $\ifac$ certificate.

\subsection{An analytic one-parameter branch}
\label{subsec:xi_selector}

For every $\sigma\in(0,1)$, diagonal matching yields a branch in closed
form. In both $\ifac$ and $\iapx$ the diagonal quadratic coefficients of the
signal and secondary blocks coincide, so we impose the same on
Eq.~\eqref{eq:apxK_two_parameter_family}: the diagonal quadratic coefficients
of $a^2+b^2$ and $u^2+v^2$ are matched,
\begin{align}
    \frac{\delta}{4-\delta^2}
    =
    \frac{\xi}{4-\sigma^2\xi^2}.
    \label{eq:apxK_coefficient_matching}
\end{align}
Clearing denominators turns this into a quadratic in $\xi$,
\begin{align}
    F_\sigma(\delta,\xi)
    :=
    \sigma^2\delta\,\xi^2+(4-\delta^2)\,\xi-4\delta
    =
    0 .
    \label{eq:apxK_selector_quadratic}
\end{align}

\begin{proposition}
\label{prop:xi_selector}
Fix $\sigma\in(0,1]$. For every $\delta\in(0,2]$,
Eq.~\eqref{eq:apxK_selector_quadratic} has exactly one root in
$(0,\,2/\sigma]$, namely
\begin{align}
    \xi_\sigma(\delta)
    =
    \frac{8\delta}
         {(4-\delta^2)+\sqrt{(4-\delta^2)^2+16\sigma^2\delta^2}}
    =
    \frac{-(4-\delta^2)+\sqrt{(4-\delta^2)^2+16\sigma^2\delta^2}}
         {2\sigma^2\delta}.
    \label{eq:apxK_selector}
\end{align}
The map $\delta\mapsto\xi_\sigma(\delta)$ is real analytic on a neighborhood
of $[0,2]$ and satisfies
\begin{enumerate}
    \item $\xi_\sigma(\delta)\in(0,2/\sigma)$ for $\delta\in(0,2)$, and
    $\xi_\sigma(2)=2/\sigma$;
    \item $\xi_\sigma'(\delta)>0$ on $[0,2]$;
    \item $\xi_\sigma(0)=0$ and $\xi_\sigma'(0)=1$;
    \item $\delta\le\xi_\sigma(\delta)\le 4\delta/(4-\delta^2)$ for
    $\delta\in(0,2)$, with equality on the left exactly when $\sigma=1$;
    \item $\xi_1(\delta)=\delta$, and, for every fixed $\delta\in[0,2)$,
    $\xi_\sigma(\delta)\to4\delta/(4-\delta^2)$ as $\sigma\downarrow0$.
\end{enumerate}
\end{proposition}

\begin{proof}
Since $F_\sigma(\delta,0)=-4\delta<0$ and
\begin{align}
    F_\sigma\!\left(\delta,\tfrac{2}{\sigma}\right)
    =
    4\delta+\frac{2(4-\delta^2)}{\sigma}-4\delta
    =
    \frac{2(4-\delta^2)}{\sigma}
    \;\ge\;0,
\end{align}
with equality exactly at $\delta=2$, the quadratic
$F_\sigma(\delta,\cdot)$, which has positive leading coefficient
$\sigma^2\delta>0$, has exactly one root in
$(0,2/\sigma]$; it is the larger root. The second expression in
Eq.~\eqref{eq:apxK_selector} is the quadratic-formula root, and rationalizing
it gives the first expression. At $\delta=2$ the equation reads
$2\sigma^2\xi^2=8$, so
$\xi_\sigma(2)=2/\sigma$. This proves (i).

For the regularity and (ii), differentiate: on $\{F_\sigma=0\}$,
\begin{align}
    \partial_\xi F_\sigma = 2\sigma^2\delta\xi+(4-\delta^2)>0,
    \qquad
    -\partial_\delta F_\sigma = 4+2\delta\xi-\sigma^2\xi^2>0,
\end{align}
the first because both terms are nonnegative and
$\partial_\xi F_\sigma(2,2/\sigma)=8\sigma>0$,
the second because $\xi\le 2/\sigma$ gives $\sigma^2\xi^2\le4$ with equality
only at $\delta=2$, where $2\delta\xi>0$. Since $F_\sigma$ is a polynomial and
$\partial_\xi F_\sigma>0$ along the branch, the implicit function theorem
gives
a real analytic solution on a neighborhood of $[0,2]$, with
\begin{align}
    \xi_\sigma'(\delta)
    =
    -\frac{\partial_\delta F_\sigma}{\partial_\xi F_\sigma}
    =
    \frac{4+2\delta\xi_\sigma(\delta)-\sigma^2\xi_\sigma(\delta)^2}
         {2\sigma^2\delta\,\xi_\sigma(\delta)+(4-\delta^2)}
    >0 .
    \label{eq:apxK_selector_derivative}
\end{align}
Item (iii) follows from the first expression in Eq.~\eqref{eq:apxK_selector}:
at $\delta=0$, its denominator equals $8$.

For (v), $F_1(\delta,\delta)=\delta^3+(4-\delta^2)\delta-4\delta=0$, so
$\xi_1(\delta)=\delta$. For every fixed $\delta\in[0,2)$, taking
$\sigma\downarrow0$ in the first expression of Eq.~\eqref{eq:apxK_selector}
gives the stated limit. For item (iv), $F_\sigma(\delta,\xi)$ is strictly
increasing in $\sigma$ for $\xi>0$, so by $\partial_\xi F_\sigma>0$ the root
$\xi_\sigma(\delta)$ is strictly decreasing in $\sigma$ on $(0,1]$, and the two
endpoint values just computed are the bounds.
\end{proof}

Eq.~\eqref{eq:apxK_selector} interpolates between the two certificates the
paper proves: $\xi_1(\delta)=\delta$ recovers $\iapx$, and the
$\sigma\downarrow0$ limit recovers $\ifac$. We write
\begin{align}
    I^\sigma(\delta;\,a,b,u,v)
    :=
    I\bigl(\delta,\xi_\sigma(\delta);\,a,b,u,v\bigr)
    \label{eq:apxK_selected_certificate}
\end{align}
for the resulting one-parameter family.

The global boundary-inward property for $I^\sigma$ has not been established.
Subsection~\ref{apx:xi_selector_verification} provides finite-sample numerical
evidence only.

\subsection{Structural axioms and terminal geometry along the branch}
\label{subsec:diag1sigma_terminal_reduction}

This subsection identifies the terminal set of $I^\sigma$ and verifies
Axioms~\ref{cond:P_is_pd}--\ref{cond:level_set_as_state}, together with the
required $C^1$ regularity on $(0,2)$.

\begin{proposition}\label{prop:diag1sigma_K2_reduction}
For $\sigma\in(0,1)$, the terminal set of the family $I^\sigma$ is
\begin{align}
    K_2^\sigma
    =
    \left\{
        (a,b,u,v)\in\RR^4 :
        a=b,\ u=v,\ a^2+\frac{u^2}{\sigma}\le 2
    \right\},
    \label{eq:apxK_K2_sigma}
\end{align}
\end{proposition}

\begin{proof}
Write $\xi=\xi_\sigma(\delta)$ and split the certificate as
\begin{align}
    I^\sigma(\delta;\,x)+1
    =
    \frac{\delta(a-b)^2}{(2-\delta)(2+\delta)}
    +
    \frac{\xi(u-v)^2}{(2-\xi\sigma)(2+\xi\sigma)}
    +
    \frac{\delta\,ab}{2+\delta}
    +
    \frac{\xi\,uv}{2+\xi\sigma}.
    \label{eq:apxK_K2_split}
\end{align}
Suppose $x$ lies in every $K_\delta$, $\delta<2$. Then the left-hand side of
Eq.~\eqref{eq:apxK_K2_split} is at most $1$ for every $\delta<2$. As
$\delta\uparrow2$, both $(2-\delta)(2+\delta)$ and, by
Proposition~\ref{prop:xi_selector}(i),
$(2-\xi\sigma)(2+\xi\sigma)$ tend to $0$, while the last two terms converge to
$ab/2$ and $uv/(2\sigma)$. Since the first two terms are nonnegative, the
upper bound forces $a=b$ and $u=v$. Substituting these equalities gives
\begin{align}
    I^\sigma(\delta;\,a,a,u,u)
    =
    \frac{\delta a^2}{2+\delta}
    +
    \frac{\xi_\sigma(\delta)u^2}{2+\sigma\xi_\sigma(\delta)}
    -1,
    \label{eq:apxK_balanced_slice_value}
\end{align}
which is nonpositive for every $\delta<2$. Letting $\delta\uparrow2$ and using
$\xi_\sigma(2)=2/\sigma$ yields $a^2+u^2/\sigma\le2$.

Conversely, suppose $a=b$, $u=v$, and $a^2+u^2/\sigma\le2$. For every
$\delta\in(0,2)$, the coefficients in
Eq.~\eqref{eq:apxK_balanced_slice_value} obey
$\delta/(2+\delta)\le1/2$ and
$\xi_\sigma(\delta)/(2+\sigma\xi_\sigma(\delta))\le1/(2\sigma)$. Hence
$I^\sigma(\delta;\,a,a,u,u)\le a^2/2+u^2/(2\sigma)-1\le0$.
\end{proof}

\paragraph{Compactness and terminal negligibility.}
Axiom~\ref{cond:P_is_pd} holds because the normalized matrix of
Eq.~\eqref{eq:apxK_two_parameter_family} is positive definite for
$\delta\in(0,2)$ and $\xi_\sigma(\delta)\in(0,2/\sigma)$, which is
Proposition~\ref{prop:xi_selector}(i); its $-1$ sublevel set is therefore a
nonempty compact ellipsoid. By
Proposition~\ref{prop:diag1sigma_K2_reduction}, the terminal set lies in the
proper linear subspace $\{a=b,\ u=v\}$ and therefore has measure zero. This is
Axiom~\ref{cond:end_point_measure_zero}.

\paragraph{Nesting and the state parameter.}
In the sum-and-difference coordinates
\begin{align}
    s_\pm := \frac{a\pm b}{\sqrt 2}, \qquad
    n_\pm := \frac{u\pm v}{\sqrt 2},
\end{align}
the certificate diagonalizes,
\begin{align}
    I^\sigma(\delta;\,a,b,u,v)
    =
    \lambda_{s,+}s_+^2+\lambda_{s,-}s_-^2
    +\lambda_{n,+}n_+^2+\lambda_{n,-}n_-^2-1,
\end{align}
with
\begin{align}
    \lambda_{s,\pm}(\delta)
    =
    \frac{\delta}{2(2\pm\delta)},
    \qquad
    \lambda_{n,\pm}(\delta)
    =
    \frac{\xi_\sigma(\delta)}{2(2\pm\sigma\xi_\sigma(\delta))}.
\end{align}
The two signal eigenvalues are strictly increasing in $\delta$. The maps
$r\mapsto r/(2(2\pm\sigma r))$ are strictly increasing on $(0,2/\sigma)$, and
$\xi_\sigma$ is strictly increasing by
Proposition~\ref{prop:xi_selector}(ii), so the two noise eigenvalues are
strictly increasing as well. Hence for $\delta<\delta'$ the ellipsoid
$\{I^\sigma(\delta';\cdot)\le0\}$ is strictly contained in
$\{I^\sigma(\delta;\cdot)<0\}$ whenever $\delta'<2$.
For a point of $K_2^\sigma$, Eq.~\eqref{eq:apxK_balanced_slice_value} gives
$I^\sigma(\delta;\cdot)=-1$ at the origin. At every other point of
$K_2^\sigma$, at least one of $a$ and $u$ is nonzero, and the two coefficients
in that equation are strictly below $1/2$ and $1/(2\sigma)$, respectively.
Thus
\begin{align}
    I^\sigma(\delta;\,a,a,u,u)
    <
    \frac{a^2}{2}+\frac{u^2}{2\sigma}-1
    \le0
\end{align}
for every $\delta<2$. Hence $K_2^\sigma\subseteq
\operatorname{int}K_\delta$, completing Axiom~\ref{cond:nesting}.

For Axiom~\ref{cond:level_set_as_state}, take a point outside $K_2^\sigma$. If
$a\ne b$ or $u\ne v$, then
\begin{align}
    \lim_{\delta\downarrow 0} I^\sigma(\delta;\,a,b,u,v)=-1,
    \qquad
    \lim_{\delta\uparrow 2} I^\sigma(\delta;\,a,b,u,v)=+\infty,
\end{align}
the first by $\xi_\sigma(0)=0$ and the second because
$\lambda_{s,-}$ or $\lambda_{n,-}$ blows up. If $a=b$ and $u=v$ with
$a^2+u^2/\sigma>2$, then Eq.~\eqref{eq:apxK_balanced_slice_value} tends to
$-1$ as $\delta\downarrow0$ and to
$a^2/2+u^2/(2\sigma)-1>0$ as $\delta\uparrow2$. In both cases continuity and
strict monotonicity give a unique $\delta\in(0,2)$ with
$I^\sigma(\delta;\cdot)=0$. Points of $K_2^\sigma$ are assigned the terminal
state $\delta=2$.

\paragraph{Regularity on the nonterminal range.}
By Proposition~\ref{prop:xi_selector}, the branch $\xi_\sigma$ is real
analytic on a neighborhood of $[0,2]$. Hence the coefficients of
$I^\sigma$ are $C^1$ on $(0,2)$, as required by the definition of a
state-dependent quadratic certificate.

Along $\xi_\sigma$ the geometric side of the framework is therefore complete:
the structural axioms hold, the terminal set is
Eq.~\eqref{eq:apxK_K2_sigma}, and the coefficients are $C^1$ on the nonterminal
range. The stationary set $\mathcal S_\sigma$ of
Appendix~\ref{subsec:model_catalog_two_mode} has measure zero, and $\gd_\eta$
is a submersion almost everywhere by
Proposition~\ref{prop:gd_submersion_ae}. What is missing is only the dynamical
pair, Axioms~\ref{cond:monotonicity} and~\ref{cond:stationarity} along the
branch above some threshold $\delta_{\mathrm{th}}$; with those,
Theorem~\ref{thm:abstract_state_dependent_convergence} applies verbatim.

\subsection{Reduced dynamics on the terminal set}
\label{subsec:diag1sigma_terminal_dynamics}

Fix $\eta\in(0,\tfrac12)$ and assume that $I^\sigma$ satisfies
Axiom~\ref{cond:monotonicity} above some threshold
$\delta_{\mathrm{th}}<2$. Under this assumption,
Lemma~\ref{lem:terminal_forward_invariance} makes $K_2^\sigma$ forward
invariant. Thus an orbit initialized in $K_2^\sigma$ satisfies $a_t=b_t$,
$u_t=v_t$, and $a_t^2+u_t^2/\sigma\le2$ for every $t\ge0$. The GD update
reduces to
\begin{align}
    a_{t+1}=\alpha_t a_t,\quad u_{t+1}=\beta_t u_t,
    \qquad
    \alpha_t:=1-\eta(a_t^2+u_t^2)+\eta,\quad
    \beta_t:=1-\eta(a_t^2+u_t^2)+\eta\sigma.
\end{align}
Since $a_t^2+u_t^2\le a_t^2+u_t^2/\sigma\le2$ and $\sigma\in(0,1)$,
$\alpha_t\ge1-\eta>0$ and $\beta_t\ge1-2\eta+\eta\sigma>0$, while
$\alpha_t-\beta_t=\eta(1-\sigma)>0$ and $\alpha_t\le1+\eta$. Hence whenever
$a_t\ne0$,
\begin{align}
    \left|\frac{u_{t+1}}{a_{t+1}}\right|
    =
    \frac{\beta_t}{\alpha_t}\left|\frac{u_t}{a_t}\right|,
    \qquad
    0<\frac{\beta_t}{\alpha_t}
    \le
    1-\frac{\eta(1-\sigma)}{1+\eta}<1,
\end{align}
and positivity of $\alpha_t$ propagates $a_t\ne0$, so
\begin{align}
    \left|\frac{u_t}{a_t}\right|
    \le
    \left(1-\frac{\eta(1-\sigma)}{1+\eta}\right)^t
    \left|\frac{u_0}{a_0}\right|
    \longrightarrow0.
\end{align}

Thus, for orbits in $K_2^\sigma$ with a nonzero signal coordinate, the
secondary mode is damped geometrically relative to the signal mode, even
though $K_2^\sigma$, unlike the terminal sets of $\ifac$ and $\iapx$, does not
reduce to the slice $u=v=0$. The ratio argument does not cover the invariant,
measure-zero noise-only slice $\{a=b=0\}$; its only stationary points are the
origin and the points satisfying $uv=\sigma$.

This conclusion concerns orbits initialized in $K_2^\sigma$. Extending it to
trajectories that only approach $K_2^\sigma$ would require an additional
asymptotic-reduction argument, which we do not pursue here.

%% file: apxL_matrix_numerics.tex

\section{Numerical protocols and evidence for the matrix models}
\label{apx:numerics}

This appendix reports numerical evidence for the matrix models.
Subsection~\ref{apx:evidence_labels} distinguishes the status of the reported
evidence, and Subsection~\ref{apx:numerical_protocols} defines the common
protocols. No numerical result is used to prove a theorem.

\subsection{Evidence labels}
\label{apx:evidence_labels}

We use four labels to distinguish the status of the claims and experiments
reported in this appendix and Appendix~\ref{apx:objective_extensions}.

\begin{itemize}
    \item \textbf{Theorem-backed.} A theorem of this paper supplies the
    convergence or terminal statement; the numerics only report that the
    proved behavior is observed.
    \item \textbf{A5-proven.} The boundary-inward property
    (Axiom~\ref{cond:monotonicity}) is proved, but no full-convergence theorem
    is claimed on that basis alone.
    \item \textbf{Candidate (A5-sampled).} Axiom~\ref{cond:monotonicity} is
    tested only on a finite set of boundary points.
    \item \textbf{Exploratory.} Finite-horizon behavior only; no theorem, no
    necessity claim, no divergence claim.
\end{itemize}

\subsection{Common numerical protocols}
\label{apx:numerical_protocols}

\subsubsection{Trajectory and basin classification}
\label{apx:trajectory_protocol}

Unless a subsection states otherwise, initializations are drawn with each
coordinate uniform on $[-2.5,2.5]$; sample counts, grid resolutions and seeds
are reported per experiment. Each trajectory is run for a horizon $T$ and
classified against a \emph{target behavior} that depends on the model and the
step-size regime:

\begin{itemize}
    \item pre-critical factorization and approximation: distance to the
    minimizer set;
    \item post-critical factorization: approach to the terminal set and
    lock-on to the period-$2$ orbit of the full dynamics;
    \item post-critical approximation: minimizer convergence and bounded
    nonconvergence, recorded as separate outcomes.
\end{itemize}

A trajectory is labelled \emph{escaped} only when an iterate is nonfinite or
leaves an explicitly stated escape radius before the horizon, and escape is
recorded as an outcome in its own right.

Certificate membership is evaluated at the initialization. Under the same
sampling measure we report the empirical soundness
$\widehat P(\text{target}\mid\text{cert})$ and the basin coverage
$\widehat P(\text{cert}\mid\text{target})$.

\subsubsection{Sampled boundary-inward test}
\label{apx:a5_protocol}

For a quadratic candidate written in normalized form
\begin{align}
    I(\delta;x)=(x-c(\delta))^\top P(\delta)\,(x-c(\delta))-1,
    \qquad P(\delta)\succ0,
    \label{eq:apxL_normalized_candidate}
\end{align}
the boundary is parametrized exactly: for any nonzero direction $d$,
\begin{align}
    x(d)=c(\delta)+P(\delta)^{-1/2}\frac{d}{\|d\|_2}
    \label{eq:apxL_boundary_point}
\end{align}
satisfies $I(\delta;x(d))=0$. Given a fixed finite set of directions
$d_1,\dots,d_N$ we report the sampled worst post-step value
\begin{align}
    \widehat M_N(\delta)
    :=
    \max_{1\le j\le N}
    I\bigl(\delta;\gd_\eta(x(d_j))\bigr).
    \label{eq:apxL_sampled_max}
\end{align}
The candidate \emph{passes} at $\delta$ when
$\widehat M_N(\delta)\le\mathrm{tol}$, with an acceptance tolerance
$\mathrm{tol}>0$ that absorbs the arithmetic error of evaluating
$I(\delta;\cdot)\circ\gd_\eta$ and is stated per experiment. A pass means that
no counterexample was found in the sampled set, never a proof.

Two caveats hold throughout. First, the direction sets used below are chosen
families, deterministic or seeded random, and none of them is a uniform sample
of the boundary. Second, Axiom~\ref{cond:stationarity} predicts exact equality
$I(\delta;\gd_\eta(x))=0$ at stationary boundary crossings, so a worst sampled
value near $0$ is the expected behavior of an admissible candidate rather than
a weak margin; $-\widehat M_N$ is not an inward margin.

\subsubsection{Empirical threshold}
\label{apx:discovery_validation}

The sampled test of Subsection~\ref{apx:a5_protocol} is run over a finite grid
$\mathcal G$ of states. The grid need not be uniform: after locating a
transition on an initial grid, we may augment it locally with additional state
values. In such runs, $\mathcal G$ denotes the final augmented grid. Whenever
the following set is nonempty, we define
\begin{align}
    \delta_{\mathrm{th}}^{\mathrm{emp}}
    :=
    \min\bigl\{
        \delta\in\mathcal G:
        \widehat M_N(\delta')\le\mathrm{tol}
        \text{ for every }\delta'\in\mathcal G\text{ with }\delta'\ge\delta
    \bigr\}
    \label{eq:apxL_empirical_threshold}
\end{align}
with $\widehat M_N$ taken for the candidate family under test. If the set is
empty, we report no empirical threshold.
Axiom~\ref{cond:monotonicity} asks for the boundary-inward property at every
state above the threshold. Thus $\delta_{\mathrm{th}}^{\mathrm{emp}}$ is the
lower endpoint of a terminal passing band on the finite grid, not the smallest
isolated state that passes. Any associated empirical coverage calculation uses
this final-grid threshold. It makes no claim about the continuum.

\subsubsection{Reproducibility}
\label{apx:sensitivity}

Every number reported in this appendix and in
Appendix~\ref{apx:objective_extensions} should be recoverable with the
provided code, which records the effective configuration of each run.
Horizons, tolerances and grid resolutions are chosen per experiment, and
Remark~\ref{rmk:classification_sensitivity} records how little the reported
classifications depend on them.

\subsection{Soundness and basin coverage (theorem-backed)}
\label{apx:ifac_tightness}

\subsubsection{An analytic witness that the certificate is not tight}
\label{apx:nonsharpness_witness}

The certified region does not exhaust the convergence basin, and this can be
shown analytically.

Let $\eta\in(0,1)$, put $c_0:=(4-\eta^2)/2\in(3/2,2)$, and take $a_0=b_0=0$,
$u_0=v_0=s$ with $c:=\eta s^2$. Then
\begin{align}
    \ifac(\eta;\,x_0) = 2\eta s^2+\eta^2-4 = 2(c-c_0),
\end{align}
while one gradient step fixes the signal coordinates and maps $u_0=v_0=s$ to
$u_1=v_1=(1-c)s$, so
\begin{align}
    \ifac(\eta;\,x_1) = 2c(1-c)^2+\eta^2-4 = 2c(1-c)^2-2c_0 .
\end{align}
At $c=c_0$ the first expression vanishes and the second equals
$2c_0^2(c_0-2)<0$, so choosing $c>c_0$ close enough makes
$\ifac(\eta;\,x_0)>0$ and $\ifac(\eta;\,x_1)<0$: the point is not certified,
and one step certifies it.

The point itself lies on the invariant slice $a=b=0$ and cannot converge to a
minimizer, but both inequalities are strict and the maps involved are
polynomial, so they persist on an open neighborhood $U$. Let $E$ be the
measure-zero set of certified initializations excluded by
Theorem~\ref{thm:convergence_region_rank_1}; by
Corollary~\ref{cor:gd_preimage_null} the set $\gd_\eta^{-1}(E)$ is null, so
almost every $x_0\in U$ lies outside the certified region and converges to a
global minimizer. Thus the certificate is sufficient but not necessary, as
Remark~\ref{rmk:convergence_region_not_sharp} states. Under any absolutely
continuous sampling law that assigns positive mass to $U$, the basin coverage
is therefore strictly below $100\%$.

\subsubsection{Quantitative coverage}
\label{apx:coverage_table}

Table~\ref{tab:matrix_coverage} reports, for the cases the theorems cover, the
fraction of the sampling box meeting the target behavior, the empirical
soundness, and the basin coverage, under the protocol of
Subsection~\ref{apx:trajectory_protocol}. Here $\mathrm{tgt}$ abbreviates the
target behavior of the row, $\mathrm{cert}$ certificate membership at the
initialization, and the target column is a percentage of the sampling box.

\begin{table}[t]
    \centering
    \caption{Empirical soundness and basin coverage of the theorem-backed
    certificates. Each row draws $N=1.5\times10^5$ initializations with every
    coordinate uniform on $[-2.5,2.5]$ (seed $0$), runs $T=4000$ steps, and
    classifies against the stated target with tolerance $10^{-3}$; a trajectory
    counts as escaped, and so as a non-target outcome, only if it leaves
    squared radius $10^{12}$ before the horizon. The one soundness entry below
    $1$ is a horizon artifact: at $\eta=0.99$ a few certified initializations
    need a longer run than $T=4000$. $\dagger$: sampling-limited, see the
    text.}
    \label{tab:matrix_coverage}
    \small
    \begin{tabular}{llccccc}
        \toprule
        Problem & Target & $\eta$ & $\delta_{\mathrm{th}}$
        & Target \%
        & $\widehat P(\text{tgt}\mid\text{cert})$
        & $\widehat P(\text{cert}\mid\text{tgt})$ \\
        \midrule
        Factorization & minimizer & $0.30$ & $0.300$ & $88.81$ & $1.0000$ & $0.9963$ \\
                      & minimizer & $0.60$ & $0.600$ & $31.60$ & $1.0000$ & $0.9555$ \\
                      & minimizer & $0.99$ & $0.990$ & $9.69$  & $0.9998$ & $0.8713$ \\
                      & $2$-cycle & $1.20$ & $1.200$ & $5.52$ & $1.0000$ & $0.8041$ \\
        \midrule
        Approximation, $k=2$ & minimizer & $0.30$ & $0.307$ & $56.45$ & $1.0000$ & $0.9584$ \\
                             & minimizer & $0.60$ & $0.667$ & $7.84$  & $1.0000$ & $0.7161$ \\
                             & minimizer & $0.99$ & $1.735$ & $1.06$  & $1.0000^\dagger$ & $0.0245^\dagger$ \\
        \midrule
        Approximation, $k=3$ & minimizer & $0.30$ & $0.307$ & $25.53$ & $1.0000$ & $0.9464$ \\
                             & minimizer & $0.60$ & $0.667$ & $1.42$  & $1.0000$ & $0.6892$ \\
                             & minimizer & $0.99$ & $1.735$ & $0.08$  & ---$^\dagger$ & $0.0000^\dagger$ \\
        \bottomrule
    \end{tabular}
\end{table}

All but five samples are classified as either target or escaped. In the
factorization rows, one sample at $\eta=0.6$ and four at $\eta=0.99$ remain
bounded without meeting the target tolerance by the stated horizon. Thus the
target column agrees with the non-escaping fraction up to these five samples.

The approximation target fraction falls as $\eta\uparrow1$, from $56.45\%$ to
$1.06\%$ for $k=2$, and coverage falls with it as the threshold
$\delta_{\mathrm{th}}$ of Theorem~\ref{thm:convergence_region_rank_1_approx}
approaches its endpoint. The two approximation rows at $\eta=0.99$ are
limited by the sampling: at $\delta_{\mathrm{th}}=1.735$ the certified
ellipsoid holds only $39$ of the samples for $k=2$ and none for $k=3$. Thus,
under uniform box sampling, the certified region covers only a small fraction
of the empirical basin; the sparse certified draws also make the $k=2$
soundness estimate imprecise and leave the $k=3$ soundness unestimated. The
ellipsoid occupies $2.6\times10^{-4}$ of the box for $k=2$ and still less for
$k=3$, so a direct assessment of soundness should instead sample inside the
ellipsoid.

The post-critical factorization row targets the period-$2$ behavior of
Corollary~\ref{cor:theorem_five_period_two_inheritance} rather than minimizer
convergence.

\begin{figure}[t]
    \centering
    \includegraphics[width=\linewidth]{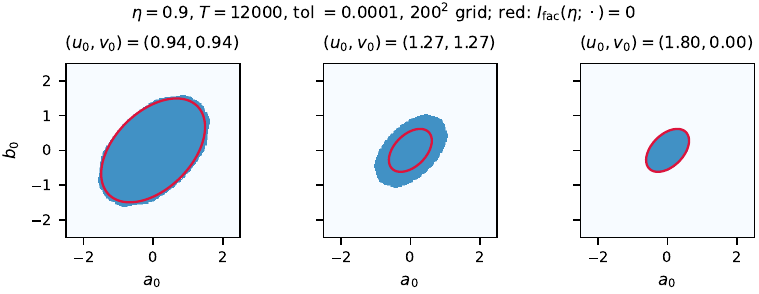}
    \caption{Certified region and empirical basin in the $(a_0,b_0)$ plane at
    fixed off-signal coordinates, at $\eta=0.9$. Each panel is a $200^2$ grid on
    $[-2.5,2.5]^2$ run for $T=12000$ steps with classification tolerance
    $10^{-4}$; blue marks the on-target initializations and the red curve is
    $\ifac(\eta;\cdot)=0$. The last two panels isolate the effect of noise
    alignment: they have the same $u_0^2+v_0^2=3.24$, and hence the same certified area,
    $4.25\%$ of the panel, but the on-target area is $11.74\%$ on the balanced
    slice against $4.25\%$ on the axis, where the basin and the certified
    region coincide. This is consistent with the term
    $-(\eta\delta)^2u_t^2v_t^2$ in the remainder $R_t^\fac(\delta)$ of
    Eq.~\eqref{eq:R_t_def} being more negative there.}
    \label{fig:matrix_coverage}
\end{figure}

\begin{remark}[Sensitivity to the classification settings]
\label{rmk:classification_sensitivity}
Panels of the kind in Figure~\ref{fig:matrix_coverage} were classified at
$\eta\in\{0.2,0.4,0.8,1.2\}$ and five off-signal pairs, twenty panels on a
$200^2$ grid and so $8\times10^5$ classifications in total. Moving
$(T,\mathrm{tol})$ from $(400,10^{-2})$ to $(1600,10^{-4})$ changes $18$ of
them, a fraction $2.3\times10^{-5}$; changing from the baseline to the
intermediate setting $(800,10^{-3})$ changes $22$. Doubling a panel to $400^2$
changes its empirically on-target area by at most $0.40$ percentage points.
These checks indicate that the qualitative comparisons are insensitive to the
reported classification settings.
\end{remark}

\subsection{Verification along the diagonal-matching branch \texorpdfstring{$\xi_\sigma$}{xi\_sigma} (A5-sampled)}
\label{apx:xi_selector_verification}
\label{apx:xi_delta_experiment}

Proposition~\ref{prop:xi_selector} settles the analytic side: the branch
exists, is smooth and strictly increasing, and reaches $2/\sigma$ at
$\delta=2$. The question left for numerics is narrower, and this subsection
answers only that: does the branch pass the sampled boundary-inward test?

\paragraph{Setup.}
The model is the two-dimensional approximation problem $X=\diag(1,\sigma)$ of
Appendix~\ref{subsec:model_catalog_two_mode}. The candidate is
$I^\sigma(\delta;\cdot)=I(\delta,\xi_\sigma(\delta);\cdot)$ of
Eq.~\eqref{eq:apxK_selected_certificate}, which is origin-centered, so
$c(\delta)=0$ in Eq.~\eqref{eq:apxL_normalized_candidate} and boundary points
come from Eq.~\eqref{eq:apxL_boundary_point} at each $\delta$ with $\sigma$
held fixed. Each state is tested on $2\times10^5$ seeded Gaussian directions
(seed $11$) with $\mathrm{tol}=10^{-4}$; the deterministic $\ell^\infty$ grid
of $13920$ rays is the coarse grid on which the two-parameter map below is
drawn.

\paragraph{Results.}
We scan $195$ states $\delta\in[0.05,1.95]$ for
$\sigma\in\{0.2,0.5,0.8\}$ and $\eta\in\{0.2,0.6\}$. In every one of the six
cases the sampled-passing set is an upper tail, and its lower endpoint,
Eq.~\eqref{eq:apxL_empirical_threshold}, is

\begin{center}
\begin{tabular}{lcccc}
    \toprule
    & $\sigma=0.2$ & $\sigma=0.5$ & $\sigma=0.8$
    & $\delta_\mathrm{th}^\apx$ \\
    \midrule
    $\eta=0.2$ & $0.200$ & $0.201$ & $0.201$ & $0.202$ \\
    $\eta=0.6$ & $0.602$ & $0.615$ & $0.641$ & $0.667$ \\
    \bottomrule
\end{tabular}
\end{center}

Note that $\eta$ is the threshold of $\ifac$, the model at $\sigma=0$, and that
the last column is $\delta_\mathrm{th}^\apx=2(1-\sqrt{1-\eta^2})/\eta$, the
threshold of $\iapx$, the model at $\sigma=1$. Every measured value falls
strictly between the two, and
$\delta_{\mathrm{th}}^{\mathrm{emp}}$ increases with $\sigma$. At $\eta=0.2$
the two ends are $0.200$ and $0.202$, too close to separate at three decimals;
the measured thresholds there are $0.2001$, $0.2005$ and $0.2013$.
Above the threshold no violation is found: at
$\sigma=0.5$, $\eta=0.6$ the worst sampled post-step value over the passing
states is $-1.8\times10^{-6}$, and below the threshold the test fails by a wide
margin.

Figure~\ref{fig:xi_selector_a5} shows, at $\sigma=0.5$, $\eta=0.6$, that
$\xi_\sigma$ falls inside the sampled-admissible region of the two-parameter
family over the tested states; that region is tested only on the coarse
$\ell^\infty$ grid.

\begin{figure}[t]
    \centering
    \includegraphics[width=\linewidth]{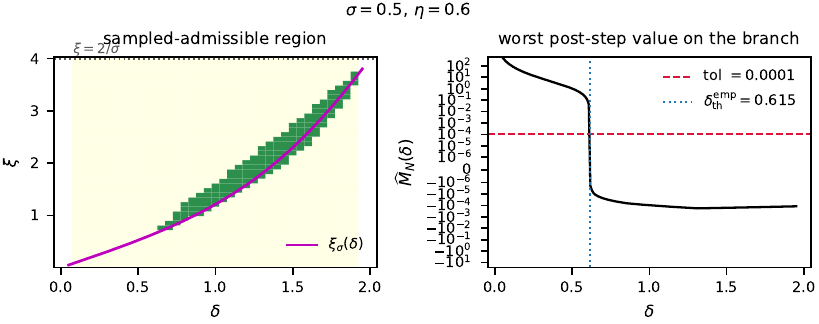}
    \caption{Sampled boundary-inward test along the diagonal-matching branch
    $\xi_\sigma$ at $\sigma=0.5$, $\eta=0.6$. Left: sampled-admissible region
    of the two-parameter family in $(\delta,\xi)$, tested on the $13920$-ray
    deterministic $\ell^\infty$ grid, with $\xi_\sigma$ overlaid.
    Right: the worst sampled post-step value $\widehat M_N(\delta)$ of
    Eq.~\eqref{eq:apxL_sampled_max} along the branch, on $2\times10^5$ seeded
    Gaussian directions, with the acceptance tolerance $10^{-4}$ and the
    threshold $\delta_{\mathrm{th}}^{\mathrm{emp}}=0.615$ marked. Evidence:
    candidate (A5-sampled); absence of a violation is finite-sample evidence,
    not a proof of Axiom~\ref{cond:monotonicity}.}
    \label{fig:xi_selector_a5}
\end{figure}

\subsection{General-spectrum candidates (A5-sampled)}
\label{apx:general_spectrum}

Applying the diagonal-matching rule of
Eq.~\eqref{eq:apxK_coefficient_matching} blockwise suggests a candidate for any
diagonal spectrum, with one block parameter per singular direction. The
construction is not forced by the axioms, no global proof of
Axiom~\ref{cond:monotonicity} is available for it, and what follows tests that
axiom on sampled boundary points only.

We test $X=\diag(1,0.75,0.5,0)$ and
$X=\diag(1,0.8,0.6,0.4,0)$ at $\eta=0.3$, in state spaces of dimension $8$ and
$10$. A coarse pass, on states $0.05$ apart and $2\times10^4$ seeded
rays, puts the threshold at $0.35$; a fine pass, on $8\times10^5$ seeded rays,
brings it to $0.304$ for both spectra. In these dimensions that many rays
still probe the boundary far less finely than the $2\times10^5$ of
Subsection~\ref{apx:xi_selector_verification} do in dimension $4$, so the
evidence here is weaker than there.

\begin{table}[t]
    \centering
    \caption{General-spectrum candidates from blockwise diagonal matching.
    Both rows use $\eta=0.3$, $N=1.5\times10^5$ initializations uniform on
    $[-2.5,2.5]$ per coordinate (seed $21$), $T=4000$, classification tolerance
    $10^{-3}$, and acceptance tolerance $10^{-4}$ on $8\times10^5$ seeded
    Gaussian directions in the fine pass; the target ($\mathrm{tgt}$) is
    convergence to the top singular mode, and $\mathrm{cert}$ is certificate
    membership at the initialization. Axiom~\ref{cond:monotonicity}
    is unproved for these spectra: $\delta_{\mathrm{th}}^{\mathrm{emp}}$ is the
    sampled threshold of Eq.~\eqref{eq:apxL_empirical_threshold}, not the
    $\delta_{\mathrm{th}}$ of Theorem~\ref{thm:convergence_region_rank_1_approx}.
    For reference, that theorem gives $\delta_\mathrm{th}^\apx=0.307$ at this
    step size.}
    \label{tab:general_spectrum}
    \small
    \begin{tabular}{lccccc}
        \toprule
        Spectrum of $X$ & $\eta$
        & $\delta_{\mathrm{th}}^{\mathrm{emp}}$
        & Target \%
        & $\widehat P(\text{tgt}\mid\text{cert})$
        & $\widehat P(\text{cert}\mid\text{tgt})$ \\
        \midrule
        $(1,\,0.75,\,0.5,\,0)$ & $0.3$ & $0.304$
            & $26.42$ & $1.0000$ & $0.9444$ \\
        $(1,\,0.8,\,0.6,\,0.4,\,0)$ & $0.3$ & $0.304$
            & $9.43$ & $1.0000$ & $0.9149$ \\
        \bottomrule
    \end{tabular}
\end{table}

Both thresholds fall below the isotropic-signal value
$\delta_\mathrm{th}^\apx=0.307$. The six thresholds in
Subsection~\ref{apx:xi_selector_verification} likewise approach their
isotropic-signal benchmarks from below as $\sigma\uparrow1$. Thus the
isotropic-signal threshold is the largest benchmark in every reported
comparison.

\begin{conjecture}[Convergence region for general-spectrum rank-1 approximation]
\label{conj:equal_spectrum_extremal}
Let $X=\diag(\sigma_1,\dots,\sigma_m)$ with
$1=\sigma_1\ge\sigma_2\ge\dots\ge\sigma_m\ge0$, let $(A_0,B_0)$ be an
initialization and let $\eta\in(0,1)$. For $\sigma\in[0,1]$ and $\xi>0$ put
\begin{align}
    P_\sigma(\xi)
    :=
    \frac{\xi}{4-\sigma^2\xi^2}
    \begin{pmatrix} 1 & -\sigma\xi/2 \\ -\sigma\xi/2 & 1 \end{pmatrix},
    \label{eq:apxL_block_certificate_matrix}
\end{align}
let $\xi_\sigma$ be the diagonal-matching branch of
Proposition~\ref{prop:xi_selector}, extended to $\sigma=0$ by its limit
$\xi_0(\delta)=4\delta/(4-\delta^2)$, and set
\begin{align}
    I(\delta;x)
    :=
    \sum_{i=1}^m x_i^\top P_{\sigma_i}\bigl(\xi_{\sigma_i}(\delta)\bigr)\,x_i-1,
    \qquad x_i=(a_i,b_i),
    \qquad
    \delta_{\mathrm{th}}^\apx
    :=
    \frac{2\bigl(1-\sqrt{1-\eta^2}\bigr)}{\eta}.
    \label{eq:apxL_conjectured_certificate}
\end{align}
Suppose that the initialization satisfies
$I(\delta_{\mathrm{th}}^\apx;\,(A_0,B_0))<0$. Then, for almost every such
initialization the GD dynamics $(A_t,B_t)_{t\ge0}$ converge to the set of
global minimizers.
\end{conjecture}

The conjectured threshold is the one in
Theorem~\ref{thm:convergence_region_rank_1_approx}. The claim is that this same
threshold remains sufficient for the spectrum-dependent candidate, with no
dependence on $m$ or on the smaller singular values. The specialization
$X=\diag(I_d,0)$ recovers that theorem, while
$\sigma_2=\dots=\sigma_m=0$ recovers
Theorem~\ref{thm:convergence_region_rank_1}, whose proved region is strictly
larger than the conjectured one.

\subsection{Finite-horizon post-critical contrast with approximation
(exploratory)}
\label{apx:finite_time_dynamics}
\label{apx:postcritical}

Theorem~\ref{thm:convergence_region_rank_1}(2) localizes almost every certified
factorization trajectory near the compact terminal set for $\eta\in(1,2)$.
By contrast, Theorem~\ref{thm:convergence_region_rank_1_approx} does not apply
to $X=\diag(I_k,0)$ with $k\ge2$ and $\eta>1$: $q_\eta(\delta)>0$ at every
state, so the sign condition used to prove the boundary-inward property fails.
This subsection reports how the tested trajectories behave there. In both
problems, every coordinate is initialized from $\mathcal N(0,s^2)$, with the
scale swept to check that the comparison is not specific to one choice:
$s\in\{0.1,0.3,0.6,1.0\}$, $\eta\in\{1.05,1.2,1.5\}$, $N=2\times10^4$ per cell
(seed $7$), horizon $600$, escape radius squared $10^{12}$.

The separation is complete across the tested scales. At every one of the
twelve cells, no approximation trajectory remains bounded to the horizon;
the median escape time falls from $190$ steps at $(s,\eta)=(0.1,1.05)$ to $4$
steps at $(1.0,1.05)$. In every factorization cell, the bounded fraction is at
least the certified fraction, consistently with
Theorem~\ref{thm:convergence_region_rank_1}(2). Figure~\ref{fig:apx_post_critical}
shows the two norm histories at $s=0.3$, each band spanning the $10$th to
$90$th percentile of $\|x_t\|^2$ over the trajectories still bounded at that
step.

What is observed is that the tested approximation trajectories leave the
stated radius within the stated horizon. That is not a proof of divergence, it
is not evidence that $q_\eta(\delta)<0$ is necessary, and it says nothing
about trajectories outside the tested set.

\begin{figure}[t]
    \centering
    \includegraphics[width=\linewidth]{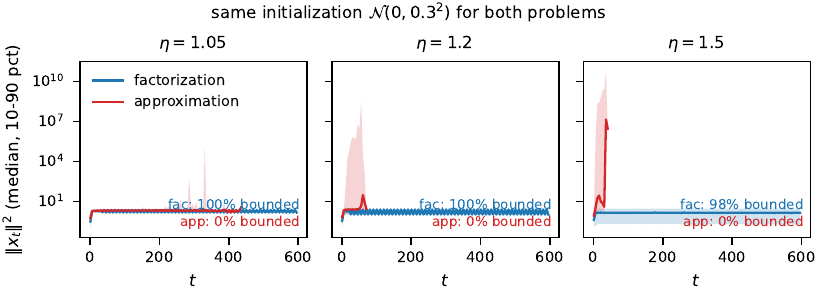}
    \caption{Finite-horizon squared norm for factorization ($X=\diag(1,0)$) and
    for isotropic-signal rank-1 approximation ($X=\diag(I_2,0)$) at
    post-critical step sizes, with the bounded fraction annotated. Every
    coordinate in both problems is drawn independently from
    $\mathcal N(0,0.3^2)$, with $N=2\times10^4$ per panel (seed $7$), horizon
    $600$, and escape radius squared $10^{12}$. Evidence: exploratory. The
    approximation panels lie outside the range in which $q_\eta(\delta)<0$
    supports the proved boundary-inward argument; the experiment does not show
    that this condition is necessary.}
    \label{fig:apx_post_critical}
\end{figure}

%% file: apxM_objective_extensions.tex
%
%

\section{Recovering and extending beyond the baseline objective}
\label{apx:objective_extensions}
\label{apx:boundary_inward_extensions}

The certificates in the preceding appendices were derived for a fixed
objective. This appendix perturbs the objective and asks what survives of the
same ellipsoidal geometry. The same certificate first retains the
boundary-inward property under the $\ell_2$-regularized dynamics, then
extends to a quartic augmentation, where a sampled scan measures how
conservative the analytic condition is. At $\mu=-1/4$ the origin-centered
certificate fails, motivating a moving-center candidate that is tested
numerically. Independently, an exact moving-center family for unperturbed
scalar factorization satisfies every axiom, so the origin-centering hypothesis
of Theorem~\ref{thm:uniqueness_state_dependent_lyapunov} cannot be dropped.

\subsection{The \texorpdfstring{$\ell_2$}{L2}-regularized objective (A5-proven)}
\label{subsec:regularized_factorization_boundary_inward}

The objective, GD map, and stationary geometry of the regularized model are
catalogued in Appendix~\ref{subsec:model_catalog_regularized}. We recall its
loss,
\begin{align}
\risk_\lambda(x)
=
\risk_\fac(x)
+
\frac{\lambda}{2}\|x\|^2,
\qquad \lambda>0,
\end{align}
with $x=(a,b,u,v)$, whose gradient-descent map is
\begin{align}
\gd_\eta^\lambda(x)
=
x-\eta\nabla \risk_\fac(x)-\eta\lambda x
=
(1-\eta\lambda)x-\eta\nabla \risk_\fac(x).
\label{eq:apxM_regularized_gd}
\end{align}
The regularizer does not require a new certificate: its update is a convex
combination of the antipode and an unregularized step with an effective step
size.

\begin{proposition}
\label{prop:regularized_factorization_boundary_inward}
Assume $\lambda>0$, $\eta\lambda<2$, and $0<\delta<2$. If
\begin{align}
\tilde\eta
:=
\frac{\eta}{1-\eta\lambda/2}
<\delta,
\label{eq:regularized_factorization_threshold}
\end{align}
then every boundary point $x$ with $\ifac(\delta;x)=0$ satisfies
$\ifac(\delta;\gd_\eta^\lambda(x))<0$.
\end{proposition}

\begin{proof}
Set $\theta:=\eta\lambda/2\in(0,1)$, so that $\tilde\eta=\eta/(1-\theta)$ and,
by assumption, $0<\tilde\eta<\delta$. Since $\ifac(\delta;-x)=\ifac(\delta;x)=0$,
the antipode $-x$ lies on the same boundary. A direct computation gives the
convex-combination identity
\begin{align}
\theta(-x)+(1-\theta)\gd_{\tilde\eta}^\fac(x)
&=(1-2\theta)x-(1-\theta)\tilde\eta\nabla\risk_\fac(x) \\
&=(1-\eta\lambda)x-\eta\nabla\risk_\fac(x) \\
&=\gd_\eta^\lambda(x).
\label{eq:apxM_effective_step_identity}
\end{align}
By Proposition~\ref{prop:boundary_inward_rank1}, $\gd_{\tilde\eta}^\fac(x)$
lies in $\{\ifac(\delta;\cdot)<0\}$ unless $x$ is stationary for the
unregularized dynamics, in which case $\gd_{\tilde\eta}^\fac(x)=x$. For
$0<\delta<2$ the set $K_\delta^\fac:=\{y:\ifac(\delta;y)\le0\}$ is a strictly
convex ellipsoid. In the first case a nontrivial convex combination of
$-x\in K_\delta^\fac$ with an interior point is interior; in the second,
$\theta(-x)+(1-\theta)x$ lies strictly between two antipodal boundary points
of a strictly convex ellipsoid and is therefore interior as well.
\end{proof}

Beyond the boundary-inward property at step $\tilde\eta$, which is the input,
the proof uses only origin-centering and strict convexity of the sublevel
sets. For any certificate with those two, the effect of $\ell_2$
regularization is to replace the step-size condition on $\eta$ by the same
condition on $\tilde\eta$.

\paragraph{The resulting sublevel set.}
In the scalar case the conclusion has a closed form. Recall from
Subsection~\ref{subsec:scalar_factorization} that
$\delta(a,b)=8/Q(1;a,b)$, so for every $\delta\in(0,2)$
\begin{align}
    \isc(\delta;\,a,b)<0
    \iff
    Q(1;\,a,b)<\frac{8}{\delta}.
    \label{eq:apxM_Q_identity}
\end{align}
Proposition~\ref{prop:regularized_factorization_boundary_inward} in its scalar
form gives the boundary-inward property at every $\delta\in(\tilde\eta,2)$, so
by nesting the corresponding initial sublevel set is
$\{\isc(\tilde\eta;\cdot)<0\}$. By Eq.~\eqref{eq:apxM_Q_identity}, this is
\begin{align}
    \left\{
        Q(1;\,a,b)<\frac{8}{\tilde\eta}
    \right\}
    =
    \left\{
        Q(1;\,a,b)<\frac{8}{\eta}-4\lambda
    \right\},
    \label{eq:apxM_recovered_region}
\end{align}
the last equality being the identity
$8/\tilde\eta=8(1-\eta\lambda/2)/\eta=8/\eta-4\lambda$. The region is nonempty
precisely when $\tilde\eta<2$, that is when $\eta(1+\lambda)<2$.

The same inequality appears as a sufficient convergence condition in the
literature. Under the hypotheses of the small-step regularization theorem of
\citet{liang2025gradient}, its specialization to scalar factorization with unit
target gives $\eta<8/(4\lambda+Q(1;a,b))$, exactly equivalent to
Eq.~\eqref{eq:apxM_recovered_region}.

The argument uses origin-centering, through $\ifac(\delta;x)=0\Rightarrow
\ifac(\delta;-x)=0$. How $\ell_2$ regularization interacts with the
moving-center families of
Subsection~\ref{subsec:moving_center}, where the antipodal argument is
unavailable, is open.

\subsection{Quartic augmentation (A5-proven)}
\label{subsec:quartic_perturbation_boundary_inward}
The quartic-augmented model is catalogued in
Appendix~\ref{subsec:model_catalog_quartic}. Its loss is
\begin{align}
\risk_\mu(a,b)
=
\frac{1}{2}(ab-1)^2+\mu(ab-1)^4 .
\end{align}
Writing $L:=1-ab$, its GD map is
\begin{align}
    a^+&=a+\eta L(1+4\mu L^2)\,b,
    &
    b^+&=b+\eta L(1+4\mu L^2)\,a,
    \label{eq:apxM_gd_quartic}
\end{align}
which we denote by $\gd_\eta^\mu$. It is exactly the scalar factorization
update with the iterate-dependent effective step size
\begin{align}
\eta_{\mathrm{eff}}(a,b)
:=
\eta(1+4\mu L^2).
\end{align}
The designated target $\{ab=1\}$ remains a local-minimizer manifold, and the
quartic term has zero Hessian there. Thus the same local Lagrange calculation
identifies $\isc$ as the origin-centered candidate
(Remark~\ref{rmk:ext_to_quartic}). On the boundary
$\isc(\delta;a,b)=0$ a scalar factorization step with step size
$\eta_{\mathrm{eff}}$ maps the point into the sublevel set whenever
$0<\eta_{\mathrm{eff}}<\delta$, strictly away from stationary points; so it
suffices to ensure
\begin{align}
0<\eta(1+4\mu L^2)<\delta
\label{eq:quartic_effective_stepsize_condition}
\end{align}
uniformly over the boundary. By Proposition~\ref{prop:ellipse_equivalence},
every point on this boundary satisfies $|L|\le 2/\delta$.

\begin{proposition}
\label{prop:quartic_boundary_inward_sufficient}
Let $0<\delta<2$. Assume either
\begin{align}
\mu\ge 0
\quad\text{and}\quad
\eta\delta^2+16\eta\mu-\delta^3<0,
\label{eq:quartic_positive_mu_condition}
\end{align}
or
\begin{align}
\mu<0,
\qquad
\delta>\eta,
\qquad
\delta>\sqrt{-16\mu}.
\label{eq:quartic_negative_mu_condition}
\end{align}
Then
$\gd_\eta^\mu(\{\isc(\delta;\cdot)=0\})
\subseteq\{\isc(\delta;\cdot)\le0\}$, and every nonstationary boundary point
is mapped into $\{\isc(\delta;\cdot)<0\}$.
\end{proposition}

\begin{proof}
It is enough to verify
Eq.~\eqref{eq:quartic_effective_stepsize_condition} on the boundary.

If $\mu\ge0$ then $1+4\mu L^2\ge1$, so positivity is automatic, and
$|L|\le2/\delta$ gives
$\eta(1+4\mu L^2)\le\eta(1+16\mu/\delta^2)$; requiring this to be less than
$\delta$ is Eq.~\eqref{eq:quartic_positive_mu_condition}.

If $\mu<0$ the upper bound is largest at $L=0$, so $\eta<\delta$ suffices for
it, while positivity is tightest at $L^2=4/\delta^2$, where it reads
$1+16\mu/\delta^2>0$, that is $\delta>\sqrt{-16\mu}$.
\end{proof}

For $\mu\ge0$ the sufficient threshold is the largest root of
\begin{align}
q_{\eta,\mu}(\delta)
:=
\eta\delta^2+16\eta\mu-\delta^3,
\end{align}
above which $q_{\eta,\mu}(\delta)<0$, in the same form as the condition
$q_\eta(\delta)<0$ of Section~\ref{sec:rank-1 convergence}. A nonterminal
range on which Axiom~\ref{cond:monotonicity} holds exists precisely when
$\eta<2/(1+4\mu)$. For $\mu<0$, the corresponding proved range is
\begin{align}
\bigl(\max\{\eta,\sqrt{-16\mu}\},2\bigr),
\end{align}
which is nonempty exactly when $\eta<2$ and $\mu>-1/4$. At $\mu=-1/16$ its
lower endpoint is $\max\{\eta,1\}$, whereas at $\mu=-1/4$ the range is empty.
Thus the effective-step argument becomes vacuous at the latter endpoint.

\subsection{How conservative are the analytic quartic conditions?}
\label{subsec:quartic_conservativeness}
\label{apx:isc_quartic_experiment}

We compare the sufficient threshold of
Proposition~\ref{prop:quartic_boundary_inward_sufficient}, denoted by
$\delta_{\mathrm{th}}^{\mathrm{an}}(\eta,\mu)$, with the sampled threshold
$\delta_{\mathrm{th}}^{\mathrm{emp}}(\eta,\mu)$ of
Eq.~\eqref{eq:apxL_empirical_threshold}. The former is determined by
Eqs.~\eqref{eq:quartic_positive_mu_condition}--\eqref{eq:quartic_negative_mu_condition},
whereas the latter uses the boundary sampling protocol of
Subsection~\ref{apx:a5_protocol} for $\isc$ under $\gd_\eta^\mu$.

\begin{table}[t]
    \centering
    \caption{Analytic and sampled boundary-inward thresholds for $\risk_\mu$.
    The sampled threshold uses $195$ states in $[0.05,1.95]$ and
    $1.2\times10^5$ Gaussian directions (seed $31$, tolerance $10^{-4}$).
    Coverage at that threshold uses a $400^2$ offset grid over
    $[-2.5,2.5]^2$ and $800$ steps with tolerance $10^{-4}$. A dash denotes
    either no passing state or, in the coverage column, no sampled trajectory
    reaching the target.}
    \label{tab:quartic_thresholds}
    \begin{tabular}{cccccc}
        \toprule
        $\mu$ & $\eta$
        & $\delta_{\mathrm{th}}^{\mathrm{an}}$
        & $\delta_{\mathrm{th}}^{\mathrm{emp}}$
        & gap
        & $\widehat P(\text{cert}\mid\text{tgt})$ \\
        \midrule
        $+1/4$ & $0.05$ & $0.602$ & $0.608$ & $-0.006$ & $92.2\%$ \\
               & $0.20$ & $1.000$ & $1.000$ & $\phantom{-}0.000$ & $73.5\%$ \\
               & $0.80$ & $1.794$ & $1.803$ & $-0.009$ & $56.1\%$ \\
               & $1.20$ & --- & --- & --- & --- \\
        \midrule
        $+1/16$ & $0.05$ & $0.386$ & $0.393$ & $-0.007$ & $99.1\%$ \\
                & $0.20$ & $0.660$ & $0.667$ & $-0.007$ & $88.9\%$ \\
                & $0.80$ & $1.285$ & $1.294$ & $-0.009$ & $82.1\%$ \\
                & $1.20$ & $1.644$ & $1.646$ & $-0.002$ & --- \\
        \midrule
        $-1/16$ & $0.05$ & $1.000$ & $1.000$ & $\phantom{-}0.000$ & $67.7\%$ \\
                & $0.20$ & $1.000$ & $1.000$ & $\phantom{-}0.000$ & $67.7\%$ \\
                & $0.80$ & $1.000$ & $1.000$ & $\phantom{-}0.000$ & $70.3\%$ \\
                & $1.20$ & $1.200$ & $1.196$ & $\phantom{-}0.004$ & --- \\
        \midrule
        $-1/4$ & all tested & --- & --- & --- & --- \\
        \bottomrule
    \end{tabular}
\end{table}

Whenever both thresholds exist, their gap lies between $-0.009$ and $+0.004$,
within one state-grid cell up to rounding; this is over the six step sizes
tested, of which the table lists four. Thus no discrepancy larger than the
grid resolution was detected. The scan also
reproduces the analytic cutoffs at the tested parameters. At $\mu=-1/4$ no
state passes, and the positive sampled maxima underlying the final row of
Table~\ref{tab:quartic_thresholds} are explicit counterexamples to
Axiom~\ref{cond:monotonicity} at the tested parameters. They do not imply an
impossibility result for origin-centered quadratics in general.

\begin{figure}[t]
    \centering
    \includegraphics[width=\linewidth]{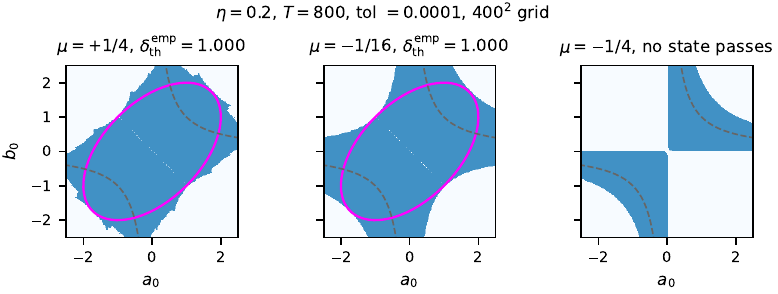}
    \caption{Certificate contour and finite-horizon target-reaching region for
    $\risk_\mu$ at $\eta=0.2$. Blue points satisfy
    $|a_Tb_T-1|<10^{-4}$ after $T=800$ steps; magenta is
    $\isc(\delta_{\mathrm{th}}^{\mathrm{emp}};\cdot)=0$; gray dashed is
    $\{ab=1\}$. At $\mu=-1/4$ no state passes, so the magenta contour is
    omitted.}
    \label{fig:quartic_convergence}
\end{figure}

The endpoint failure and the changed target-reaching geometry motivate allowing
the center of the certificate to depend on the state.

\subsection{Moving-center construction (A5-sampled)}
\label{subsec:moving_center}
\label{apx:moving_center}

We work in the normalized composite-scalar setting of
Appendix~\ref{subsec:model_catalog_composite_scalar}. Relabeling the normalized
one-variable loss, write $\risk(a,b)=f(ab)$ and
$\mathcal M=\{ab=1\}$. At
$x_*=(a_*,1/a_*)^\top\in\mathcal M$, the Hessian formula
Eq.~\eqref{eq:model_catalog_composite_hessian} has tangent null direction
$(a_*^2,-1)^\top$ and transverse eigendirection $(1,a_*^2)^\top$. Only this
$f$-independent eigendirection structure is used in the construction below.
The quartic loss at
$\mu=-1/4$ is the sampled instance.

Theorem~\ref{thm:uniqueness_state_dependent_lyapunov} characterizes the
admissible geometry inside the origin-centered exchange-symmetric subclass.
The general framework of
Subsection~\ref{subsec:structure_of_state_dependent_lyapunov} allows an affine
term $2q(\delta)^\top x$, which for $P(\delta)\succ0$ is equivalently a center
that moves with the state. Taking $q=-\kappa P(\kappa)\onevector$ with
$\onevector=(1,1)^\top$ gives
\begin{align}
    I(\kappa;\, x)
    =
    (x - \kappa\onevector)^\top P(\kappa)\,(x - \kappa\onevector) - 1,
    \qquad
    P(\kappa)
    =
    \begin{pmatrix}
        c(\kappa) & d(\kappa) \\
        d(\kappa) & c(\kappa)
    \end{pmatrix},
    \label{eq:mc_centered_ansatz}
\end{align}
with $x=(a,b)^\top$, sublevel sets $K_\kappa:=\{x:I(\kappa;\,x)\le0\}$, and the
center confined to the balanced diagonal by symmetry, so a single scalar
exhausts the extra freedom.

\paragraph{The local necessary condition.}
Suppose a family of the form~\eqref{eq:mc_centered_ansatz} satisfies
Axioms~\ref{cond:monotonicity}--\ref{cond:stationarity} at some step size.
Applying the Lagrange-multiplier argument of
Appendix~\ref{appendix:uniqueness_state_dependent_lyapunov} after translating
the center by $\kappa\onevector$ shows that every stationary boundary crossing
$x_*$ must have $P(\kappa)(x_*-\kappa\onevector)$ in a Hessian eigendirection.
We select the tangent null direction $(a_*^2,-1)^\top$, the choice that
recovers the scalar certificate at $\kappa=0$.

In the symmetric coordinate $\tau := a_* + a_*^{-1}\ge 2$ of the crossing with
$\{ab=1\}$, both conditions are quadratic in $\tau$. Writing $c=c(\kappa)$,
$d=d(\kappa)$, the requirement that $x_*$ lie on the boundary and the
alignment requirement
$\langle P(\kappa)(x_*-\kappa\onevector),(1,a_*^2)^\top\rangle=0$ read
\begin{align}
    \text{boundary:}\quad
    &c\,\tau^2 - 2\kappa(c+d)\,\tau + 2\kappa^2(c+d) - 2c + 2d - 1 = 0,
    \label{eq:mc_boundary_relation}
    \\
    \text{alignment:}\quad
    &d\,\tau^2 - \kappa(c+d)\,\tau + 2(c-d) = 0 .
    \label{eq:mc_alignment_relation}
\end{align}

\paragraph{A distinguished branch fixes the shape.}
To impose the alignment relation at every root $\tau\ge2$ of the boundary
relation, choose the \emph{proportional branch}, on which
Eq.~\eqref{eq:mc_alignment_relation} is a multiple of
Eq.~\eqref{eq:mc_boundary_relation}. Coefficient matching gives
\begin{align}
    d(\kappa) = \frac{c(\kappa)}{2},
    \qquad
    c(\kappa) = \frac{1}{3(\kappa^2-1)},
    \qquad
    P(\kappa)
    =
    \frac{1}{3(\kappa^2-1)}
    \begin{pmatrix}
        1 & 1/2 \\
        1/2 & 1
    \end{pmatrix},
    \label{eq:mc_proportional_branch}
\end{align}
so that
\begin{align}
    I(\kappa;\, a, b)
    =
    \frac{(a-\kappa)^2 + (b-\kappa)^2 + (a-\kappa)(b-\kappa)}
         {3(\kappa^2-1)}
    - 1 .
    \label{eq:mc_certificate}
\end{align}
For the first-quadrant branch take $\kappa>1$, where $P(\kappa)\succ0$ and the
center lies beyond $(1,1)$ on the balanced diagonal. This is a constructive
fixed-shape branch, since $d/c=1/2$; no uniqueness is claimed.

For the structural properties it is convenient to multiply
Eq.~\eqref{eq:mc_certificate} by $3(\kappa^2-1)>0$. This leaves every
$K_\kappa$ unchanged and gives the equivalent certificate
\begin{align}
    I(\kappa;\,a,b)
    = q_0(a,b) - 3\kappa\,(a+b) + 3,
    \qquad
    q_0(a,b) := a^2+ab+b^2,
    \qquad
    \kappa>1 .
    \label{eq:mc_unnormalized}
\end{align}

\paragraph{Orientation of the state.}
The abstract framework orients the state so that the level sets contract as
the state increases. Here they contract as $\kappa\downarrow1$. We therefore
use $\kappa\in(1,\infty)$ itself as the nonterminal state, with the reverse
order understood, and write $\mathcal K_I:=(K_\kappa)_{\kappa\ge1}$, with
$K_1$ as the terminal member. Equivalently, all definitions and results apply
after the order-preserving relabeling $s=1/\kappa$; we suppress this relabeling
below. The dynamical range above a threshold $\kappa_{\mathrm{th}}$ is
$1<\kappa<\kappa_{\mathrm{th}}$.

\begin{proposition}[Geometry of the moving-center family]
\label{prop:mc_structural_axioms}
The family $\mathcal K_I$ satisfies
Axioms~\ref{cond:P_is_pd}--\ref{cond:level_set_as_state}
and~\ref{cond:symmetry}. Its terminal set is $K_1=\{(1,1)\}$.
\end{proposition}

\begin{proof}
Completing the square in Eq.~\eqref{eq:mc_unnormalized} gives
\begin{align}
    K_\kappa
    =\{x:q_0(x-\kappa\onevector)\le3(\kappa^2-1)\},
    \qquad \kappa>1.
\end{align}
Since $q_0$ is positive definite and $3(\kappa^2-1)>0$, every nonterminal
$K_\kappa$ is a nonempty compact nondegenerate ellipse centered at
$\kappa\onevector$. Exchange symmetry is immediate.
Writing $S:=a+b$, we also have
\begin{align}
    K_\kappa=\{x:q_0(x)+3\le3\kappa S\}.
    \label{eq:mc_sublevel_linear_kappa}
\end{align}
The left-hand side of the defining inequality is positive, so
$K_\kappa\subset\{S>0\}$. Hence its residual
$q_0(x)+3-3\kappa S$ is strictly decreasing in $\kappa$ on the certified
union. In particular, if $1<\kappa<\kappa'$, then
$K_\kappa\subset\operatorname{int}K_{\kappa'}$, which is nesting in the
reverse $\kappa$-order. The terminal inclusion is also strict because
$I(\kappa;1,1)=6(1-\kappa)<0$ for every $\kappa>1$.

For a point in the union outside the terminal set, the boundary equation is
linear in $\kappa$ and has the unique solution
\begin{align}
    \kappa(x)=\frac{q_0(x)+3}{3S}>1.
\end{align}
Here the strict inequality follows from
$q_0(x)+3-3S=q_0(x-\onevector)>0$ away from $(1,1)$.
Finally, letting $\kappa\downarrow1$ in the completed-square form gives
\begin{align}
    K_1=\bigcap_{\kappa>1}K_\kappa
    =\{x:q_0(x-\onevector)\le0\}=\{(1,1)\},
\end{align}
so the terminal set is negligible.
\end{proof}

In the state $\kappa$, Eq.~\eqref{eq:mc_unnormalized} has the constant quadratic
coefficient
$P=\left(\begin{smallmatrix}1&1/2\\1/2&1\end{smallmatrix}\right)$,
$q(\kappa)=-\tfrac{3\kappa}{2}\onevector$, and $r\equiv3$. The normalized form
Eq.~\eqref{eq:mc_certificate} is used below only for the numerical search.

\paragraph{A one-dimensional search over the center.}
Apply the protocol of Subsection~\ref{apx:a5_protocol} to
Eq.~\eqref{eq:mc_certificate} and evaluate
$\kappa\mapsto\widehat M_N(\kappa)$ from Eq.~\eqref{eq:apxL_sampled_max}.

Scanning $\kappa\in[1.0005,1.08]$ with spacing $5\times10^{-4}$ at
$\mu=-1/4$, $\eta=0.2$, on $2\times10^5$ seeded Gaussian directions (seed $51$)
with acceptance tolerance $10^{-4}$, the sampled-passing set is a contiguous
interval,
\begin{align}
    1<\kappa\le\kappa_{\mathrm{th}}^{\mathrm{emp}}=1.0605,
\end{align}
and a deterministic $\ell^\infty$ grid of $16000$ rays selects the same
interval.

The boundary of $K_\kappa$ first touches the stationary hyperbola $\{ab=2\}$
at $\kappa=3/(2\sqrt2)=1.06066$, one grid step above where the sampled
interval stops.

\begin{figure}[t]
    \centering
    \includegraphics[width=\linewidth]{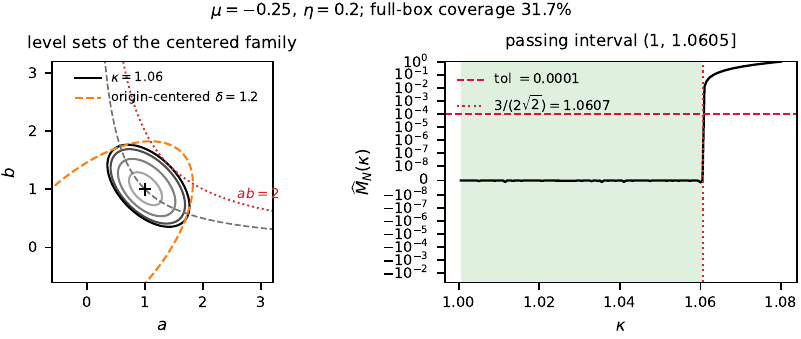}
    \caption{The centered family Eq.~\eqref{eq:mc_certificate} for
    $\risk_{-1/4}$. \textbf{(a)} Boundaries $\partial K_\kappa$ at
    $\kappa=1.01,\,1.03,\,1.05$ and at $\kappa_{\mathrm{th}}^{\mathrm{emp}}$,
    darkening with $\kappa$, against the origin-centered
    $\{\isc(1.2;\cdot)=0\}$ (dashed); $\{ab=1\}$ and the stationary
    $\{ab=2\}$ are marked, as is the balanced minimizer $(1,1)$.
    \textbf{(b)} The sampled maximum under the
    scan above, with the passing interval shaded and $\kappa=3/(2\sqrt2)$
    marked, the value at which $\partial K_\kappa$ first meets the stationary
    hyperbola $\{ab=2\}$.}
    \label{fig:moving_center_family}
\end{figure}

\paragraph{Coverage.}
At $\kappa=\kappa_{\mathrm{th}}^{\mathrm{emp}}$ and $\eta=0.2$, a $1400^2$
grid over $[-2.5,2.5]^2$ run for $3000$ steps with tolerance $10^{-5}$ gives
coverages of $10.85\%$ for the certificate together with its reflection and
$34.22\%$ for the empirical basin. The reflection is valid because the loss is
even. Thus
\begin{align}
    \widehat P(\text{cert}\mid\text{target}) = 31.7\%,
    \qquad
    \widehat P(\text{target}\mid\text{cert}) = 1.0000,
\end{align}
where the overlap is counted only once.

At the tested step sizes, numerical scans of several other composite scalar
losses also identify intervals
$1<\kappa\le\kappa_{\mathrm{th}}^{\mathrm{emp}}$ on which the same construction
passes the sampled A5 test. This repeated pattern motivates the following
broader conjecture.

\begin{conjecture}[Moving-center certificates for composite scalar losses]
\label{conj:moving_center_composite}
Let $f\in C^2$, let $m\ne0$ satisfy $f'(m)=0$ and $f''(m)>0$, and use the
normalization of
Appendix~\ref{subsec:model_catalog_composite_scalar}.
For every sufficiently small normalized step size $\eta>0$, there exists
$\kappa_{\mathrm{th}}=\kappa_{\mathrm{th}}(f,\eta)>1$ such that the triplet
$(\mathcal K_I,\gd_\eta,\kappa_{\mathrm{th}})$ is a state-dependent Lyapunov
family generated by Eq.~\eqref{eq:mc_unnormalized}, under the reverse-order
convention above.
\end{conjecture}

Whenever Conjecture~\ref{conj:moving_center_composite} holds, the terminal
case of Theorem~\ref{thm:abstract_state_dependent_convergence} gives
$x_t\to(1,1)$ because $K_1=\{(1,1)\}$. If, in addition, the regularity
hypotheses of that theorem hold, its stationary slices are finite, and every
non-target stationary point in the certified region has an expanding
eigenvalue, then the nonterminal case and the unstable-limit exclusion give
convergence to the target manifold for almost every certified initialization.

\subsection{Origin centering is necessary for uniqueness}
\label{subsec:moving_center_sharpness}

For $\mu=0$, i.e., plain scalar factorization, the proposition below gives an
exact terminal-side range on which Eq.~\eqref{eq:mc_certificate} satisfies the
two dynamical axioms for every $\eta\in(0,1)$. Together with
Proposition~\ref{prop:mc_structural_axioms}, this gives an exact
state-dependent Lyapunov family. Its sublevel sets are not invariant under
$x\mapsto-x$, so they are none of the sets in $\mathcal K_{\isc}$, and the
hypothesis $q\equiv0$ in
Theorem~\ref{thm:uniqueness_state_dependent_lyapunov} is a genuine restriction.

\begin{proposition}[Exact dynamical threshold for scalar factorization]
\label{prop:mc_axioms_mu0}
Let $\risk_\sca(a,b)=\tfrac12(1-ab)^2$ and let $(K_\kappa)_{\kappa>1}$ be the
family of Eq.~\eqref{eq:mc_unnormalized}, with the reverse $\kappa$-order.
For every $\eta\in(0,1)$,
Axioms~\ref{cond:monotonicity}--\ref{cond:stationarity} hold for
\begin{align}
    1<\kappa<\kappa_{\mathrm{th}}(\eta)
    :=\frac{1+\eta}{2\sqrt{\eta}}.
    \label{eq:mc_exact_threshold}
\end{align}
At $\kappa=\kappa_{\mathrm{th}}(\eta)$,
Axiom~\ref{cond:monotonicity} holds but Axiom~\ref{cond:stationarity} fails;
for $\kappa>\kappa_{\mathrm{th}}(\eta)$,
Axiom~\ref{cond:monotonicity} fails.
For $\eta\ge1$, Axiom~\ref{cond:monotonicity} fails at every nonterminal
$\kappa>1$.
\end{proposition}

\begin{proof}
Put $L:=1-ab$ and $S:=a+b$, so that
$\gd_\eta(x)=x+\eta L\,(b,a)^\top$ and $S(\gd_\eta(x))=(1+\eta L)S$. Expanding
the quadratic form $q_0$, which is invariant under the exchange
$(a,b)\mapsto(b,a)$,
\begin{align}
    I(\kappa;\,\gd_\eta(x))
    = I(\kappa;\,x)
      + \eta L\bigl(q_0(x)+3ab-3\kappa S\bigr)
      + \eta^2L^2 q_0(x) .
\end{align}
On $\partial K_\kappa$ one has $q_0(x)-3\kappa S=-3$, so the middle bracket is
$-3L$ and $q_0(x)=3\kappa S-3$, whence
\begin{align}
    I(\kappa;\,\gd_\eta(x)) = 3\eta L^2\bigl[\eta(\kappa S-1)-1\bigr]
    \qquad\text{on }\partial K_\kappa .
    \label{eq:mc_exact_step_identity}
\end{align}
The linear functional $S:(a,b)\mapsto a+b$ attains its maximum
$2\kappa+2\sqrt{\kappa^2-1}$ at the far balanced tip
$a=b=\kappa+\sqrt{\kappa^2-1}>1$, where $L\ne0$. Hence
Axiom~\ref{cond:monotonicity} holds if and only if
$2\eta\kappa\bigl(\kappa+\sqrt{\kappa^2-1}\bigr)\le1+\eta$.
Write $\kappa=\cosh\theta$ with $\theta>0$. Then
$\kappa+\sqrt{\kappa^2-1}=e^\theta$, so the preceding condition reduces to
\begin{align}
    \eta e^{2\theta}\le1.
\end{align}
For $\eta<1$, this is equivalent to
\begin{align}
    \theta\le\frac12\log\frac1\eta,
    \qquad
    \kappa\le
    \cosh\!\left(\frac12\log\frac1\eta\right)
    =\frac{1+\eta}{2\sqrt\eta}
    =\kappa_{\mathrm{th}}(\eta).
\end{align}
Thus Axiom~\ref{cond:monotonicity} holds up to and including the threshold.
At the threshold, Eq.~\eqref{eq:mc_exact_step_identity} vanishes at the
nonstationary far tip, so Axiom~\ref{cond:stationarity} fails. Above the
threshold, the same expression is positive at that tip, so even
Axiom~\ref{cond:monotonicity} fails. If $\eta\ge1$, the inequality
$\eta e^{2\theta}\le1$ has no solution with $\theta>0$, proving failure at
every nonterminal level.

It remains to verify Axiom~\ref{cond:stationarity} below the threshold. There
the bracket in Eq.~\eqref{eq:mc_exact_step_identity} is strictly negative at
every point of $\partial K_\kappa$. Consequently,
$I(\kappa;\gd_\eta(x))=0$ holds on the boundary if and only if $L=0$. For the
scalar risk, $\nabla\risk_\sca(x)=-L(b,a)^\top$, so its stationary points have
$L=0$ or $x=0$. Since $I(\kappa;0)=3>0$, the origin does not lie in
$K_\kappa$. Hence $L=0$ is equivalent to $\nabla\risk_\sca(x)=0$ on
$\partial K_\kappa$, as required.
\end{proof}

The obstruction for $\eta>1$ can also be read from
Theorem~\ref{thm:abstract_state_dependent_convergence}. At a minimizer
$(a,b)$ with $ab=1$, the transverse multiplier of the GD map is
$1-\eta(a^2+b^2)$, whose modulus is greater than one because
$a^2+b^2\ge2$. The origin is outside every $K_\kappa$. Thus, if the
moving-center sets satisfied the dynamical axioms above some threshold,
conclusion~(5) of that theorem would force the terminal case for almost every
certified initialization. Since $K_1=\{(1,1)\}$, conclusion~(2) would then
give $x_t\to(1,1)$, whereas conclusion~(4) excludes convergence to this
unstable fixed point outside a null set. This contradiction shows why the
present moving-center construction cannot extend to $\eta>1$.

The moving-center family distinguishes the two sign branches of the minimizer
manifold. The family above certifies convergence to the positive branch, with
terminal point $(1,1)$, while its reflection certifies convergence to the
negative branch, with terminal point $(-1,-1)$. By contrast,
$\mathcal K_{\isc}$ is invariant under $x\mapsto-x$, and its convergence
statement does not determine which branch contains the limit.

%% file: apxN_related_work.tex
\section{Relation to prior work on gradient dynamics and matrix factorization}
\label{apx:related_work}

This appendix places the state-dependent certificate of
Sections~\ref{sec:rank-1 convergence} and~\ref{sec:state_dependent_lyapunov}
inside five lines of work: strict saddles avoidance of gradient descent; small-step analyses of gradient descent and
gradient flow in matrix factorization; large-step analyses of gradient descent
in matrix factorization; edge-of-stability dynamics; and parameterized
Lyapunov constructions. These lines answer different questions, so
we discuss them separately.

\subsection{Avoidance of strict saddles and other unstable stationary points}
\label{subsec:rw_strict_saddle}

Almost-everywhere strict-saddle avoidance rests on the center-stable manifold
theorem: when the gradient descent map $\gd_\eta(x)=x-\eta\nabla f(x)$ is a
diffeomorphism, the initializations converging to a given strict saddle form a
countable union of iterated preimages of a local center-stable manifold of
positive codimension, hence a null set. \citet{lee2016gradient} established this
under the global step-size condition $0<\eta<1/\beta$, where $\beta$ bounds the
Hessian norm, and extended it to countably many strict saddles.
\citet{panageas2017gradient} removed the isolation and countability requirements
(their Theorem~2) and localized the conclusion to an open convex
forward-invariant set $S$, with
$\beta=\sup_{x\in S}\norm{\nabla^2 f(x)}_2<\infty$ (their Theorem~3).

Two limitations are relevant here. First, the step-size conditions make the
gradient map a diffeomorphism onto its image. Second, the conclusion excludes
pointwise convergence to a strict saddle, but not approach to the strict-saddle
set without a pointwise limit, which for our rank-one factorization model would
mean $\dist{x_t}{\mathcal S\setminus\mathcal M}\to0$. The distinction matters
when the strict saddles form a connected or positive-dimensional set.

The first limitation can be relaxed: diffeomorphism is more than the argument
needs. \citet{schaeffer2020extending} showed that it suffices for the degeneracy
set $\{x : \det D\gd_\eta(x)=0\}$ to have measure zero and to contain no saddle
point, and obtained strict-saddle avoidance for $\eta\beta\in(0,2)$. Their
Example~2.3 shows that the measure-zero hypothesis cannot be dropped: there the
degeneracy set has positive measure and a positive-measure set of
initializations converges to a strict saddle.

Once the null-preimage property is isolated, the same argument applies beyond
strict saddles. At a stationary point $p$, the multiplier of the gradient
descent map along a Hessian eigendirection with eigenvalue $\lambda$ is
$1-\eta\lambda$, which lies outside the unit disk both when $\lambda<0$ and when
$\lambda>2/\eta$. Assumption~1 of \citet{ahn2022understanding} requires a
null-preimage property and that $1/\eta$ is not a Hessian eigenvalue at any
stationary point. Under it, they used this observation to show that, outside a
null set of initializations, gradient descent converges to no stationary
point $p$ satisfying
\begin{align}
    \lambda_{\min}(\nabla^2 f(p))<0
    \qquad\text{or}\qquad
    \lambda_{\max}(\nabla^2 f(p))>2/\eta.
\end{align}
\citet{craciun2024convergence} subsequently verified this property for
linear-network losses rather than assuming it: their Theorem~12 shows that the
gradient descent map is non-singular by exploiting the polynomial structure of
$\det(I-\eta\nabla^2 f)$, and their Theorem~15 concludes that above a critical
step size $\eta_E$ the stable set of the entire minimizer manifold is null. For
the scalar factorization $\tfrac12(1-ab)^2$, which is our scalar model, their
critical value is $\eta_E=1$, the threshold at which minimizers lose stability
here.

Beyond this almost-everywhere non-convergence result, our post-critical
analysis describes the subsequent dynamics. For almost every certified
initialization in the scalar and rank-one factorization models with
$\eta\in(1,2)$, the trajectory approaches the terminal set, and its
$\omega$-limit set projects to an internally chain transitive set of the reduced
one-dimensional map (Appendix~\ref{sec:terminal_inheritance}). Moreover, their
stable-manifold argument uses local invertibility at the relevant minima,
whereas our exclusion argument covers exceptional stationary points at which
the gradient map itself is singular. At those points,
Corollary~\ref{cor:local_nonescaping_unstable_fixed_point} supplies the
measure-zero local nonescaping set through the stable-manifold theorem for
pseudo-hyperbolic endomorphisms \citep[Theorem~5.1]{hirsch1977invariant}.

The second limitation is the gap between pointwise saddle avoidance and
approach to a saddle set. Stable-manifold arguments such as those of
\citet{panageas2017gradient} exclude convergence to a strict saddle, but do not
ensure that a bounded trajectory has a pointwise limit. When the objective
decreases along the trajectory, analyticity supplies such a mechanism:
\citet{absil2005convergence} used the {\L}ojasiewicz gradient inequality to show
that bounded descent sequences converge to a single limit point under suitable
angle and step-size conditions. Whether this route extends to the large step
sizes considered here, where the objective need not decrease, is unclear.

Our certificate analysis bridges this gap differently. The state-dependent
certificate first confines the orbit and places its accumulation points on the
stationary contacts of the limiting boundary. In the scalar and rank-one
factorization models, the contacts at a nonterminal limiting state are finite,
and the accompanying vanishing-increment argument then forces the whole
trajectory to converge to a single contact. The local avoidance argument enters
afterwards to exclude the unstable candidates for almost every initialization.
Here the null-preimage property comes from
Corollary~\ref{cor:gd_preimage_null}, via the submersion theorem of
\cite{ponomarev1987submersions}.

\subsection{Small-step gradient descent and gradient flow in matrix factorization}
\label{subsec:rw_small_step}

A second line studies matrix factorization in gradient flow and in discrete
time under small-step and small-initialization regimes where descent and
balancing can be controlled. \citet{du2018algorithmic}
proved in their Theorem~2.2 that, for two consecutive linear layers, the
continuous-time quantity $W_hW_h^\top-W_{h+1}^\top W_{h+1}$ is conserved along
gradient flow. For the rank-one problem, their Theorem~3.2 shows that Gaussian
initialization with entrywise standard deviation
$c_{\mathrm{init}}\sqrt{\sigma_1/d}$ (here, $d = \max\{m, n\}$) and the dimension-free step size
$\eta\le c_{\mathrm{step}}/\sigma_1$, with $c_{\mathrm{step}}$ fixed, drive
$\norm{u_tv_t^\top-M^\star}_F$ below $\epsilon\sigma_1$ within
$O(\log(d/\epsilon))$ iterations, with probability bounded below by a positive
universal constant. Their analysis keeps the signal strengths of the two
factors comparable, which enables the linear-rate guarantee.
\citet{ye2021global} extended the analysis to exactly parameterized asymmetric
rank-$k$ factorization of a general target, with a high-probability
polynomial-time guarantee and linear contraction after the initial escape
phase, at the cost of an initialization and a step size that both shrink
polynomially in the dimensions. \citet{jiang2023algorithmic} studied the
overparameterized, model-free setting and obtained a step size that is
dimension-free, depending only on $\sigma_r$, $\sigma_1$ and the relative
singular-value gap, while the admissible initialization scale remains
dimension-dependent; their conclusion is that the iterates pass close to the
best rank-$s$ approximations during explicit windows, rather than converging to
one.

For the gradient-flow rank-one approximation problem,
\citet{mohammadi2018stability} partition the state space into invariant
manifolds, reduce the dynamics on each leaf to a symmetric problem, and use a
Lyapunov argument to characterize regions of attraction and prove
almost-everywhere convergence to minimizers. Our approximation result instead
concerns the finite-step gradient map, for which invariance is a one-step
boundary inequality rather than a differential inequality.

The discrete-time analyses of
\cite{du2018algorithmic,ye2021global,jiang2023algorithmic} collectively cover
broader factor ranks and more general targets under specific random
initializations.
Theorems~\ref{thm:convergence_region_rank_1}
and~\ref{thm:convergence_region_rank_1_approx} instead describe closed-form,
dimension-independent regions and give almost-everywhere conclusions within
them, with no rate; in the factorization case the same certificate also
continues into the post-critical regime. For the rank-one problem, the
norm-control event used in the analysis of \citet{du2018algorithmic} is
contained in our certified region. Indeed, on this event, after normalization
their initialization bounds give $\norm{x_0}^2<\tfrac12$ for the sufficiently
small $c_{\mathrm{init}}$ in their theorem, and hence
$\lvert a_0b_0\rvert<\tfrac14$, so
\begin{align}
    \ifac(\eta;x_0)
    \le
    \tfrac{\eta}{2}+\tfrac{\eta^2}{4}+\eta^2-4
    <0
    \qquad\text{for all }0<\eta<1.
\end{align}
The certified regions are sufficient rather than sharp
(Remark~\ref{rmk:convergence_region_not_sharp}).

\begin{table}[t]
\centering
\scriptsize
\setlength{\tabcolsep}{2pt}
\caption{Comparison with small-step results. The initialization column reports
squared-norm scales for prior work and certified sublevel sets for our results.
The associated constants and accuracy-dependent restrictions differ across the
results.}
\label{tab:rw_small_step}
\begin{tabular}{@{}llllll@{}}
\toprule
& rank / target & squared-norm scale & step size & prob. / measure & rate / horizon \\
\midrule
\cite{du2018algorithmic}
& rank-$1$, exact
& $O(c_{\rm init}^2\sigma_1)$; dim.-free
& $\eta\le c_{\rm step}/\sigma_1$
& \shortstack{positive universal\\lower bound}
& linear \\
\cmidrule(lr){1-6}
\cite{ye2021global}
& rank-$k$, exact
& $\widetilde O\!\left(\sigma_k^2/[k^2\sigma_1(m{+}n)]\right)$
& dim.-dependent
& high
& \shortstack{polynomial time;\\linear post-escape} \\
\cmidrule(lr){1-6}
\cite{jiang2023algorithmic}
& rank-$k\ge r$, general
& $\tfrac{(m+n)k}{9(m{+}n{+}k)}\rho^2\sigma_1$
& dim.-free
& high
& window \\
\cmidrule(lr){1-6}
Thms.~\ref{thm:convergence_region_rank_1},
\ref{thm:convergence_region_rank_1_approx}
& rank-$1$, exact / approx.
& explicit sublevel set; dim.-free
& \shortstack[l]{fac.: $(0,1)$, conv.\\
fac.: $(1,2)$, terminal\\
apx.: $(0,1)$, conv.}
& Lebesgue a.e.
& --- \\
\bottomrule
\end{tabular}
\end{table}

\subsection{Large-step gradient descent in matrix factorization}
\label{subsec:rw_large_step}

\paragraph{Convergence and balancing at large learning rates.}
In the bounds below, we retain $\sigma$ rather than normalize it to one.
\citet{wang2022large} established that gradient descent on homogeneous matrix
factorization can converge at learning rates beyond the classical smoothness
threshold, and that large learning rates induce a balancing bias. For the
scalar-output vector factorization objective $\tfrac12(\sigma-x^\top y)^2$, their
Theorem~3.1 proves convergence to a global minimum for almost every
initialization when
\begin{align}
    \eta\le\min\Bigl\{\tfrac{4}{\norm{x_0}^2+\norm{y_0}^2+4\sigma},\ \tfrac{1}{3\sigma}\Bigr\},
\end{align}
and their Theorem~3.2 bounds the limiting squared norms by
$\norm{x}^2+\norm{y}^2\le 2/\eta$, hence the imbalance by
$\norm{x-y}^2\le 2/\eta-2\sigma$. For the isotropic rank-one approximation problem
$X=\sigma I_n$ their Theorems~4.1--4.3 establish alignment, convergence and the
corresponding balancing bound under
\begin{align}
    \eta\le\min\Bigl\{\tfrac{4}{\norm{x_0}^2+\norm{y_0}^2+4\sqrt7\,\sigma},\
    \tfrac{1}{2\sqrt7\,\sigma}\Bigr\}.
\end{align}
Their Section~5 covers an arbitrary square target and an arbitrary factor rank,
a considerably broader model class than the one treated here; the conclusion
there is a balancing bound that holds conditionally on convergence, rather than
a convergence theorem.

\paragraph{Sharp global dynamics in the scalar model.}
\citet{liang2025gradient} gave a sharp analysis of the scalar and scalar-vector
models. They identify an initialization-dependent critical step size separating
convergence to a global minimizer from failure to converge to any minimizer, and
they show that the convergence region is the sublevel set $\{Q<8/\eta\}$ of the
explicit quantity $Q$ recalled in Subsection~\ref{subsec:liang_summary}. They further show that the large-step dynamics are chaotic --- topological
entropy at least $\log3$, with periodic orbits of every period --- and that the
regularized problem has a self-similar convergence boundary. The scalar
certificate used here recasts this $Q$-sublevel description in state-dependent
coordinates: as recorded in Subsection~\ref{subsec:scalar_factorization},
$\sgn\isc(\delta;a,b)=\sgn(Q(1;a,b)-8/\delta)$ for $\delta\in(0,2)$, and the
state parameter is
$\delta_t=8/Q(1;a_t,b_t)$, so the monotone decrease of $Q$ in
Theorem~\ref{thm:convergence_region} and the monotone increase of $\delta_t$ are
the same statement read in two coordinates.

\paragraph{What the certificate adds.}
Building on these two precedents, the certificate contributes three further
elements. First, among origin-centered, exchange-symmetric quadratic families
satisfying the structural and dynamical axioms, the scalar family is uniquely
determined up to a monotone relabeling of the state
(Theorem~\ref{thm:uniqueness_state_dependent_lyapunov}); the same local Lagrange
analysis gives necessary block constraints for rank-one extensions. Second, the
rank-one target introduces off-signal null directions that turn the isolated
non-minimizing stationary point of the scalar model into a positive-dimensional
non-minimizing stationary set. Under the regularity conditions used here, the
boundary-inward argument forces pointwise convergence to a stationary point in
the nonterminal case (Proposition~\ref{prop:boundary_inward_rank1} and
Theorem~\ref{thm:abstract_state_dependent_convergence}). The stable-manifold
theorem and a compact-covering argument then exclude convergence to the
non-minimizing stationary set for almost every certified initialization.
Third, the same state $\delta_t$ and nested family $K_\delta$ organize both the
pre-critical minimizer-convergence regime and the post-critical terminal regime.

\paragraph{Step-size endpoints after normalization.}
Table~\ref{tab:rw_large_step} compares the initialization-independent step-size
bounds stated in the respective theorems. The
approximation results concern different targets: $\sigma I_n$ is isotropic and
full rank, whereas $\diag(\sigma I_{n-1},0)$ retains one null direction.

\begin{table}[t]
\centering
\footnotesize
\setlength{\tabcolsep}{5pt}
\caption{Normalized step-size bounds. The conditions of
\cite{wang2022large} also contain an initialization-dependent branch. The
approximation row compares different targets and is not a dominance comparison.}
\label{tab:rw_large_step}
\begin{tabular}{@{}llll@{}}
\toprule
Problem & Reference & Prior bound on $\eta\sigma$ & Our bound on $\eta\sigma$ \\
\midrule
\shortstack[l]{scalar/vector\\inner-product fact.}
& \cite{wang2022large} Thm.~3.1
& $\eta\sigma\le 1/3$
& \shortstack[l]{scalar certificate, $\eta\sigma<1$\\
(Thm.~\ref{thm:convergence_region}; App.~\ref{app:scalar_vector_proof})} \\
\shortstack[l]{rank-one approx.\\(different targets)}
& \cite{wang2022large} Thms.~4.1--4.2
& $\eta\sigma\le 1/(2\sqrt7)$
& $\eta\sigma<1$ (Thm.~\ref{thm:convergence_region_rank_1_approx}) \\
\bottomrule
\end{tabular}
\end{table}

\subsection{Edge of stability as a bifurcation phenomenon}
\label{subsec:rw_eos}

The edge-of-stability phenomenon of \citet{cohen2021gradient} links the
stability threshold $2/\eta$ to oscillatory, non-monotone large-step dynamics.
\citet{ahn2023learning} established a sharp transition from a gradient-flow
regime to an edge-of-stability regime in simple nonconvex models. Beyond the
transition, the iterates bounce and select a minimizer whose sharpness is close
to $2/\eta$, but still converge to a global minimizer.

\citet{song2023trajectory} developed the bifurcation viewpoint in more complex
models. After reparameterization, their gradient-descent trajectories align
with the bifurcation diagram of a frozen one-dimensional map whose
global-minimum branch loses stability through a period-doubling bifurcation.
The full trajectory nevertheless converges to a global minimizer: the
period-two oscillation describes the approach to the bifurcation point, not the
final $\omega$-limit set. In a minimalist model,
\citet{zhu2022understanding} likewise prove convergence to a global minimizer
with sharpness near $2/\eta$, while their global bifurcation picture is
observational. In these analyses, bifurcation organizes the large-step
trajectory and the selection of a near-threshold solution, while the
asymptotic limit remains a minimizer.

\paragraph{Periodic attractors beyond minimizer convergence.}
\citet{chen2023beyond} take a different step by analyzing the two-step map
$F_\eta^2$. For scalar factorization, they prove convergence to a balanced
period-two orbit when $1<\eta\sigma<\sqrt{4.5}-1$, under the assumption that the
two factors remain positive along the trajectory. For higher-dimensional
matrix factorization, they give theoretical conditions and numerical evidence
for period-two behavior, but no full-dimensional convergence classification.

Our result provides a full-dimensional convergence classification for rank-one
factorization, without a positivity assumption and on an explicit certified
region. In the normalized model,
Theorem~\ref{thm:convergence_region_rank_1}(2) first gives
\begin{align}
    \dist{x_t}{K_2}\to0,
    \qquad
    K_2=\{(a,a,0,0)\in\RR^4:a^2\le2\},
\end{align}
for almost every certified initialization when $1<\eta<2$. On $K_2$ the
dynamics reduce exactly to
\begin{align}
    s\longmapsto s(1+\eta-\eta s)^2.
\end{align}
Its fixed point loses stability at $\eta=1$ through a period-doubling
bifurcation. For $1<\eta<\sqrt5-1$, the reduced map has a unique attracting
period-two orbit, and internal chain transitivity transfers this classification
to the full dynamics: the full $\omega$-limit set is one of the two signed
balanced period-two orbits for almost every certified initialization
(Appendices~\ref{subsec:stable_2_period} and~\ref{sec:terminal_inheritance}).
Thus, unlike the minimizer-convergent dynamics above, the period-doubled orbit
itself is the attracting limit set. Its two sharpness values bracket $2/\eta$
with mean $(\eta+3)/(2\eta)$
(Corollary~\ref{cor:period_two_sharpness}), providing, in this model, an exact
threshold-straddling analogue of edge-of-stability hovering.

\subsection{Parameterized Lyapunov constructions}
\label{subsec:rw_lyapunov_construction}

Parameterized Lyapunov geometry has precedents in continuous-time control.
Implicit Lyapunov functions define the Lyapunov value through an equation
involving the state rather than through an explicit formula
\citep{adamy2005implicit}, while soft variable-structure control has used
continuously parameterized nested ellipsoids whose state-selected parameter
itself serves as an implicit Lyapunov variable
\citep{roethig2016stabilizing}. In the variable-structure construction, the
selected parameter is coupled to the feedback law and changes the closed-loop
vector field. Here the gradient map is fixed, and the state-selected parameter
is purely a certificate variable. Separately, transverse-stability theory relates
attraction to an invariant manifold to a field of quadratic forms contracting
directions transverse to that manifold \citep{andrieu2016transverse}. For the
fixed discrete gradient map, exact one-step inwardness makes stationary points
on a leaf tight contacts, and the resulting contact equations constrain, and
in the scalar setting determine, the geometry of the nested certificate family.

%% file: 99_llm_disclosure.tex
\section*{LLM usage disclosure}
We used ChatGPT-5.6 Sol and Opus 5 as auxiliary tools for writing, editing, and
technical discussion. In particular, they partially assisted with the
organization and exposition of Appendices~G and~H and with checking parts of
the arguments presented there. They also assisted in searching for algebraic
identities, including quotient--remainder decompositions. The models were not
part of the research methodology or experiments. All final theorem statements,
proofs, experiments, and claims were independently verified by the authors.

\newpage
\clearpage